\documentclass[
    english,
    11pt,
]{scrartcl}
\usepackage[margin=1in]{geometry}
\usepackage[round]{natbib}
\usepackage[utf8]{inputenc}
\usepackage[T1]{fontenc}
\usepackage{booktabs}
\usepackage[title]{appendix}

\usepackage{url,booktabs,amsfonts,dsfont,hyperref}
\usepackage{src/pkgs}
\usepackage{src/macros}

\hypersetup{%
    breaklinks=false,%
    colorlinks=true,%
    citecolor=black!50,%
    urlcolor=black!50,%
    linkcolor=black!50%
}
\renewcommand{\mathbb}{\mathds}
\usepackage[font=footnotesize]{caption}
\usepackage{sectsty}
\usepackage{titling}
\pretitle{\begin{center}\sffamily\bfseries\Large}
\posttitle{\par\end{center}}

\usepackage{authblk}

\title{Adversarially Robust Topological Inference}
\author[1]{\textrm{Siddharth Vishwanath}\thanks{{svishwanath@ucsd.edu}}}
\author[1]{\textrm{Bharath K. Sriperumbudur}}
\author[3]{\textrm{Kenji Fukumizu}}
\author[3]{\textrm{Satoshi Kuriki}}
\affil[1]{\small \textrm{Department of Mathematics, University of California San Diego}}
\affil[1]{\small \textrm{Department of Statistics, The Pennsylvania State University}}
\affil[3]{\small \textrm{The Institute of Statistical Mathematics}}
\date{}

\pgfplotsset{compat=1.17}
\begin{document}
\maketitle

\vspace*{-2em}

\begingroup
\let\clearpage\relax


\begin{abstract}
    \begingroup
    \centering
    The distance function to a compact set plays a crucial role in the paradigm of topological data analysis. In particular, the sublevel sets of the distance function are used in the computation of persistent homology---a backbone of the topological data analysis pipeline. Despite its stability to perturbations in the Hausdorff distance, persistent homology is highly sensitive to outliers. In this work, we develop a framework of statistical inference for persistent homology in the presence of outliers. Drawing inspiration from recent developments in robust statistics, we propose a \textit{median-of-means} variant of the distance function (\md{}) and establish its statistical properties. In particular, we show that, even in the presence of outliers, the sublevel filtrations and weighted filtrations induced by \md{} are both consistent estimators of the true underlying population counterpart and exhibit near minimax-optimal performance in adversarial settings. Finally, we demonstrate the advantages of the proposed methodology through simulations and applications. 
    \endgroup
\end{abstract}



\section{Introduction}
\label{sec:intro}

Given a compact set $\bX \subset \R^d$, its persistence diagram encodes the subtle geometric and topological features that underlie $\bX$ as a multiscale summary and forms the cornerstone of topological data analysis. Persistent homology serves as the backbone for computing persistence diagrams and encodes the homological features underlying $\bX$ at different resolutions. The computation of persistent homology is typically achieved by constructing a \textit{filtration} $V_{\bX}$, i.e., a nested sequence of topological spaces, which captures the evolution of geometric and topological features as the resolution varies. The persistent homology, which is encoded in its persistence module, $\bbv_{\bX}$, extracts the homological information from the filtration $V_{\bX}$. This is then summarized in a persistence diagram $\dgm\pa{\bbv_{\bX}}$.

Broadly speaking, there are two different methods for obtaining filtrations. The first, and, arguably more classical method is obtained by examining the union of balls of radius $r$ centered on the points of $\bX$ called the $r$--\textit{offset} of $\bX$, denoted ${\bX}(r)$, for each resolution $r > 0$. The resulting filtration $V[\bX] = \pb{\Xb(r): r > 0}$, depends only on the metric properties of $\bX$. The second, and more general approach is based on constructing a \textit{filter function} $f_{\Xb}$, which reflects the topological features underlying $\bX$. The resulting filtration $V[f_{\bX}]$, in this case, is obtained by probing the sublevel sets $f\inv_{\Xb}\left( (-\infty, r] \right)$ or the superlevel sets $f\inv_{\Xb}\qty( [r, \infty))$ associated with $f_\Xb$. While these two methods are vastly different, in principle, they both attempt to explore the topological features underlying $\bX$.

In this context, the distance function $\dx$ to the set $\Xb$ plays a special role in topological data analysis, and satisfies the property that $V[\bX] = V[\dx]$. That is, the sublevel sets of the distance function encode the same topological information as the filtration from its offsets. The appeal of using the distance function in the computation of persistence diagrams comes from the celebrated stability of persistence diagrams \citep{chazal2016structure}. In a nutshell, the stability result for persistence diagrams guarantees that (i) the persistence diagrams resulting from two compact sets $\Xb$ and $\Yb$ are close whenever the sets themselves are close in the Hausdorff distance, and, (ii)
the functional persistence diagrams resulting from two filter functions $f$ and $g$ are close whenever $f$ and $g$ are close w.r.t. the $\norminf{\cdot}$ metric.

In the statistical setting, one has access to $\bX$ only through samples $\Xn\! =\!\pb{\Xv_1, \dots, \Xv_n}$ obtained using a probability distribution $\pr$ which is supported on the (unknown) set $\bX$. The objective, in a statistical inference framework, is to use the samples $\Xn$ to infer the true population persistence diagram $\dgm\pa{\bbv_\Xb}$. The offset $\Xn(r)$ and filter function $f_n$, constructed using the sample points, are themselves random quantities associated with their population counterparts $\Xb(r)$ and $f_\Xb$, respectively, and these may be used to construct a sample estimator $\dgm\pa{\bbv_{\Xn}}$. To this end, several existing works have studied the statistical properties of persistence diagrams from samples in a noiseless setting, e.g.,~constructing confidence bands and characterizing the convergence rate of $\dgm\pa{\bbv_{\Xn}}$ to $\dgm\pa{\bbv_{\Xb}}$ in the space of persistence diagrams \citep{fasy2014confidence,chazal2015subsampling,chazal2015convergence,chazal2017robust}. 

However, in practical settings, real-world data is likely subject to measurement errors and the presence of outliers. While some assumptions may be imposed on the noise and the outliers, in the most baneful settings, the given data may be subject to adversarial contamination. In this setting, for $m<n/2$, we assume that the samples $\Xn$, which we have access to, contain only ${n-m}$ points obtained from the probability distribution $\pr$ with $\supp\pa{\pr} = \bX$, and make no further assumptions on the remaining $m$ points. In principle, the $m$ outliers may be carefully chosen by an adversary after examining the remaining $n-m$ points. While the stability of persistence diagrams guarantees that small perturbations in the sample points induce only small changes in the resulting persistence diagrams, even a few outliers in the samples can lead to deleterious effects. This issue is further exacerbated in the adversarial setting, where the adversary is free to place the $m$ points where it may drastically impact the resulting topological inference.  The overarching objective of this paper is to construct an estimator of the (unknown) population quantity $\dgm\pa{\bbv[\bX]}$ using the corrupted sample points $\Xn$ which is, both, statistically consistent and computationally efficient.

\subsection{Contributions}
To this end, our contributions are as follows. First, we introduce \md{}, denoted by $\dnq$, as an outlier-robust variant of the empirical distance function which is constructed using the median-of-means principle, and we establish its theoretical properties. Notably the \md{} relies on a tuning parameter $Q$ which is easy to interpret, and, roughly speaking, reflects an estimate for the number of outliers in the observed sample. While the persistence diagram resulting from the sublevel filtration of $\dnq$ is a valid candidate for statistical inference, it can be expensive to compute in practice. To overcome this, we use the weighted filtrations introduced by \cite{buchet2016efficient} and \cite{anai2019dtm} to construct $\dnq$-weighted filtrations, $V[\Xn, \dnq]$, as computationally efficient estimators of $\dgm\pa{V[\bX]}$. Our main contributions are the following:

\begin{enumerate}[label=\textup{(\Roman*)}]
    \item We show that sublevel set persistence diagrams of $\dnq$ are consistent estimators of the sublevel set persistence diagram of the true population counterpart $\dx$ even in the presence of outliers; we establish its convergence rate (Theorem~\ref{theorem:momdist-sublevel}) and show that it is near minimax optimal up to a $\log{n}$ factor (Theorem~\ref{thm:minimax}).
    \item We establish a stability result for the the $\dnq$-weighted filtrations, $V[\Xn, \dnq]$, and we show that they are stable w.r.t.~adversarial contamination (Theorem~\ref{theorem:momdist-stability}).
    \item Furthermore, we show that the persistence diagram $\dgm\pa{\bbv[\Xn, \dnq]}$ is both a computationally efficient and statistically near-optimal estimator of $\dgm\pa{\bbv[\bX]}$ (Theorem~\ref{theorem:momdist-consistency}).
    \item Next, in a sensitivity analysis framework, we quantify the gain in robustness achieved when using the $\dnq$-weighted filtrations vis-\`{a}-vis its non-robust $\dsf_n$-weighted counterpart; (Theorem~\ref{theorem:momdist-influence}). 
    \item Lastly, we propose a data-driven procedure for adaptively selecting the tuning parameter $Q$ using Lepski's method. For the data-driven choice $\widehat Q$, we show that the resulting estimator $\dgm\qty\big{\bbv[\Xn, \dsf_{n, \widehat Q}]}$ is also minimax optimal up to a $\log{n}$ factor (Theorem~\ref{theorem:lepski}).
    \end{enumerate}

\subsection{Related Work}

Several approaches have been proposed in existing literature to overcome the sensitivity of persistence diagrams to noise. The prevailing ideas in these approaches rely on constructing a filter function, $f_\pr$, which reflects both the topological information and the distribution of mass underlying the support $\supp\pa{\pr} = \bX$. Replacing the population probability measure $\pr$ with the empirical measure $\pr_n$ associated with the samples $\Xn$ results in an empirical estimator $f_{\pr_n}$. Some notable examples include the distance-to-measure \citep{chazal2011geometric}, the kernel distance \citep{phillips2015geometric}, and kernel density estimators \citep{fasy2014confidence,vishwanath2020robust}.

While these approaches mitigate, to some extent, the influence of noise on the resulting persistence diagrams, they are not without their drawbacks. For starters, while it may be argued that $\dgm\pa{\bbv[f_{\pr_n}]}$ is more resilient to noise, ultimately, this sample estimator corresponds to the population quantity $\dgm\pa{\bbv[f_{\pr}]}$, which may, nevertheless, omit some subtle geometric and topological features present in $\dgm\pa{\bbv[\bX]}$. Furthermore, from a statistical perspective, if $\Xn$ comprises only $n-m$ points from $\pr$ and the remaining $m$ points constitute outliers, then the sample estimator $V[f_{\pr_n}]$, obtained using $\Xn$, will no longer be a valid estimator of the population quantity $\dgm\pa{\bbv[f_{\pr}]}$ which we wish to infer. 

Lastly, the exact computation of these estimators can be prohibitively expensive, if not impossible in practice. For instance, the exact computation of the distance-to-measure requires computing an order-$k$ Voronoi diagram. Moreover, in the general setting, the sublevel/superlevel filtrations arising from these approaches are computed using cubical homology, which relies on a (nuisance) grid resolution \mbox{parameter}. If this resolution is too coarse, then some subtle topological features are affected. On the flipside, if the resolution is too fine, then the accuracy is still impacted, as noted in \cite{fasy2014confidence}. In the high-dimensional setting, cubical homology also falls victim to the curse of dimensionality, i.e., for a fixed grid resolution, the number of simplices in the resulting cubical complex grows exponentially with the dimension of the ambient space.

In order to overcome these computational drawbacks, \cite{buchet2016efficient} and \cite{anai2019dtm} propose weighted filtrations, $V[\Xn, f_{\Xn}]$, using power distances. While the weighted filtrations circumvent the need for constructing grid-based approximations, they come at the expense of exact inference, i.e., the weighted filtrations $V[\Xn, f_{\Xn}]$ only approximate $V[f_{\Xn}]$ and do not provide valid statistical inference, even in the absence of outliers.

{
In order to mitigate the effect of noise, \cite{buchet2015topological} propose a method for estimating $V[f_{\bX}]$ under a functional noise condition on $f_{\Xn}$---which encompasses several commonly encountered noise scenarios, e.g., Wasserstein noise and additive Gaussian noise. In a more general setting, \cite{buchet2018declutter} propose a decluttering algorithm that uses the distance-to-measure to select a subset of the sample points $\Xn$ which are representative of the underlying ground truth $\bX$ even when the sample contains outliers. Furthermore, the authors propose a parameter-free variant of the decluttering algorithm which can provably recover the ground truth under some weak-uniformity assumptions. A comprehensive study of these methods is outlined in the PhD thesis of \cite{buchet2014topological}. Our proposed methodology can be seen as an analogue of the parameter-free decluttering algorithm, which comes with certifiable statistical guarantees. In a similar vein, \cite{brecheteau2018k} propose a $K$-points approximation of the distance-to-measure, which produces a coreset of $K$ points that can closely approximate the distance-to-measure of the original sample, and noticeably reduces the computational time when $K \ll n$. The authors also propose a trimmed heuristic that is resilient to noise. This is further extended by \cite{brecheteau2020robust} where the coreset of $K$-points approximates the sublevel sets of the distance-to-measure using unions of ellipsoids as opposed to isotropic balls.
}

On a similar note, \cite{vishwanath2020robust} proposed robust persistence diagrams which are resilient to outliers using kernel density estimators (KDE), and also proposed a framework for characterizing the sensitivity to outliers using an analogue of influence functions. Although \citet[Theorem 1]{vishwanath2020robust} describes the gain in robustness by considering the robust KDE $\fn$ using the persistence influence function, Theorems~2~\&~3 of the same work establish that as $n \rightarrow \infty$ and $\sigma \rightarrow 0$, the persistence diagram $\dgm\pa{\fn}$ recovers the same information which underlies the sample points $\Xn$. However, if the underlying distribution is contaminated, e.g., ${\pr = (1-\pi) \pr^* + \pi \qr}$, then the topological inference we hope to target is that of $\pr^*$ and not that of $\pr$.

Finally, with a similar objective of mitigating the impact of noise in topological inference, recent approaches have considered multi-parameter persistent homology as a robust tool for inferring the topological features underlying $\Xn$ \citep{carlsson2009theory}. While some recent results have demonstrated some promise (e.g., \citealp{vipond2021multiparameter}), they are, nevertheless, computationally infeasible for most applications, in addition to being hard to interpret \citep{otter2017roadmap,bjerkevik2020computing}. 

On the statistical front, founded on the seminal works of \cite{tukey1960survey} and \cite{huber1964robust}, robust statistics has witnessed renewed interest \citep{diakonikolas2017being}. In particular, the classical problem of mean and covariance estimation has been revisited in several works \citep{audibert2011robust,minsker2015geometric,devroye2016sub,joly2016robust} with the objective of easing model assumptions to, either, the regularity of the data generating mechanism, or, the presence of outliers. See \cite{lugosi2019mean} for a recent survey. In this regard, median-of-means (MoM) estimators---originally introduced by \cite{nemirovskij1983problem}---and the broader median-of-means principle \citep{lecue2020robust} have emerged as powerful tools for "robustifying" existing estimators in near-linear time. Although this comes slightly at the expense of statistical optimality, median-of-means estimators are, nevertheless, easier to compute than statistically optimal \textit{and} robust methods such as the tournament estimators introduced by \citep{lugosi2019risk}. \cite{audibert2011robust} showed that, in the univariate setting, the MoM estimator achieves sub-Gaussian rates of convergence for heavy tailed data. \cite{minsker2015geometric} and \cite{devroye2016sub} extend these results to the multivariate setting by considering the geometric median. The MoM idea has subsequently been extended in several other directions, e.g., U-statistics \cite{joly2016robust}, kernel mean embeddings \cite{lerasle2019monk} and general M-estimators \cite{lecue2020robust} among others. Most importantly, these extensions move away from the heavy-tailed framework and provide significant insights on how MoM estimators can overcome the second relaxation, i.e., estimation in the presence of outlying contamination.

\subsection{Organization} 
The remainder of this paper is organized as follows. In Section~\ref{sec:preliminaries} we present the necessary background on persistent homology. We first introduce the proposed methodology in Section~\ref{sec:proposal}, and then present the main results in the remainder of the section. We establish the statistical properties of the proposed estimator in Section~\ref{sec:statistical}, and we present the influence analysis in Section~\ref{sec:influence}. Numerical results supporting the theory are provided in Appendix~\ref{sec:experiments}. The proofs of all the results are collected in the Appendix~\ref{sec:proofs} of the Supplementary Material.



\section{Preliminaries}
\label{sec:preliminaries}

The following subsections introduce the essential ingredients used for the remaining of the paper. 

\subsection{Definitions and Notations} 
For $n \in \Z_+$, we use the notation $[n] = \pb{1, 2, \dots, n}$, and for real-valued functions~$f$~and~$g$ we employ the notation $f(n) \lesssim g(n)$  if $f(n) = O\qty\big(g(n))$. The closed ball of radius $r$ centered at $\xv \in \R^d$ is denoted $B(\xv, r)$. For a compact set $\bX \subset \R^d$, the $r$--offset of $\bX$ is given by 
$$
\bX(r) = \bigcup_{\xv \in \bX}B(\xv, r).
$$ 
The distance function w.r.t. the compact set $\Xb$ plays a central role in extracting the geometric and topological features underlying $\Xb$. 

\begin{definition}[Distance function]
    For a compact set $\bX \subseteq \R^d$, the distance function to the set $\bX$, denoted as $\dx$, is given by
    \eq{
        \dx(\yv) \defeq \inf_{\xv \in \bX}\norm{\xv - \yv}, \ \ \ \text{for all } \yv \in \R^d.\nn
    }
\end{definition}

For a finite collection of points $\Xn$, the distance function $d_{\Xn}$ is simply denoted as $\dsf_n$. For two compact sets $\bX, \bY \subset \R^d$ the \textit{Hausdorff distance} {between $\bX$ and $\bY$} is given by
$$
{\haus{\bX, \bY}[] \defeq \inf\qty\Big{\e > 0: \bX \subseteq \bY[](\e), \bY \subseteq \bX[](\e)}} = \norminf{\dsf_{\bX} - \dsf_{\bY}},
$$ 
and metrizes the space of all compact subsets of $(\R^d, \norm{\cdot})$. 
While results here should extend to general metric spaces $(\M,\rho)$ with simple modifications along the lines of \cite{chazal2015convergence} and \cite{buchet2016efficient}, we assume that $(\M, \rho) = (\R^d, \norm{\cdot})$ throughout the paper.

$\mathcal{P}(\bX)$ denotes the set of Borel probability measures defined on $\R^d$ with support $\bX \subseteq \R^d$, and for $\xv \in \R^d$, $\dir{\xv}$ is used to denote a Dirac measure at $\xv$. A key assumption used throughout the paper is a regularity condition for the data-generating mechanism. For $a,b>0$, the probability measure satisfies the $(a,b)-$standard condition if
{
\eq{\label{eq:ab-standard}
    \pr\qty\Big( B(\xv, r) ) \ge 1 \wedge a r^b \quad \text{for all } r\ge0, \;\text{and for all } \xv \in \supp(\bX).
}}
We denote by $\mathcal{P}(\Xb, a, b)$ the subset of $\mathcal{P}(\Xb)$ which satisfies the $(a, b)-$standard condition in \eref{eq:ab-standard} for $a, b >0$. This regularity assumption is standard in the domain of geometric and topological inference (e.g., \citealp{cuevas2004boundary,chazal2015convergence,chazal2015subsampling,chazal2017robust}). For instance, if $b$ is an integer and $\bX$ is a $b$-dimensional submanifold with positive reach, then $\mathcal{P}(\Xb, a, b)$ consists of probability measures supported on $\Xb$ which admit a density with respect to the $b$-dimensional Hausdorff measure which is lower bounded by $O(a)$. Throughout the paper, we assume that the samples $\Xn$ are obtained in an adversarial contamination setting \ref{setting}, as defined below.

\begin{enumerate}[label=\textbf{Sampling Setting ($\scr{S}$).}, ref=\textup{($\scr{S}$)}, leftmargin=*, itemindent=11.5em, labelindent=2em]
    \item\label{setting} The data comprises of $n$ samples $\Xn = \{\Xv_1, \dots, \Xv_n\}$, where for a collection $\mathcal{O} \subset [n]$ with $\abs{\mathcal{O}} = m < {n}/{2}$, the samples $\Ym = \{\Xv_i: i \in \mathcal{O}\}$ are contaminated with arbitrary outliers. No distributional assumption is made on these outliers. The remaining $n-m$ samples, {$\Xnm=\{\Xv_i: i \in \mathcal{O}^c\}$,} are assumed to be observed i.i.d. from a distribution $\pr \in \mathcal{P}(\Xb, a, b)$, for compact $\Xb \subset \R^d$ and $a,b > 0$.
\end{enumerate}

A glossary of notations for additional definitions and notations introduced in the subsequent sections is provided in Table~\ref{tab:glossary} of the Supplementary Material.

\subsection{Background on Persistent Homology}
\label{sec:persistent}

In this section, we provide the necessary background on persistent homology. A detailed background is provided in Appendix~\ref{sec:persistent-homology-appendix}, and we refer the reader to \cite{chazal2017introduction,edelsbrunner2010computational} for a comprehensive introduction.

Given a compact set $\bX$, the building block of any topological data analysis pipeline to extract meaningful information from $\bX$ begins with a nested sequence of filtered topological spaces called a filtration, simply denoted by $V$. The sequence of spaces {is} parametrized by a resolution parameter $t$. There are several approaches for constructing filtrations using $\Xb$. One approach is to consider the collection of offsets built on top of $\Xb$, i.e., $V^t = V^t[\bX] = \Xb(t)$. For $s < t$, the offsets are nested $V^s \subseteq V^t$, and $V[\bX] \defeq \pb{ V^t[\Xb] : t \in \R }$ is a nested sequence of topological spaces and defines the filtration built using the offsets of~$\Xb$. 

The second approach to constructing a filtration is using a filter function $f_{\Xb}: \R^d \rightarrow \R$ which carries the topological information underlying $\Xb$. In this scenario, one typically constructs the filtration from the sublevel sets associated with $f_\Xb$, given by $V^t = f\inv_{\Xb}\qty\big({ (-\infty, t] })$ for each resolution $t$. Again, for $s<t$, {$V^s[f_\Xb] \subseteq V^t[f_\Xb]$} and the sequence $V[f_\Xb] = \pb{ V^t[f_{\Xb}] : t \in \R }$ constitutes the sublevel filtration from $f_{\Xb}$. Mutatis mutandis a similar notion holds for the superlevel filtration.

In general, the filtration $V[\Xb]$ can be very different from $V[f_{\Xb}]$, although the prevailing objective is for $V[f_{\Xb}]$ to encode the same information as in $V[\Xb]$. In this context, the distance function $\dx$ plays a special role owing to the fact that its sublevel filtration is the same filtration associated with the offsets, i.e., $V[\dx] = V[\bX]$. This fact plays an important role in motivating the \md{} estimator introduced in Section~\ref{sec:proposal}, and follows by noting that for every resolution $t > 0$,
$\dsf_{\Xb}\inv\qty\Big({ (-\infty, t] }) = \pb{\xv \in \R^d : \dx(\xv) \le t} = \bigcup_{\xv \in \Xb}B(\xv, t).$

For general $f$, one can hope to get the best of both worlds by constructing $f$-weighted offsets as follows. For a \textit{power parameter} $p \ge 1$, the $f$-weighted radius at $\xv$ is given by
\begin{equation}\label{eq:weighted-radius}
    \rfx \defeq \begin{cases}
        \pa{t^p - f(\xv)^p}^{1/p} & \text{ if } t \ge f(\xv) \\
        -\infty                   & \text{ if } t < f(\xv),
    \end{cases}
\end{equation}
and $\Bfx[f][][][] \defeq \qty{\yv \in \R^d: \norm{\xv - \yv} \le \rfx}$ is the \emph{weighted ball of resolution $t$ at $\xv$}. For a single resolution $t$, each $\Bfx[f][][][]$ is a Euclidean ball with a different radius determined by \cref{eq:weighted-radius}. The $f$-weighted offset at resolution $t$ is then given by,
$$
    \Vt[][\bX,f][] \defeq \bigcup_{\xv \in \bX} \Bfx[f][\xv][].\nonumber\vspace*{-0.6em}
$$
See Figure~\ref{fig:ball} for an illustration. On the one hand, if $f(\xv) = 0$ for all $\xv \in \Xb$, as is the case with $\dx$, then the $f$-weighted offsets reduce to the usual offsets $V^t[\Xb]$. On the other hand, if $\Xb=\R^d$, then the $f$-weighted offsets coincide with the sublevel sets $V^t[f]$.

Let $V = \pb{V^t : t \in \R}$ denote a generic filtration. The set $V^t$ encodes meaningful topological features at each resolution $t$, e.g., connected components, loops, holes, etc. Formally, this information is ``stored'' in the homology groups $\bbv^t = \text{H}_k(V^t)$, which are vector spaces. The basic idea of persistent homology is that as the resolution parameter $t$ varies, the topological features evolve, and this evolution can be captured by examining the linear maps between vector spaces in $\bbv^t$. Roughly speaking, the collection $\bbv = \qty{\bbv^t: t \in \R}$ constitutes a persistence module. See Figure~\ref{fig:interleaving}. Imagine observing $\bbv^t$ as $t$ varies continuously; one of three things can happen: (1) the topology of $\bbv^t$ is identical to that at $t^+$ and $t^-$, (2) at $t\equiv t_b$, $\bbv^t$ contains an additional topological feature not present at $t^-$, or, (3) at $t\equiv t_d$, $V^t$ contains a feature which disappears at $t^+$. In the last two cases, $t_b$ and $t_d$ are known as the birth/death times of a topological feature. A \textit{persistence diagram}, given by 
$$\dgm(\bbv) \defeq \qty{ (t_b, t_d) \in \R^2: \exists \text{ a feature born at $t_b$ and dies at $t_d$} },
$$
is obtained by pairing the birth/death times of all topological features in $\bbv$. Since $t_b < t_d$, persistence diagrams are multisets supported on the upper half-plane $\Omega = \qty{(x, y): x \le y}$.

\subsection{Interleaving and Bottleneck distance} Given persistence modules $\bbv = \{\bbv^t: t \in \R\}$ and $\bbw=\{\bbw^t: t \in \R\}$, there is a natural way to measure the similarity between them by considering two nondecrasing maps, $\alpha(t)$ and $\beta(t)$. In particular, if $\alpha$ and $\beta$ form well-defined criss-crossing maps between $\bbv$ and $\bbw$, as shown in Figure~\ref{fig:interleaving}, for all $s \le t$, then $\bbv$ and $\bbw$ are said to be $(\alpha, \beta)$-interleaved. 

\begin{figure}[t]
    \centering
    \includegraphics[width=0.8\textwidth]{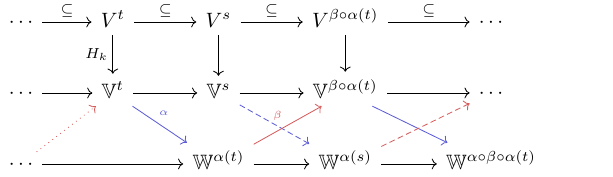}
    \vspace{10pt}
    \caption{Informal illustration of interleaving of persistence modules. Here $s\le t$ and $t \le (\beta \circ \alpha)(s)$ for two nondecreasing maps $\alpha, \beta$. See \cref{sec:interleaving} for a formal definition.}
    \label{fig:interleaving}
    \vspace{10pt}
\end{figure}

When the interleaving maps $\alpha$ and $\beta$ are \textit{additive} and \textit{identicial}, i.e., $\alpha(t) = \beta(t) = t + \epsilon$, such maps induce a pseudometric on persistence modules and, from the isometry theorem \citep{chazal2016structure}, a metric on persistence diagrams, referred to as the \emph{bottleneck distance},
$$
\Winf\qty\big( \dgm(\bbv), \dgm(\bbw) ) \defeq \inf\qty\Big{ \e > 0 : \bV \text{ and } \bW \text{ are } (\alpha,\alpha)\text{--interleaved for } \alpha: t \mapsto t + \e }.
$$
The bottleneck distance provides an avenue to study perturbations of persistence diagrams based on the input data, and is summarized in the following stability result. The stability of sublevel filtrations and unweighted offsets is due to \cite{cohen2007stability} and \cite{chazal2016structure}, whereas the stability of the $f$-weighted filtrations is due to \cite{buchet2016efficient} and \cite{anai2019dtm}.

\begin{proposition}[Stability of persistence diagrams]
\label{lemma:anai-et-al}
    For two compact sets ${\bX, \bY \subset \R^d}$,
    $$
        \winf\qty\big({ \dgm\pa{ \bbv[\Xb] }, \dgm\pa{ \bbv[\Yb]} }) \le \haus{\Xb, \Yb}[].\nn
    $$
    Furthermore, given two filter functions $f,g : \R^d \rightarrow \R$,
    $$
        \winf\qty\Big({ \dgm\pa{ \bbv[\Xb, f] }, \dgm\pa{ \bbv[\Xb, g]} }) \le \norminf{f-g}.\nn
    $$
    Additionally, given $h: \bX \cup \bY \rightarrow \R_+$, if $h$ is $L$--Lipschitz and $\haus{\bX,\bY} \le \e$, then
    $$
    \winf\qty\Big( \bbv[\bX,h], \bbv[\bY, h] ) \le \epsilon\pa{1+L^p}^{1/p}.
    $$
\end{proposition}

The stability of persistence diagrams guarantees that small changes in the function $f$ or the underlying data $\Xb$ result in small changes in the persistence diagrams, and is a key ingredient in establishing statistical guarantees. However, as noted in Section~\ref{sec:intro}, stability is insufficient to guarantee robustness, and this will form the starting point for our analysis in Section~\ref{sec:proposal}.

\begin{figure}
    \centering
    \begin{subfigure}[b]{0.32\textwidth}
        \includegraphics[width=\textwidth]{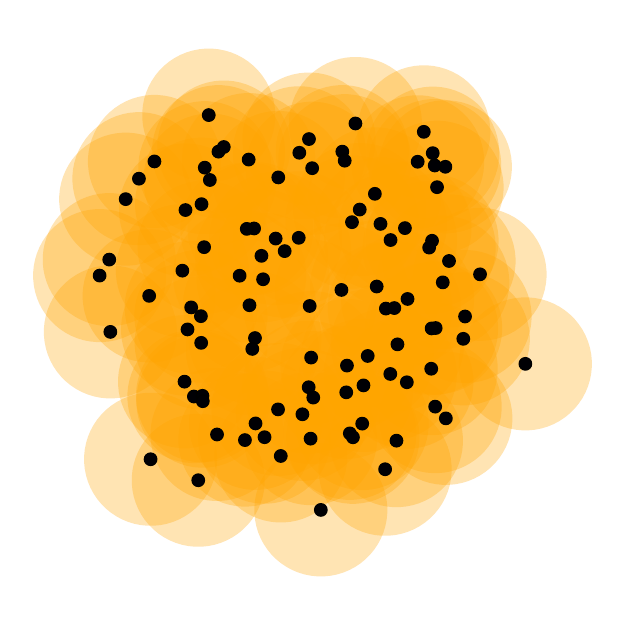}
        \caption{$V^t[\Xn]$ unweighted}
    \end{subfigure}
    \begin{subfigure}[b]{0.32\textwidth}
        \includegraphics[width=\textwidth]{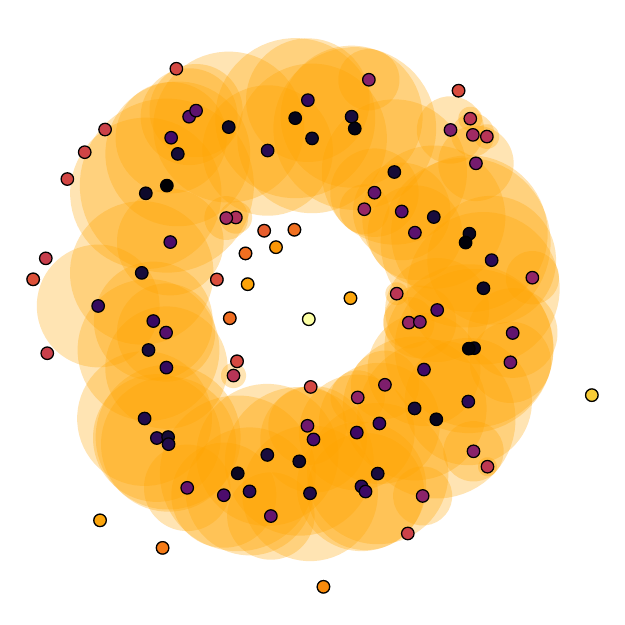}
        \caption{$V^t[\Xn, f]$ for $p=1$}
    \end{subfigure}
    \begin{subfigure}[b]{0.32\textwidth}
        \includegraphics[width=\textwidth]{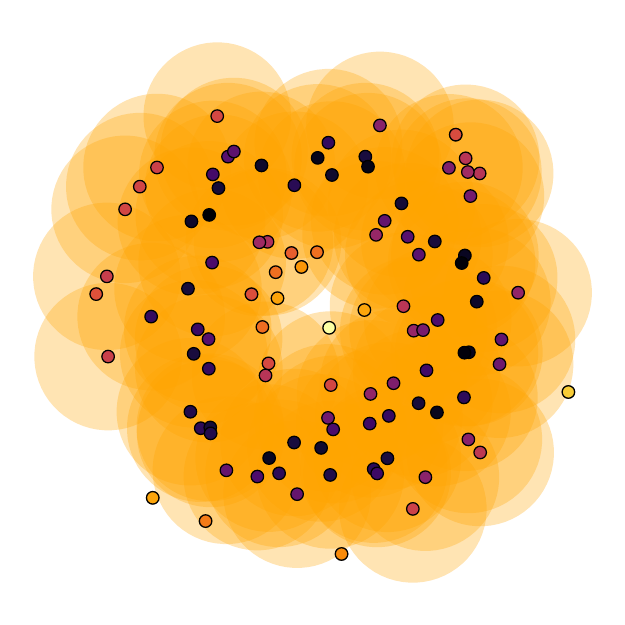}
        \caption{$V^t[\Xn, f]$ for $p=\infty$}
    \end{subfigure}
    \caption{Illustration of offsets for $t=0.5$ and $f(\xv) = \inf_{\yv \in \mathbb{S}^1}\norm{\xv-\yv}$.}
    \label{fig:ball}
\end{figure}



\section{Estimation thresholds under adversarial contamination}
\label{sec:minimax}

Let $\Xn = \pb{\Xv_1, \Xv_2, \dots, \Xv_n} \subset \R^d$ be a sample of $n$ observations. We assume that the samples are obtained under-sampling setting \ref{setting}. We emphasize that this setting encompasses the following~scenarios:
\vspace*{-0.25em}
\begin{enumerate}[label=(\alph*), itemindent=*, itemsep=-1pt]
    \item The samples $\Xn$ are obtained i.i.d.~from $\pr \in \mathcal{P}(\Xb, a, b)$ for compact $\Xb \subset \R^d$.
    \item Samples $\{\Xv_1^*, \dots, \Xv^*_n\}$ are first obtained i.i.d.~from $\pr \in \mathcal{P}(\Xb, a, b)$, and for a collection $\mathcal{O} \subset [n]$ of size $\abs{\mathcal{O}} = m < n/2$, an adversary is free to replace $\{\Xv^*_i: i \in \mathcal{O}\}$ with any points $\{\Yv_i: i \in \mathcal{O}\}$ of their choice. Then, $\Xn = \{\Xv^*_i: i \in \mathcal{O}^c\} \cup \{\Yv_i: i \in \mathcal{O}\}$ is shuffled and handed over to the topologist for inference, who has no prior knowledge of the original samples or the collection $\mathcal{O}$.
\end{enumerate}

The central objective is to derive a statistically optimal and computationally efficient estimator of $\dgm\pa{\bbv[\Xb]}$ which is robust to the misspecification scenarios detailed above, using the samples $\Xn$. It is, therefore, instructive to first consider the estimation thresholds for persistence diagrams under adversarial contamination. To this end, we adopt the minimax framework. Specifically, given $\Xn = \Xnm \cup \Ym$ obtained under sampling setting \ref{setting}, i.e.,  $\Xnm$ is observed i.i.d. from $\pr \in \mathcal{P}\equiv\mathcal{P}(\Xb, a, b)$ and $\Ym$ are the outliers, the minimax risk is given by
\eq{
    \mathfrak{R}_{n,m}(\mathcal{P}) = \inf_{\widehat\D_n} \sup_{\pr \in \mathcal{P}} \sup_{\Ym} \; \E_{_{\pr}}\qty\bigg[ \Winf\qty\Big( \widehat\D_n, \dgm(\bbv[\Xb]) ) ],
    \label{eq:minimax-risk}
}
where the infimum is taken over all persistence diagram estimators $\widehat\D_n$ based on the sample $\Xn$, and the expectation is taken over the randomness in $\Xnm \sim \pr$. The minimax risk $\mathfrak{R}_n(\mathcal{P})$ quantifies the best achievable performance of any estimator in estimating the target population quantity $\dgm(\bbv[\Xb])$ under the sampling setting \ref{setting}. The following result establishes a minimax lower bound for the estimation of persistence diagrams under adversarial contamination.

\begin{theorem}[Minimax lower bound]\label{thm:minimax}
    For $a >0$ and $b \in [d]$, suppose $\Xn = \Xnm \cup \Ym$ is obtained under sampling setting \ref{setting}. Then,
    \eq{
        \mathfrak{R}_{n,m}(\mathcal{P}) \gtrsim \qty(\frac{m}{2n-m})^{1/b} \vee \qty(\frac{\log{n}}{n})^{1/b}.\label{eq:minimax-n}
    }
    Moreover, if $m \equiv m_n = cn^\e$ for $0 < \e < 1$ and $c > 0$, then
    \eq{
        \mathfrak{R}_{n,m}(\mathcal{P}) \gtrsim \qty(\frac{1}{n^{1-\epsilon}})^{1/b}.
        \label{eq:minimax-rate}
    }
\end{theorem}

In other words, \cref{thm:minimax} establishes that the minimax rate of convergence roughly depends on the fraction of points from the true signal to the number of outliers, and, therefore, in the regime where $n, m \to \infty$ and $m=o(n)$, the rate of convergence is attenuated by a factor $\asymp(m/n)^{1/b}$. In the absence of outliers, i.e., when $m=0$, and when $b$ is an integer, Theorem~5 of \cite{chazal2015convergence} shows that the minimax rate has a necessary $\log{n}$ factor, and we believe that the same should hold true in the presence of outliers in \cref{eq:minimax-rate}. In what follows, we present an estimator to obtain outlier robust persistence diagrams whose statistical performance is guaranteed to be minimax-optimal up to a logarithmic factor.

\section{Empirical distance function using the Median-of-Means principle}
\label{sec:proposal}

In the following section, we present an estimator to obtain outlier robust
persistence diagrams. Its statistical properties along with the influence analysis are presented in Sections~\ref{sec:statistical}--\ref{sec:influence}. In \cref{sec:lepski} we present a method for adaptively calibrating the tuning parameter using a data-driven procedure. To this end, the MoM Distance (\md{}) function $\dnq$ is defined as follows. 

\begin{definition}[\md{}]
    Given a collection of points $\Xn \subset \R^d$ and $1 \le Q \le n$, let  $\pB{S_1, S_2, \dots S_Q}$ be a partition of $\Xn$ into $Q$ disjoint blocks, such that each subset $S_q \subset \Xn$ comprises of $|{S_q}| = \floor{n/Q}$ samples\footnote{ \ Without loss of generality, we may assume that $n$ is divisible by $Q$, so that $n/Q \in \Z_+$}. The MoM distance function $\dnq: \R^d \rightarrow \R_{\ge 0}$ is defined to be
    \eq{
        \dnq(\yv) \defeq \med\qty\Big{ \dsf_{n, S_q}(\yv) : q \in [Q] } = \med\qty\Big{ \inf_{\xv \in S_q} \norm{\xv -\yv} : q \in [Q]}.
        \label{eq:momdist}
    }
    The proposed outlier robust persistence diagram $\dgm\pa{ \bbv[{\Xn, \dnq}] }$ is then obtained using $\dnq$-weighted filtration $V[\Xn, \dnq]$. 
    \label{def:mom}
\end{definition}

\begin{remark}
    {Without loss of generality, and for convenience of the theoretical analysis we assume that $Q$ is odd; in this case, the median in \cref{eq:momdist} is realized by a unique $q \in [Q]$. In practice, if $Q$ is even, then, as per convention, the median can be taken as the average of the two middle values.} Note that we recover the usual empirical distance function, i.e., $\dsf_{n, 1} \equiv \dsf_n$ when $Q = 1$.
\end{remark}
    
    \begin{remark} For each block $S_q$, distance function $d_{n,q} \in \F({\R^d})$ can be viewed as the Kuratowski embedding of $S_q$ \cite[pg.~4]{lim2020vietoris}. The most natural generalization of the multivariate median-of-means estimators proposed by \cite{minsker2015geometric} and \cite{lerasle2019monk} would suggest the following estimator as the natural candidate for \md{}: 
    \vspace*{-1em}
    $$
    \widetilde{\dsf}_{n,Q} = \arginf_{f \in \F(\R^d)} \sum_{q=1}^{Q} \norminf{f - \dsf_{n,S_q}},\nonumber
    \vspace*{-0.5em}
    $$
    where the median under consideration corresponds to the geometric median in $L_{\infty}(\R^d)$. Although $\widetilde{\dsf}_{n,Q}$ has its appeal from a theoretical perspective, the computation of $\widetilde{\dsf}_{n,Q}$ involves an infinite-dimensional optimization problem, making it infeasible in practice. In contrast, the proposed estimator in Definition~\ref{def:mom}, is a pointwise median-of-means estimator with a tractable computational cost. This has the promise of being highly modular, and widely applicable in many practical settings. The technical difficulty arises in showing that the pointwise estimator $\dnq$ achieves an exponential concentration bound around $\dx$ in the $L_{\infty}(\R^d)$ metric.   
\end{remark}

Similar to the proposed methodology in Definition~\ref{def:mom}, the procedure of partitioning the data $\Xn$ into smaller subsets, and then aggregating them as an estimator of persistent homology has been shown to satisfy several favorable properties by \cite{solomon2021geometry} and \cite{gomez2021curvature}, albeit in a different context. We argue that a similar principle, in our setting, also leads to provably robust estimators.

\begin{table}
    \caption{Comparison of computational complexity for robust weighted filtrations.}
    \vspace*{-0.75em}
    \centering
    \resizebox{\textwidth}{!}{\begin{tabular}{llll}
    \toprule
    Method & Pre-processing & Evaluation & Provably robust? \\ 
    \midrule
        $V[\Xn, \dnq]$ (\md{}--filtration) & $O\qty\big( n/Q \cdot \log(n/Q))$  & $O\qty\big( n \cdot (Q + \log n/Q) )$ & Yes \\ 
        $V[\Xn, \delta_{n, k}]$ \cite[DTM--filtration]{anai2019dtm} & $O( n \log n )$  & $O( kn \log n )$ & Sometimes \\ 
        $V[\Xn, \fns]$ \cite[RKDE--filtration]{vishwanath2020robust} & $O(n^2 \ell)$  & $O(n^2)$ & Sometimes \\ 
        $V[\Xn, \mathbf{C}_K]$ \cite[\kpdtm{}]{brecheteau2018k} & $O(n K L)$  & $O(kK\log K)$ & Sometimes \\ 
        \bottomrule
    \end{tabular}}
    \medskip
    {\scriptsize 
    $n=\#$samples, $Q=\#$blocks, $k=\!\floor{mn}\!=\!$ DTM parameter,\\[-1.25em] $\ell=\#$iterations of KIRWLS algorithm, $K=\#$centroids, $L=\#$iterations for \kpdtm{}.
    }
    \label{tab:comparison}
\end{table}

\begin{figure}[t]
    \centering
    \begin{subfigure}[b]{0.22\linewidth}
        \includegraphics[width=1.2\linewidth]{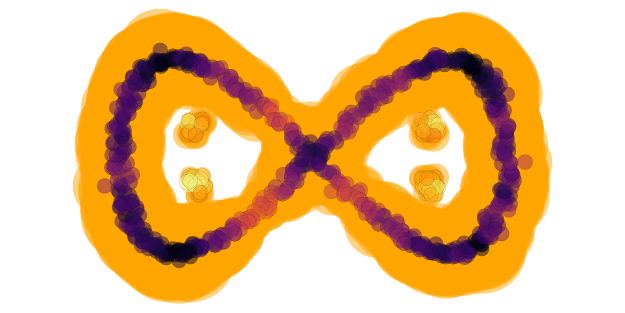}
        \caption{$V^t[\Xn, \dnq]$}
    \end{subfigure}
    \begin{subfigure}[b]{0.22\linewidth}
        \includegraphics[width=1.2\linewidth]{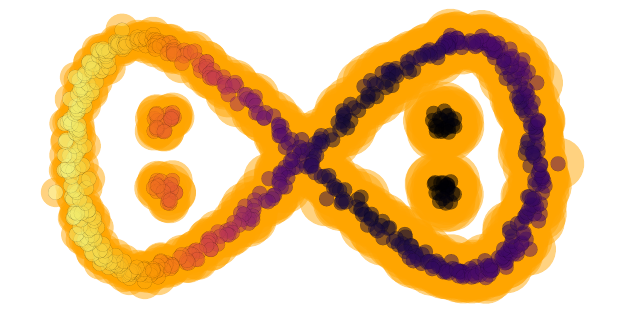}
        \caption{$V^t[\Xn, \fns]$}
    \end{subfigure}
    \begin{subfigure}[b]{0.22\linewidth}
        \includegraphics[width=1.2\linewidth]{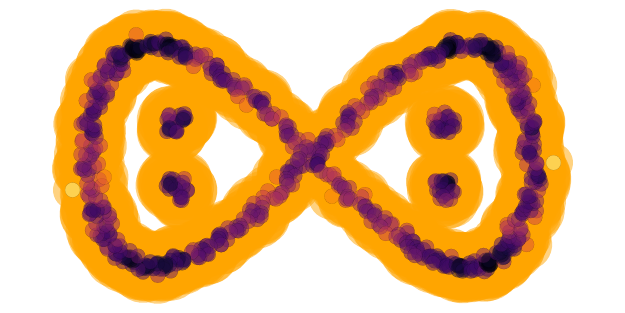}
        \caption{$V^t[\Xn, \delta_{n,k}]$}
    \end{subfigure}
    \begin{subfigure}[b]{0.22\linewidth}
        \includegraphics[width=1.2\linewidth]{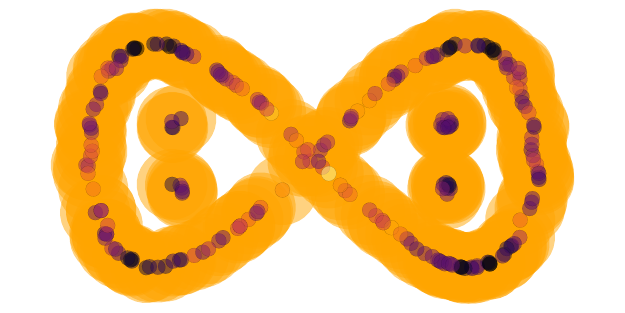}
        \caption{$V^t[\mathbf{C}_N]$}
    \end{subfigure}
    \includegraphics[height=0.17\linewidth]{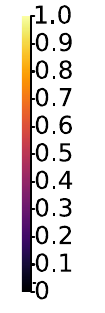}
    \caption{$\Xn$ with $n=620$ points from a Lemniscate with $m=80$ outliers. Illustration of the robust weighted filtrations with $p=1$ for \md{} $V^t[\Xn, \dnq]$, RKDE $V^t[\Xn, \fns]$, DTM $V^t[\Xn, \delta_{n,k}]$, and the k-PDTM $V^t[\mathbf{C}_N]$ filtration.}
    \label{fig:robust-examples}
\end{figure}

\subsection{Computational considerations} The first step in constructing the $f$-weighted filtration involves estimating the weights for the sample points, $w_i = f(\Xv_i)$ for all $i \in [n]$. Once this is done, the complexity of constructing the filtration $V[\Xn, f]$ does not depend on the specific choice of $f$. Table~\ref{tab:comparison} compares the complexities of four filtrations: (i) \md{} $\dnq$, (ii) robust kernel density estimator $\fns$ (RKDE, \citealp{vishwanath2020robust}), (iii) distance-to-measure $\delta_{n,k}$ (DTM, \citealp{anai2019dtm}), and (iv) $K$-Power Distance-to-measure $\delta_{n, k, K}$ (\kpdtm{}, \citealp{brecheteau2018k}). For a test point $\xv \in \R^d$, distances to blocks $S_q$ are efficiently computed using a $k$-d tree with pre-processing time $O(\abs{S_q} \log \abs{S_q})$ per block where $\abs{S_q}=n/Q$. This operation can be parallelized, and the overall pre-processing can be achieved in $O(\abs{S_q} \log\abs{S_q})$ time. Thereafter, $O(\log \abs{S_q})$ time is needed for a single query \cite[Chapter~10]{cormen2009introduction}. The results for each block $q\in[Q]$ are then aggregated to compute the median, which takes an additional $O(Q)$ time per query. This results in a total evaluation time of $O(n \cdot (Q + \log{\abs{S_q}}))$ for $n$ samples. 

The DTM with parameter $m$ involves finding the $k$th nearest neighbor for $k=\floor{mn}$, also optimized with a $k$-d tree, leading to a total complexity of $O(n \log n)$ for pre-processing and $O(n \cdot k \log n)$ for evaluation. The RKDE requires $O(n^2)$ time per iteration of the KIRWLS algorithm, with a total of $O(n^2 \ell)$ for $\ell$ iterations. Thereafter, the RKDE weights may be used to evaluate each query in $O(n)$ time. The \kpdtm{} constructs a quantized dataset via Lloyd's algorithm, with each iteration requiring $K$ comparisons to each of the $N$ points, resulting in $O(nKL)$ time across $L$ iterations, and outputs $K$ points. Typically, $L = O(n)$ iterations suffice. The trimmed variant of the \kpdtm{} shares the same preprocessing cost and produces a subset of $\Xn$ containing $n\alpha < n$ points, where $\alpha \in (0, 1)$ is a user-specified tuning parameter for the fraction of inliers. If $K \ll n$ or $\alpha \ll 1$, the \kpdtm{} and its trimmed variant offer significant speed-ups compared to weighted filtrations. The four filtrations are illustrated in Figure~\ref{fig:robust-examples}. A comparison of the computational and memory trade-offs are summarized in Table~\ref{tab:summary} for the experiment in \cref{exp:pathological}.

We conclude this section with the following property of the \md{} function.
\vspace{-0.2em}
\begin{lemma}
    Given samples $\Xn = \pb{\Xv_1, \Xv_2, \dots, \Xv_n}$ and $Q < n$,
    \eq{
        \abs{ \dnq(\xv) - \dnq(\yv) } \le \norm{ \xv - \yv }, \qq{for all $\xv, \yv \in \R^d$.}\nn
    }
    \label{lemma:lipschitz}
\end{lemma}
\vspace*{-1.25em}
In other words, the \md{} function $\dnq$ is $1$-Lipschitz --- a key property for distance-like functions in geometric and topological inference \cite[Chapter~9]{boissonnat2018geometric}.

\section{Statistical properties of $\bbv[\dnq]$}
\label{sec:statistical}

We begin our analysis by characterizing the persistence diagrams obtained using the sublevel filtration of $\dnq$. The following result (proved in Section~\ref{proof:theorem:momdist-sublevel}),  establishes that $\dgm\pa{\bbv[\dnq]}$ is a statistically consistent estimator of target population quantity $\dgm\pa{\bbv[\Xb]}$ under sampling setting~\ref{setting}, and establishes its rate of convergence in the $\Winf$ metric. 

\begin{theorem}[Sublevel filtration]
    Suppose $\pr \in \mathcal{P}(\bX, a, b)$ is a probability distribution with support~$\bX$ satisfying the $(a, b)-$standard condition and $\Xn$ is obtained under the sampling condition \ref{setting}. For $2m < Q < n$ and for all $0 < \delta < e^{-(1+b)Q}$,
    \eq{
        \pr\qty\bigg{ \Winf\qty\bigg( \dgm\pa{\bbv[\dnq]}, \dgm\pa{\bbv[\Xb]} ) \le {2}\mathfrak{g}(n, Q, a, b) } \ge 1 - \delta,
        \label{eq:mom-confidence-band}
    } 
    where 
    \vspace*{-0.25em}
    \eq{
        \mathfrak{g}(n, m, Q, a, b) = \qty\bigg( \frac{Q\log(n / Q)}{a n} + \frac{4Q \log(1/\delta)}{a(Q-2m)n} )^{1/b}.\label{eq:dnq-rate}
    }
    Furthermore, if the number of outliers grows with n as $m = cn^\epsilon$ for $c > 0$ and $\epsilon \in [0, 1)$ then {for all $Q=Cn^\e$ where $C > 2c$,}
    \eq{
        \E\qty\bigg[ \Winf\qty\bigg( \dgm\pa{\bbv[\dnq]}, \dgm\pa{\bbv[\Xb]} ) ] \lesssim \pa{\f{\log n}{n^{1-\e}}}^{1/b}.
        \label{eq:rate-with-noise}
    }
    \label{theorem:momdist-sublevel}
\end{theorem}

\begin{remark}
    The following salient observations can be made from {Theorem~\ref{theorem:momdist-sublevel}}.
    \begin{enumerate}[label=\textup{{(\roman*)}}, itemindent=*]
        \item For the rate of convergence in \cref{eq:rate-with-noise}, the dependence of $Q$ with $n$ is implicit through the dependence on $m$ with $n$ and the constraint that $2m < Q$. Given $m=cn^\epsilon$, the optimal choice of $Q$ is obtained when $Q = 3cn^\epsilon$. On the other hand, if $m$ is fixed relative to $n$, then choosing any fixed $Q > 2m$ yields a convergence rate of $(\log{n}/n)^{1/b}$, which coincides with the minimax rate for persistence diagram estimation in the absence of outliers (\cref{thm:minimax}). In other words, in the regime where $m$ is fixed relative to $n$, there is no price to pay for accommodating outliers in the data. Similarly, it becomes apparent that accommodating for more adverse noise conditions comes at the price of an attenuated rate of convergence.
        
        \item The two terms appearing in $\mathfrak{g}(n, m, Q, a, b)$ may be interpreted as follows: The first term is similar to the term appearing in \citet[Theorem~2]{chazal2015convergence} with an effective sample size of $n/Q$ instead of $n$, which is a consequence of the Median-of-Means procedure. The second term incorporates the desired confidence level  adaptive to the volume dimension $b>0$, with an effective sample size of $n/Q$. Notably, as the number of outliers $m$ increases, the number of blocks $Q$ must also increase; thereby widening the resulting confidence band.
        
        \item The admissible confidence level $\delta$ for constructing the confidence band is implicitly dependent on the parameter $Q$. This phenomenon is unavoidable with estimators based on the median-of-means principle. We refer the reader to \citet[Section~2.4]{lugosi2019mean} for a discussion on how robustness must come at the price of the confidence level $\delta$ being restricted.
    \end{enumerate}
\end{remark}

The proof of Theorem~\ref{theorem:momdist-sublevel} relies on Lemma~\ref{lemma:mom} (in the supplementary material) which allows us to control the deviation of a empirical processes arising from general classes of pointwise median-of-means estimators. 

\section{Statistical properties of $V[\Xn, \dnq]$}
\label{sec:statistical-2}

In practice, the sublevel filtration $V[\dnq]$ cannot be computed exactly, and one must rely on approximations using cubical homology. To this end, we now turn our attention to $\dnq$-weighted filtrations computed on the sample points directly. Before we study the statistical properties of the $\dnq$--weighted filtration, we state the following result which guarantees that the persistence module $\bbv[{\Xn, \dnq}]$ is sufficiently regular and is amenable to the construction of a persistence diagram\footnote{$q$-tame and pointwise finite dimensional persistence modules result in persistence diagrams whose points in $\Omega$ are locally finite.}.

\begin{lemma}[Regularity] 
    For $\Xn$ obtained under sampling setting \ref{setting} and $\dnq$ defined in \cref{eq:momdist}, the persistence module $\bbv[{\Xn, \dnq}]$ is $q-$tame and pointwise finite-dimensional.
    \label{lemma:momdist-regularity}
\end{lemma}
The proof of Lemma~\ref{lemma:momdist-regularity} is a direct consequence of \citet[Proposition~{3.1}]{anai2019dtm}, and ensures that the persistence diagram $\dgm\pa{\bbv[\Xn, \dnq]}$ is well-defined. We now turn our attention to the $\dnq$-weighted filtration $V[\Xn, \dnq]$. The following result, which establishes an analogue of the stability result for $\dnq$-weighted filtrations, but unlike the stability for the usual distance function $\dn$, it is also robust to outliers. 

\begin{theorem}[Stability\,\&\,robustness of $\dnq$-weighted filtrations]
    Let $\Xn\! =\!\Xnm \cup \Ym$ be a collection of points obtained under the sampling condition \ref{setting}. For $Q > 2m$ let $\dnq$ be the \textup{\md{}} function computed on the contaminated points $\Xn$ and let $\dsf_{n-m}$ be the distance function w.r.t.~the inliers $\Xnm$. Then
    \eq{
        \Winf\qty\bigg( \dgm(\bVt[  ][\Xn, \dnq]), \dgm(\bVt[  ][\Xnm, \dsf_{n-m}]) ) 
        &\le \sup_{{\xv \in \Xnm}}\dnq(\xv) + \norminf{\dnq - \dsf_{n-m}}\nn\\
        &\qquad\quad + \qty(1 - \f1p)t(\Xnm), \nn
    }
    where 
    $$
    t(\Xnm) = \inf \qty\Big{t>0: {\textstyle \bigcap\limits_{\xv \in \Xnm} }  B_{f, \rho}(\xv, t) \neq \varnothing} 
    $$ 
    is the filtration value $t$ at which the inliers the $\dnq$-weighted balls of the inliers  $\Xnm$ form a connected component. In particular, when $p=1$ we have
    \eq{
        \Winf\qty\bigg( \dgm(\bVt[  ][\Xn, \dnq]), \dgm(\bVt[  ][\Xnm, \dsf_{n-m}]) ) \le \sup_{\xv \in \Xnm}\!\!\dnq(\xv) + \norminf{\dnq - \dsf_{n-m}}.
        \label{eq:stability-p1}
    }
    \label{theorem:momdist-stability}
\end{theorem}

\begin{remark}The following observations follow from Theorem~\ref{theorem:momdist-stability}.
\begin{enumerate}[label=\textup{{(\roman*)}}, itemindent=*]
    \item In contrast to what would follow from Lemma~\ref{lemma:anai-et-al}~(ii) for the standard unweighted filtration, the term appearing in the r.h.s.~of \cref{eq:stability-p1} completely eliminates the dependence on the Hausdorff distance between $\Xn$ and $\Xnm$ in the $\dnq-$filtration. More generally, the same bound in {Theorem~\ref{theorem:momdist-stability}} holds even when $V[\Xn,\dnq]$ is replaced by $V[\mathbb{M}, \dnq]$ for any set $\mathbb{M} \supseteq \Xnm$. 

    \item Notably, $V[\Xn, \dnq]$ remains resilient to outliers. To see this, observe that the first term appearing in the r.h.s. of \cref{eq:stability-p1} may be bounded as
    \eq{
        \sup_{\xv \in \Xnm}\dnq(\xv) = \sup_{\xv \in \Xnm}\abs{ \dnq(\xv) - \dx(\xv) } \le \norminf{ \dsf_{n, Q} - \dx },\nn
    }
    where the first equality follows from the fact that $\dx(\xv)=0$ for all $\xv \in \Xnm$. Therefore, from the proof of Theorem~\ref{theorem:momdist-sublevel}, the r.h.s.~of \cref{eq:stability-p1} vanishes with high probability for sufficiently large sample sizes. 
    
    \item For $p=1$, a similar analysis for the DTM-filtrations appears in \cite[Theorem~{4.5}]{anai2019dtm} and the bottleneck distance is bounded above as
    \eq{
        \winf\qty\bigg( \dgm\pa{\bbv[\Xn, \delta_{n, k}]}, \dgm\pa{\bbv[\Xnm, \delta_{n-m, k}]} ) \le \sqrt{\f nk} W_2\pa{\Xnm, \Xn} + \sup_{\mathclap{\xv \in \Xnm}}\delta_{n-m, k}.\nn
    }
    While the last term on the r.h.s. converges to the uncontaminated population analogue with high probability, the first term involving the Wasserstein distance $W_2(\Xnm, \Xn)$ can be large even for a few extreme outliers. In contrast, the r.h.s. of \cref{eq:stability-p1} converges to zero with high probability with no assumptions on the outliers $\Ym$. 
    \item  {The power parameter $p$ determines a computational vs. statistical trade-off in Theorem~\ref{theorem:momdist-stability}. When $p=1$, the $\dnq$-weighted filtration, $\bVt[  ][\Xn, \dnq]$, provides a provably good approximation of the underlying signal, $\bVt[  ][\Xnm, \dsf_{n-m}]$. However, as $p$ increases the number of non-trivial points in the filtration $\bVt[  ][\Xn, \dnq]$ is non-increasing \cite[Proposition~8]{anai2019dtm}. Therefore, the $\dnq$-weighted persistence diagrams obtained using $\bVt[  ][\Xn, \dnq]$ for larger $p$ are \textit{sparser} than their counterpart when $p=1$.}

    \item The proof of the result is based on a generalization of the ideas presented in Lemma~4.8 and Proposition~4.9 in \cite{anai2019dtm} for DTM-filtrations. Lemmas~\ref{lemma:ab-filtration}~and~\ref{lemma:ab-module} in the supplementary material hold for general $f$-filtrations which satisfy a certain property, and the proof of Theorem~\ref{theorem:momdist-stability} follows by showing that $\dnq$ satisfies the required property.
\end{enumerate}
\label{remark:stability}
\end{remark}

With this background, we are now in a position to state our main result, which characterizes the rate of convergence for the $\dnq$--weighted filtration on the contaminated sample points $V[\Xn, \dnq]$ to $V[\bX]$. 

\begin{theorem}[$\dnq$-weighted filtration]
    Let $p=1$. Suppose $\pr \in \mathcal{P}(\bX, a, b)$ is a probability distribution with support $\bX$ satisfying the $(a, b)-$standard condition, and $\Xn = \Xnm \cup \Ym$ is obtained under sampling condition \ref{setting}. Then, for $2m < Q < n$ and for all $\delta \in (0, 1)$,
    \eq{
        \pr\qty\Bigg{\Winf\qty\bigg( \bVt[  ][\Xnm \cup \Ym, \dnq], \bvt{}[\bX] )  \le {2}\mathfrak{f}(n, m, Q, \delta_1, \delta_2)} \ge 1 - \delta,\nn
    }
    where
    {
    \eq{
        \mathfrak{f}(n, m, Q, \delta_1, \delta_2) \defeq 2\qty\bigg( \frac{Q\log(n / Q)}{an} + \frac{4Q \log(1/\delta_1)}{a(Q-2m)n} )^{1/b}\!\! +\! \qty\bigg( \frac{\log (n-m)}{a (n-m)} + \frac{4 \log(1/\delta_2)}{a (n-m)} )^{1/b},\nn
    }
    }
    for $\delta_1, \delta_2 \in (0, 1)$ such that $\delta_1 \le e^{-(1+b)Q}$ and $\delta_1 + \delta_2 = \delta$. In particular, if $m_n = cn^\e$ for $0 \le \e < 1$, then {for all $Q = Cn^\e$ where $C > 2c$,}
    \eq{
        \E\qty\bigg[ \Winf\qty\bigg( \bVt[  ][\Xnm \cup \Ym, \dnq], \bvt{} [\bX] ) ] \lesssim \qty\Bigg( \f{\log n}{n^{1-\e}} )^{1/b}.
        \label{eq:dnq-rate-1}
    }
    \label{theorem:momdist-consistency}
\end{theorem}

\begin{remark} We make the following observations from Theorem~\ref{theorem:momdist-consistency}.
        \begin{enumerate}[label=\textup{(\roman*)}, itemindent=*]
            \item The term appearing in the r.h.s.~of \cref{eq:dnq-rate-1} is identical to the term appearing in the r.h.s.~of \cref{eq:dnq-rate} in Theorem~\ref{theorem:momdist-sublevel}. Therefore, the $\dnq$--weighted filtration and the $\dnq$ sublevel filtration converge to the same population limit with identical convergence rates and match the minimax lower bound from \cref{thm:minimax} up to a $(\log{n})^{1/b}$ factor.
            \item The uniform confidence band we obtain from Theorem~\ref{theorem:momdist-consistency} can, in principle, be computed for any confidence level $\delta \in (0, 1)$. However, the restriction on $\delta_1$ makes the confidence band obtained using $V[\Xn, \dnq]$ wider than that obtained using Theorem~\ref{theorem:momdist-sublevel}. This is, ultimately, the price we have to pay for choosing the computationally tractable $\dnq$-weighted filtration as the estimator as opposed to the $\dnq$ sublevel filtration. 
        \end{enumerate}
\end{remark}

\begingroup

\section{Influence analysis}
\label{sec:influence}

The statistical analysis in the previous sections establishes that, even in the presence of outliers, as the number of samples increases we can eventually mitigate the effect of the outliers. In this section, we provide a more precise characterization for the influence the outliers have on the resulting $\dnq$--weighted filtrations, in contrast to the non-robust counterpart---the $\dn$--weighted filtrations.

Given a probability measure $\pr \in \mathcal{P}(\bX, a, b)$, \citet[Definition~4.1]{vishwanath2020robust} characterized the influence an outlier at $\xv \in \R^d$ has on a persistence diagram $\dgm\pa{\bbv[f_\pr]}$---obtained using the sublevel sets of $f_\pr$---using the \textit{persistence influence} function 
\eq{
    \boldsymbol{\Psi}(f_\pr; \xv) \defeq \lim_{\e \rightarrow 0} \winf\qty\Big( \dgm\pa{\bbv[{f_{\pr^{\e}_{\xv}}}]}, \dgm\pa{\bbv[f_\pr]} ),
    \label{eq:persinf}
}
where {$\pr^{\e}_{\xv} = (1-\e)\pr + \e\delta_{\xv}$} is the perturbation curve w.r.t.~$\xv$ in the space of probability measures. The persistence influence is a generalization of the influence function in robust statistics \citep{hampel2011robust} to general metric spaces. The analysis in this section is similar in spirit to the analysis based on the persistence influence but differs in two important aspects. First, the $\dnq$--weighted filtration is computed purely on the sample points---by partitioning the samples into $Q$ disjoint blocks---and, therefore, the notion of persistence influence is adapted to the samples, in contrast to \eref{eq:persinf}, which is based on the data-generating distribution $\pr$. Additionally, unlike the case of the persistence influence function---where the influence of outliers in the resulting persistence diagram is quantified in terms of the bottleneck distance---here we directly examine the influence the outlying point has on the resulting persistence diagram itself. This provides a more tractable interpretation of how outliers impact the resulting topological inference.

\begingroup
\renewcommand{\Xnm}{\bX[n+m]}

With this background, we now introduce the empirical persistence influence framework. Suppose we are given a collection of observations $\Xn$, which is sampled i.i.d.~from a probability distribution $\pr$ of interest. Let $\dgm\pa{\bVt[ ][\Xn, f_n][]}$ be its weighted--Rips persistence diagram, where the weight function $f_n$ is constructed using the samples $\Xn$. Suppose $\Xn$ is contaminated with $m < \f n2$ outliers to obtain the contaminated dataset $\Xnm$. In particular, we may assume that the $m$--points are placed at an outlying location $\xvo$, i.e.,
\eq{
    \Xnm = \Xn \bigcup \pb{\mathop{\medcup}\limits_{j=1}^m\pb{\xvo}},\label{eq:Xnm-influence}
}
such that the factor $m$ and the location $\xvo$ together control the relative influence the outliers have. This is similar to the role played by the factor $\e$ in the perturbation curve associated with the persistence influence. Note that when $m=0$, the influence of the outliers is non-existent in the dataset. 

\begin{remark}
    {Note that, unlike the results in the preceding sections where we assumed that $\Xn = \bX[n-m] \cup \bY[m]$ with $(n-m)$ points sampled from $\pr$, here, for ease of exposition, we assume that $\Xnm = \Xn \cup \Ym$ where $n$ points are sampled from the signal.}
\end{remark}

Let $\dgm\pa{\bVt[ ][\Xb, f_{n+m}]}$ be the $f_{n+m}$-weighted persistence diagram constructed on $\Xnm$. %

In a similar vein, we may characterize the influence the outliers have on the persistence diagrams resulting from the sublevel filtrations as
\eq{
    \boldsymbol{\widetilde\Psi}\pa{\winf; \Xn, f_n, m, \xvo} = \winf\qty\Big( \dgm\pa{\bbv[f_{m+n}]}, \dgm\pa{\bbv[f_n]} ) \le \norminf{f_{n+m} - f_{n}}.
    \label{eq:winf-influence}
}

Indeed, when $\winf\qty\big( \dgm\pa{\bbv[f_{m+n}]}, \dgm\pa{\bbv[f_n]} )$ is small, it follows that the persistence diagram $\dgm\pa{\bVt[ ][\Xnm, f_{n+m}]}$ is more robust (and vice versa). The following result establishes that, under some mild conditions and with high probability, the $\dnq$--weighted persistence diagrams are more robust than their non-robust counterpart.

\endgroup


\begin{theorem}[Influence analysis of $\dnq$-weighted filtrations]
    For $\Xn$ observed i.i.d. from $\pr \in \mathcal{P}(\bX, a, b)$ and $\xvo \in \R^d$, let $\Xmn$ be given by \cref{eq:Xnm-influence}. 
    For $2m < Q < n+m$, let $\dsf_{n+m}$ and $\dsf_{n+m, Q}$ denote the distance and MoM distance function w.r.t. $\Xmn$, %
    and let $\nQ = (n+m)/Q$ and $c = {\min\qty{a2^{-(1+2b)}, a2^{-3b}}}$. If
    
    {
    \eq{
        \vp \defeq c\ \!\dx(\xvo)^b > \f{\log\nQ}{\nQ} + \f{4(1+b)Q}{\nQ} , \tag{I}
    }
    }
    then, for all $\delta \in (0,1)$ satisfying
    \eq{
        (1+b)Q \le \log(2/\delta) \le \f{\nQ\vp - \log\nQ}{4}, \tag{II}
    }
    with probability greater than $1-\delta$,
    \eq{
        \norminf{ \dsf_{n+m} - \dsf_n } - \norminf{\dsf_{n+m, Q} - \dsf_n} \ge %
        \qty({ \f{2\log\nQ}{a\nQ} } + { \f{8 \log(2/\delta)}{a\nQ} })^{1/b}.\nn%
    }
    \label{theorem:momdist-influence}
\end{theorem}

\begin{remark} The result from Theorem~\ref{theorem:momdist-influence} may be interpreted as follows.
    \begin{enumerate}[label=\textup{{(\roman*)}}, itemindent=*]

        \item When conditions (I) and (II) hold, then with high probability, persistence diagrams obtained using $\dnq$ are closer to the truth than those obtained using $\dsf_n$. Therefore, the interplay between $n$, $m$, and $\xvo$ is better understood by characterizing when conditions (I) and (II) hold. 
        
        \item For fixed $n$ observe that (I) is satisfied whenever $\dx(\xvo)$ is sufficiently large, i.e., $\xvo$ is sufficiently far away from the support. On the other hand, if $\xvo$ is fixed, then (I) is satisfied when $\log \nQ / \nQ$ is sufficiently small, i.e., $n$ is sufficiently large. Together, this implies that for condition (I) to be satisfied, either (a) we need the outliers to be sufficiently well-separated from the support $\Xb$ such that we are able to distinguish outliers $\xvo$ from the inliers $\Xn$, or (b) for outliers placed very close to the support $\Xb$ we need sufficiently many inliers $n$ for us to be able to distinguish them from the outliers. On the other hand, note that if $n$ and $m$ are fixed, then the r.h.s. of (I) is directly proportional to $Q$. Although $Q$ can take any values between $2m < Q < (n+m)$, choosing a value of $Q$ much larger than $2m+1$ will likely breach condition (I) for a fixed $\xvo$. Equivalently, for a suboptimal choice of $Q$, we need the outliers to be sufficiently far away from the inliers in order to be able to distinguish them. 

        \item The l.h.s. of (II) is equivalent to the constraint that $\delta \le e^{-(1+b)Q}$, which appears in Theorems~\ref{theorem:momdist-sublevel}~and~\ref{theorem:momdist-consistency}. The r.h.s. of (II) specifies a lower bound on the confidence level $\delta$. Condition (I) guarantees that the admissible values of $\delta \in (0,1)$ satisfying (II) is nonempty. For fixed $m, Q$ and $\xvo$, the r.h.s. of (II) is directly proportional to $n$, i.e., the lower bound vanishes as $n \rightarrow \infty$. 
        
        \item When conditions (I) and (II) are satisfied, we have the following lower bound from the l.h.s. of (II):
        \eq{
            \norminf{ \dsf_{n+m} - \dsf_n } - \norminf{\dsf_{n+m, Q} - \dsf_n} \gtrsim \qty({ \f{\log(n+m/Q)}{a(n+m)/Q}} + { \f{Q}{(n+m)/Q} })^{1/b}.\label{eq:inf-lb}
        }
        In the regime when $n,m \rightarrow \infty$, and for the optimal choice of $Q$, i.e., $Q=km$ for $k > 2$, the r.h.s. of \cref{eq:inf-lb} is non-trivial when $m = \Omega(n^{1/2})$. Therefore, under conditions (I) and (II), when there are sufficiently many outliers, there is greater evidence to support the robustness of $\dnq$. 
    \end{enumerate}
\end{remark}

\endgroup


\begingroup
\providecommand{\hQ}{\widehat{Q}}
\providecommand{\hm}{\widehat{m}}
\renewcommand{\ms}{m^*}
\providecommand{\h}{\mathfrak{h}}

\section{Auto-tuning the parameter $Q$}
\label{sec:lepski}

The result in Theorem~\ref{theorem:momdist-consistency} relies on the crucial assumption that the number of outliers $\ms$ is known \textit{a priori}. While this assumption may hold in certain adversarial settings, in general, this information may be unavailable. In order to make Theorem~\ref{theorem:momdist-consistency} more useful in practical settings, we discuss two solutions for calibrating the parameter $Q$. The first procedure is based on Lepski's method \citep{lepskii1991problem}, which is a powerful data-driven method for adaptive parameter selection. In this case, we also provide theoretical guarantees for the adaptively tuned estimator. The second procedure---which is based on some heuristic observations regarding the sample estimator $\bbv{}[\Xn, \dnq]$--- works well in practice, and may be used as a precursor to Lepski's method.  

When the number of outliers $\ms$ is known, choosing $Q^*=2\ms+1$ results in the rate of convergence in Theorem~\ref{theorem:momdist-consistency}. However, without access to $m^*$, Lepski's method provides a systematic procedure for selecting a parameter $\hQ$ which provides the same error guarantees as $Q^*$ \citep{birge2001alternative}. The procedure is as follows. Let $m_{\min}$ and $m_{\max}$ be two coarse bounds on (unknown) $\ms$ such that ${m_{\min} \le \ms \le m_{\max}}$. For a choice of  $\theta>1$, let $m(j) = \theta^{j}\mmin$ and define 
$$
\J \defeq \qty\Big{ j \ge 1 : \mmin \le m(j) < \theta\mmax }.
$$ 
For $\Xn$ obtained under sampling condition \ref{setting}, let $\bbv_n(j) = \bbv[\Xn, \dsf_{n, Q(j)}]$ be the persistence module obtained using the \md{-weighted} filtration with ${Q(j) = 2m(j)+1}$. 

For $\delta \in (0,1)$ and $\delta_{\max} = \delta - e^{-(1+b)(2\mmax+1)}$, let $\mathfrak{h}(n, m, \delta)$ be defined as follows:
\eq{
    \mathfrak{h}(n, m, \delta) &= {4}\qty\Bigg( \f{2m+1}{an} \wo\qty( \f{ne^{ 4(1+b)(2\mmax+1) }}{2m+1} ) )^{1/b}\\ 
    &\qquad+ {2}\qty\Bigg( \f{1}{a(n-m)} \wo\qty( (n-m)e^{4\log(1/\delta_{\max})} ) )^{1/b},\nn
}
where for $z>0$, $\wo(z)$ is the Lambert $\wo$ function given by the identity $\wo(z)e^{\wo(z)} = z$. With this background, let $\hj$ be the output of the following procedure:
\eq{
    \hj \defeq \min \qty\Big{ j \in \J : \winf\pa{ \bbv_n(j), \bbv_n({j'}) } \le 2\mathfrak{h}(n, m(j'), \delta) \qq{for all} j' \in \J, j' > j },
    \label{eq:lepski-j}
}
the resulting weighted persistence module $\widehat{\bbv}_n = \bbv_n({\hj}) = \bbv[\Xn, \dsf_{n, Q(\hj)}]$ is the Lepski estimator for $\bbv[\bX]$. The following result establishes that the adaptive selection of $Q$ results in an estimator with the same convergence guarantees as in Theorem~\ref{theorem:momdist-consistency}.

\begin{theorem}[Adaptive $\dnq$-weighted filtration]
    Suppose $\Xn$ is obtained under sampling condition \ref{setting} for $\pr \in \mathcal{P}(\bX, a, b)$, and suppose $\mmin$ and $\mmax$ are known such that unknown number of outliers, $\ms$, such that $0 \le \mmin \le \ms \le \mmax < n/2$.  For a chosen $\theta > 1$ let $\hj$ be the output of data-driven procedure in \cref{eq:lepski-j} and let $\widehat{\bbv}_n = \bbv_n{(\hj)}$. Then, for all $\delta \in (0, 1)$,
    \eq{
        \pr\qty\bigg( \winf\qty\Big( \dgm\qty\big(\widehat{\bbv}_n), \dgm\qty\big(\bbv[\bX]) ) \le 3\mathfrak{h}(n, \theta m^*, \delta) ) \ge 1 - \delta \log_\theta\qty( \f{\theta \mmin}{\mmax} ).\nn
    }
    \label{theorem:lepski}
\end{theorem}

\begin{remark}
    We make the following useful observations from Theorem~\ref{theorem:lepski}.
    \begin{enumerate}[label=\textup{{(\roman*)}}, itemindent=*]
        \item We make the distinction that the output $\widehat\bbv_n$ of Lepski's method does not necessarily correspond to the optimal choice $\bbv_n^*$ if $\ms$ were known. Instead, Theorem~\ref{theorem:lepski} guarantees that the error associated with $\widehat\bbv_n$ is of the same order (up to constants) as that of $\bbv_n^*$. 

        \item While Lepski's method guarantees optimal errors for the adaptive estimator without any knowledge of the true $m^*$; in practice, however, the empirical performance depends on several factors. Since the procedure in Theorem~\ref{theorem:lepski} is designed to match the guarantee of Theorem~\ref{theorem:momdist-consistency}, the success of the procedure crucially depends on the tightness of the bound $\mathfrak{f}(n, m, Q, \delta_1, \delta_2)$ in Theorem~\ref{theorem:momdist-consistency}. Furthermore, the implementation described in \cref{eq:lepski-j} requires knowledge of the parameters $a, b > 0$ arising from the $(a,b)-$standard condition. While the calibration of $a$ and $b$ in practice is more of an art and beyond the scope of the paper, we emphasize here that it is possible to construct a statistically consistent estimator of the true population quantity $\bbv[\bX]$ in a purely data-adaptive fashion, even in the presence of adversarial contamination. 
        
        \item Unlike a standard grid search, Lepski's method adapts to the true noise level $m^*$ in an efficient manner. Given a reasonable estimate for $\mmin$ and $\mmax$, Lepski's method has a computational cost of $O( \log^2_\theta(\mmax/\mmin ) )$. However, the choice of $\theta > 1$ must also be made judiciously, e.g., replacing $\theta$ with $\sqrt{\theta}$ for the procedure in \cref{eq:lepski-j} will require $\sim4$ times more computational time.

        \item In the worst case, when there are no reasonable estimates for $\mmin$ and $\mmax$, choosing $\mmin=1$ and $\mmax=n/2$ requires $O(\log^2_\theta(n))$ computational time. Notably, more than just the additional computational price, a suboptimal choice of $\mmin$ and $\mmax$ leads to poor performance. To see this, note that the term $\h(n, m, \delta)$ is a lower bound for the term $\mathfrak{f}(n, m, Q, \delta_1, \delta_2)$ in Theorem~\ref{theorem:momdist-consistency} when $Q=2m+1$ and $\delta_1 = e^{-(1+b)(2\mmax+1)} \le e^{-(1+b)Q}$. Therefore, when the number of outliers grows with $n$ as $m^* = cn^\e$ for $c>0$ and $\e \in [0, 1)$, a similar analysis to that in Theorem~\ref{theorem:momdist-sublevel} and Theorem~\ref{theorem:momdist-consistency} yields that
        \eq{
        \E\qty\bigg[ \winf\qty\Big( \dgm\qty\big(\widehat{\bbv}_n), \dgm\qty\big(\bbv[\bX]) ) ] \lesssim \pa{\f{\log n}{n/\mmax}}^{1/b}.\nn
        }
        Therefore, if the bound $\mmax$ is not tight, i.e., $\mmax = Cn^\beta$ for $\epsilon < \beta$, then, asymptotically, the output of Lepski's method is not adaptive to the true noise $m^*$, and, instead, reflects the suboptimal choice of $\mmax$.
    \end{enumerate}
    \label{remark:lepski}
\end{remark}

This method may also be used to adaptively select the parameter $Q$ for the sublevel set persistence module. The following result outlines a data-driven procedure to obtain ${\cj \in \J}$ such that the resulting sublevel persistence module $\overline{\bbv}_n = \bbv_n(\cj) = \bbv[\dsf_{n, Q(\cj)}]$ has the same convergence guarantee as Theorem~\ref{theorem:momdist-sublevel}.

\begin{corollary}[Adaptive sublevel filtration]
    For $\pr \in \mathcal{P}(\bX, a, b)$, suppose $\Xn$ is obtained under sampling condition \ref{setting}, and suppose $\mmin$ and $\mmax$ are known such that unknown number of outliers, $\ms$, is such that $0 \le \mmin \le \ms \le \mmax < n/2$. Let $\bbw_n(j) = \bbv[\dsf_{n, Q(j)}]$ be the sublevel persistence module obtained using $\dsf_{n, Q(j)}$ with ${Q(j) = 2m(j)+1}$ for all $j \in \J$. For a chosen $\theta > 1$, let $\cj$ be the output of data-driven procedure,
    \eq{
        \cj = \min \qty\Big{ j \in \J : \winf\pa{ \bbv_n(j), \bbv_n({j'}) } \le 2\mathfrak{p}(n, m(j'), \delta) \qq{for all} j' \in \J, j' > j },\nn
    }
    where
    \eq{
        \mathfrak{p}(n, m, \delta) = \qty\Bigg( \f{2m+1}{an} \wo\qty( \f{ne^{ (1+b)\log(1/\delta) }}{2m+1} ) )^{1/b}.\nn
    } 
    Then, for all $\delta \le e^{-(1+b)(2\mmax+1)}$ and $\overline{\bbv}_n = \bbv_n{(\cj)}$,
    \eq{
        \pr\qty\bigg{ \winf\qty\Big( \dgm\qty\big(\overline{\bbv}_n), \dgm\qty\big(\bbv[\bX]) ) \le 3\mathfrak{h}(n, \theta m^*, \delta) } \ge 1 - \delta \log_\theta\qty( \f{\theta \mmin}{\mmax} ).\nn
    }
    \label{corollary:lepski}
\end{corollary}

\endgroup

The proof is identical to that of Theorem~\ref{theorem:lepski}, and is, therefore, omitted. The success of Lepski's method depends on the tightness of the probabilistic bounds, knowledge of the (nuisance) parameters (i.e. $a,b$) appearing in these bounds, and a prudent choice for $\mmin$ and $\mmax$. While the calibration of $a$ is beyond the scope of this paper, in $\R^d$ a conservative choice for $b$ would be the dimension $d$ of the ambient space. We refer the reader to \citet[Section~4]{chazal2015convergence} for further details. 

To address the last bottleneck in Lepski's method, we describe a heuristic method to select the parameter $Q$, which may be used to obtain reasonable choices for $\mmin$ and $\mmax$. %

The method is based on the observation that the blocks $\pb{S_q : q \in [Q]}$ may be resampled by shuffling the sample points $\Xn$ prior to partitioning it. The resulting estimator $\bbv[\Xn, \dnq]$ is an unbiased estimator of the same population quantity when ${2m < Q < n}$. Therefore, we may choose the smallest value of $Q$ for which the pairwise bottleneck distance over permutations of the data is minimized. Specifically, suppose $\Xn^\sigma = \pb{ \Xv_{\s(1)}, \Xv_{\s(2)}, \dots, \Xv_{\s(n)} }$ is a permutation of  $\Xn$, then
\eq{
    \widehat{Q}_{R} = \argmin_{Q \ge 1} \sum_{\mathclap{1 \le i < j \le N}} \winf\qty\Big({ \bbv[{\Xn^{\sigma_i}}, \dnq], \bbv[{ \Xn^{\sigma_j} }, \dnq] }),\nn
}
where, for a chosen number of replicates $N$, $\sigma_i, \sigma_j$ are permutations of $[n]$ for each $i, j \in [N]$. Furthermore, for $\widehat{m}_{R} = \lfloor{\widehat{Q}_{R}/2}\rfloor$ and for a constant $C > 1$, the bounds $\mmin$ and $\mmax$ may be taken to be $C\inv\widehat{m}_{R}$ and $C\widehat{m}_{R}$, respectively.



\section{Experiments}
\label{sec:experiments}

In the following section, we supplement the theory through the illustration of the performance of the robust filtrations $\bbv[\dnq]$ and $\bbv[\Xn, \dnq]$ in synthetic experiments. 

\subsection{Adaptive calibration of $Q$}
\label{exp:adaptive}

For $n=500$, $K=30$ replicates and for each $i \in [K]$, point clouds $\Xn^{(i)}$ are generated on a circle, and $m^{(i)} \sim \text{Unif}\qty(50, 150)$ outliers added from a Matérn cluster process. This is illustrated in Figure~\ref{fig:lepski}\,(a). Taking $m_{\min} = 20$, $m_{\max}=200$ and $\theta=1.07$, the adaptive estimate $\widehat{m}^{(i)}$ is computed using Lepski's method, and $\widehat{m}_R^{(i)}$ is computed using the heuristic method described in Section~\ref{sec:lepski} with $N=50$. For a single replicate ${i \in [K]}$, Figure~\ref{fig:lepski}\,(b) plots $\sum_{1 \le i < j \le N} \winf\qty\big({ \bbv[{\Xn^{\sigma_i}}, \dnq], \bbv[{ \Xn^{\sigma_j} }, \dnq] })$ vs. $Q$. In most cases, we have observed that the resampled bottleneck distance criterion stabilizes shortly before the optimal value of $m$. Figure~\ref{fig:lepski}\,(c) shows a boxplot for the relative errors $\qty\big{ \widehat{m}^{(i)} - {m}^{(i)} / {m}^{(i)} : i \in [K] }$ and $\qty\big{ \widehat{m}_R^{(i)} - {m}^{(i)} / {m}^{(i)} : i \in [K] }$ for Lepski's method and the heuristic procedure, respectively. Lepski's method is fairly robust to the choice of the hyperparameters, and, consistently selects $\widehat m^{(i)} \ge m^{(i)}$. In contrast, since the resampled bottleneck distance from the heuristic procedure often stabilizes before $m^{(i)}$, we observe that $\widehat{m}^{(i)}_R < m^{(i)}$.

\begin{figure}[t]
    \begin{subfigure}[b]{0.32\textwidth}
        \includegraphics[height=\textwidth]{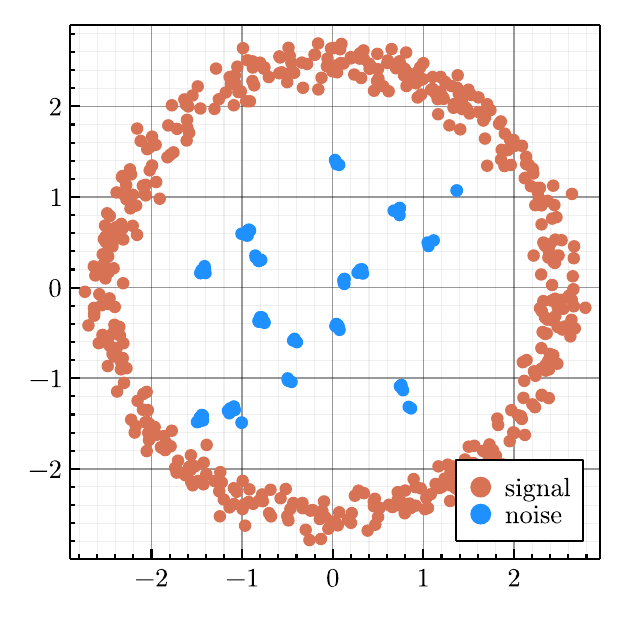}
        \caption{Scatterplot of $\Xn$}
    \end{subfigure}
    \begin{subfigure}[b]{0.32\textwidth}
        \includegraphics[height=\textwidth]{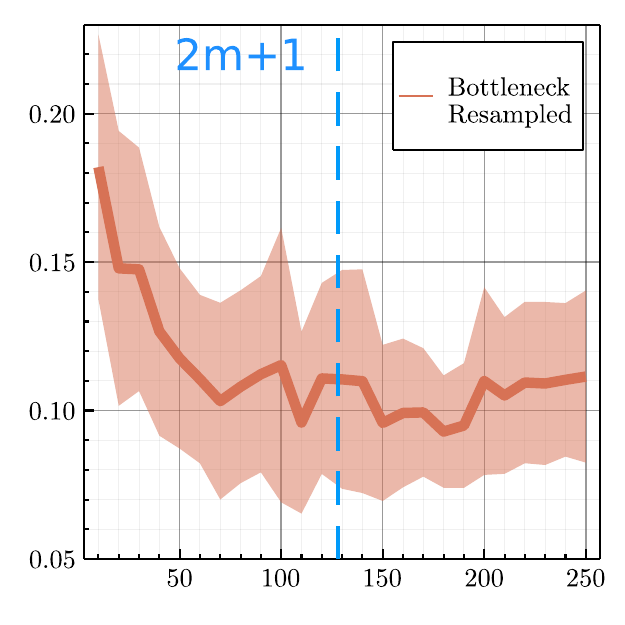}
    \caption{Heuristic procedure}
    \end{subfigure}
    \begin{subfigure}[b]{0.3\textwidth}
        \includegraphics[height=1.05\textwidth]{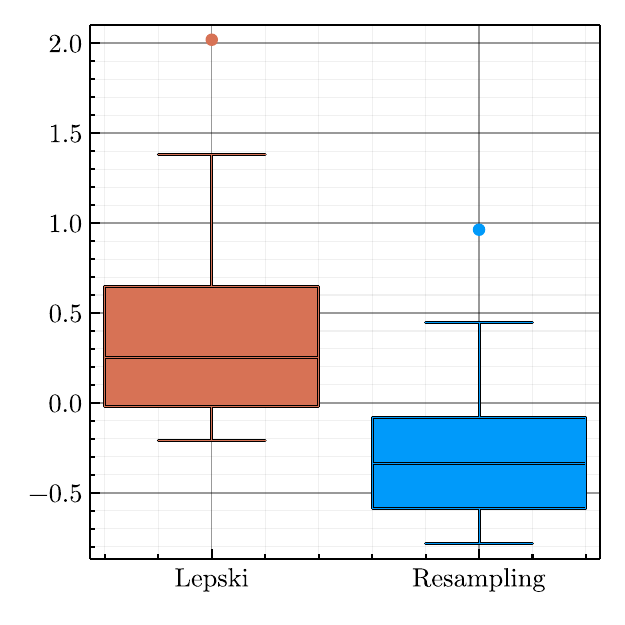}
        \caption{Relative error: Lepski vs. Resampling}
    \end{subfigure}
    \caption{Comparison of Lepski's method and the heuristic procedure for selecting the parameter $Q$.}
    \label{fig:lepski}
\end{figure}

\subsection{High dimensional topological inference}
\label{exp:highdim}

In this experiment, we illustrate the advantage of using $\dnq$-weighted filtrations for high dimensional topological inference. Points are uniformly sampled in $\R^3$ from two interlocked circles. Using a random rotation matrix $Q \in {SO}(100)$, the points are transformed to an arbitrary configuration in $\R^{100}$. The samples $\Xn \subset \R^{100}$ are obtained by replacing $12.5\%$ of the points in $\R^{100}$ with outliers sampled from $\textup{Uniform}\pa{ [-0.2, 0.2]^{100} }$. A scatterplot for $\Xn$ projected to $3$ arbitrary coordinates is shown in Figure~\ref{fig:highdim}\,(a). Since the point cloud is embedded in $\R^{100}$, computing sublevel filtrations using cubical homology with the same resolution as earlier requires $(10/0.5)^{100} \approx 10^{131}$ simplices to be stored in memory. In contrast, computing the $\dnq$-weighted filtrations 
is less intensive. Figure~\ref{fig:highdim}\,(b) shows the persistence diagram $\dgm(\widehat{\bbv}_n)$ obtained using $\dnq$-weighted filtrations, where the parameter $Q$ is adaptively selected using Lepski's method. The two $1$st order homological features underlying the interlocked circles are recovered. Figure~\ref{fig:highdim}\,(c) illustrates the persistence diagram $\dgm(\bbv [{\Xn, \delta_{n,k}}])$ obtained using DTM-weighted filtrations where the parameter $k=10$ was selected manually. 

\begin{figure}[H]
    \begin{subfigure}[b]{0.34\textwidth}
        \centering
        \includegraphics[width=\textwidth]{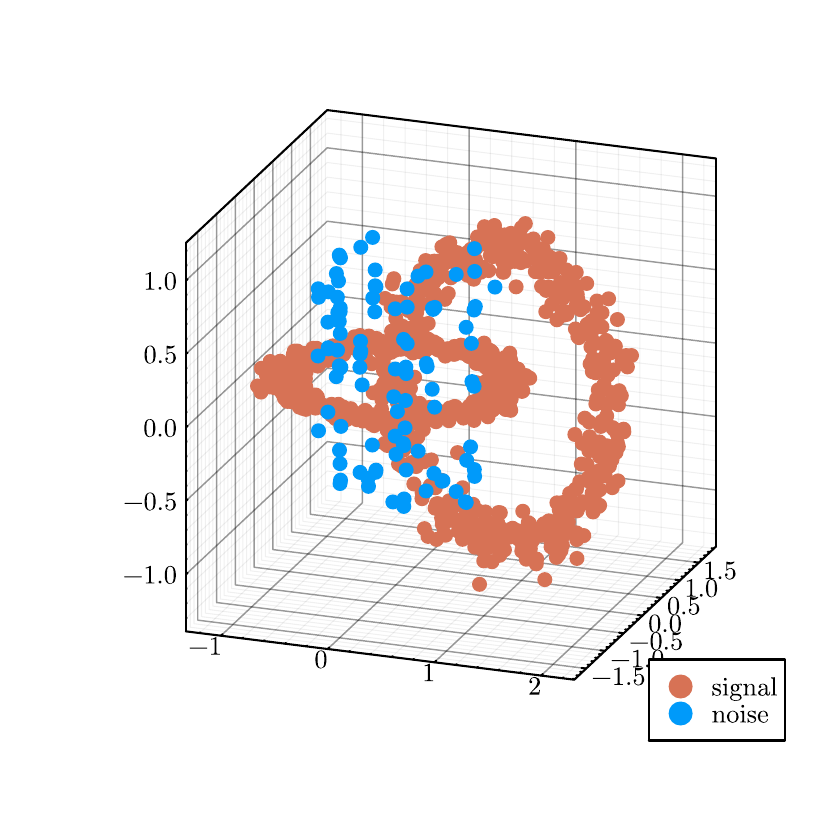}
        \caption{{$\Xn$ before transformation}}
    \end{subfigure}
    \begin{subfigure}[b]{0.31\textwidth}
        \centering
        \includegraphics[height=\textwidth]{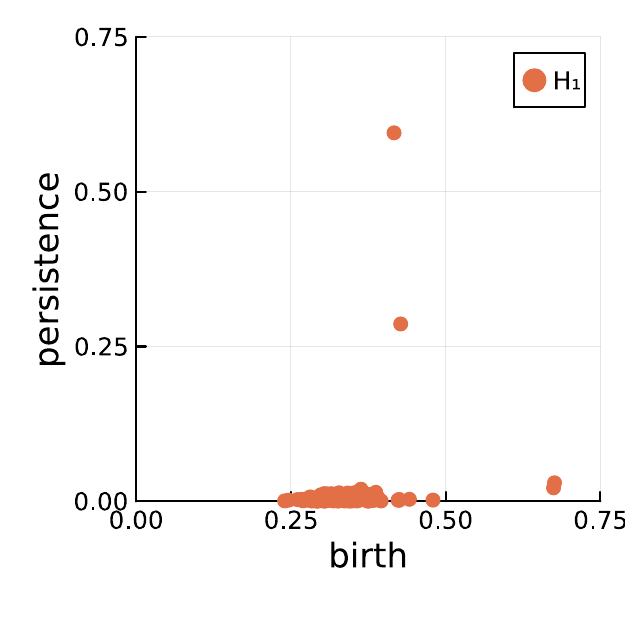}
        \caption{MOM; $\dgm({\widehat{\bbv}_n})$}
    \end{subfigure}
    \begin{subfigure}[b]{0.31\textwidth}
        \centering
        \includegraphics[height=\textwidth]{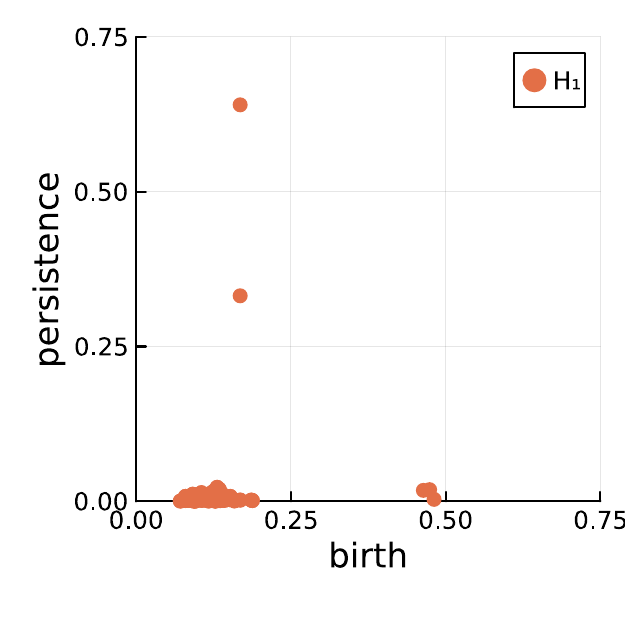}
        \caption{{DTM}; $\dgm\pa{\bbv[\Xn, \delta_{n, k}]}$}
    \end{subfigure}
    \caption{Robust persistence diagrams for interlocked circles in $\R^{100}$ using $\dnq$ and $\delta_{n, k}$ weighted filtrations.}
    \label{fig:highdim}
\end{figure}

\subsection{Recovering the true signal under adversarial contamination}
\label{exp:mnist}

In this experiment, we illustrate how $\bbv[{\Xn, \dnq}]$ can be used to recover the true topological features in the presence of adversarial contamination. In Figure~\ref{fig:mnist}\,(a), we consider a $28 \times 28$ image for the digit ``6'' from the MNIST database \citep{deng2012mnist}. We consider the setting in which an adversary is allowed to manipulate $10\%$ of the image by modifying the pixel intensities. Figure~\ref{fig:mnist}\,(b) depicts the adversarially contaminated version of the image by transforming the ``6'' to an ``8''.

For each pixel $p$ with pixel intensity $\iota(p)$, we convert the image to a point cloud $\Xn \subset \R^2$ by sampling $10 * \iota(p)$ points uniformly from the region enclosed by the pixel. \mbox{Figures~\ref{fig:mnist}(d, e)} illustrate the point clouds obtained from the true and contaminated images with $n-m \approx 1100$ and $n \approx 1300$, respectively. The persistence diagrams constructed using the distance function $\dsf_n$ for the two point clouds are reported in Figures~\ref{fig:mnist}\,(g, h). The persistence diagram in Figure~\ref{fig:mnist}\,(h) indicates the presence of the additional loop introduced by the adversary. To account for the adversarial contamination, we compute the \md{} function $\dnq$ with the parameter $Q$ selected using the contamination budget, i.e., $Q = 1 + 2(1100 \times 10\%) = 221$. Figure~\ref{fig:mnist}\,(f) shows the adversarially contaminated point cloud with each point $\xv_i \in \Xn$ colored by the value of $\dnq(\xv_i)$. The resulting $\dnq$-weighted persistence diagram $\dgm(\bbv[\Xn, \dnq])$ is reported in Figure~\ref{fig:mnist}\,(f). We note that $\dgm(\bbv[\Xn, \dnq])$ recovers the prominent features of Figure~\ref{fig:mnist}\,(g) up to a rescaling. 

Additionally, for each pixel $p$ we compute a rescaled version of $\dnq$, given by
\eq{
    f_{n, Q}(p) = \f{\max\limits_x\dnq(x) - \dnq(p)}{\max\limits_x\dnq(x)},\nn
}
as a proxy for the pixel intensity obtained using $\dnq$. In Figure~\ref{fig:mnist}\,(c), we plot the level sets $\pb{p: f_{n,Q} = t}$ on the original image for $t \ge 0.8$.

\begingroup

\renewcommand{\Xnm}{\mathbb{X}_{n\minus m}}

\begin{figure}
    \begin{subfigure}[c]{0.3\textwidth}
        \includegraphics[width=\textwidth]{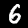}
        \caption{Signal}
    \end{subfigure}
    \quad
    \begin{subfigure}[c]{0.3\textwidth}
        \includegraphics[width=\textwidth]{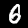}
        \caption{Adversarial contamination}
    \end{subfigure}
    \quad
    \begin{subfigure}[c]{0.305\textwidth}
        \includegraphics[height=\textwidth]{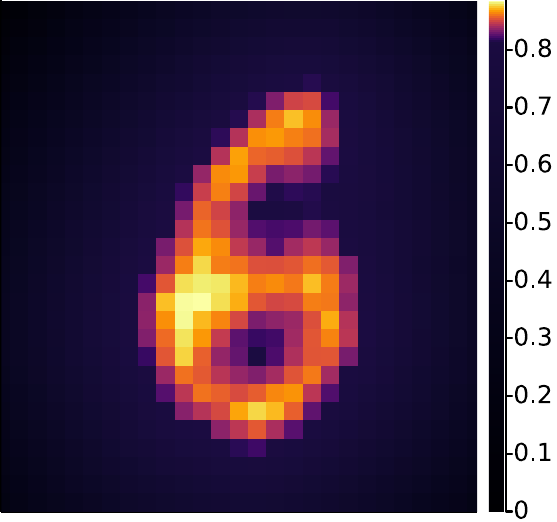}
        \caption{Recovered}
    \end{subfigure}
    \medskip
    \begin{subfigure}[c]{0.315\textwidth}
        \includegraphics[height=\textwidth]{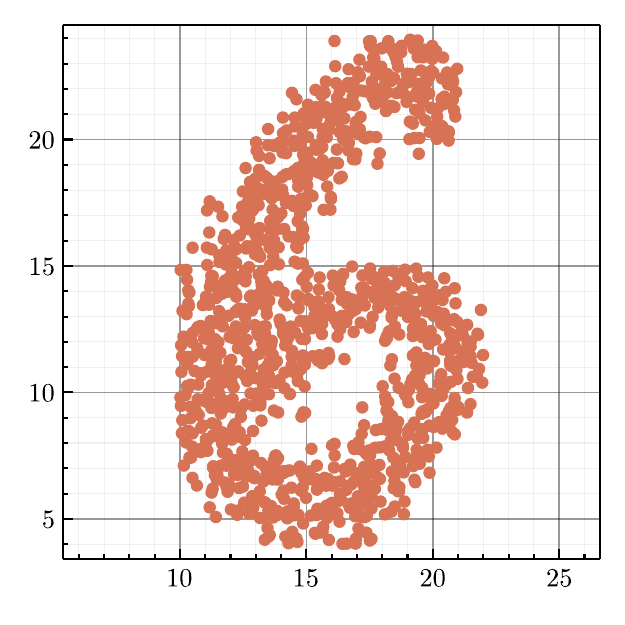}
        \caption{$\Xnm$}
    \end{subfigure}
    \begin{subfigure}[c]{0.315\textwidth}
        \includegraphics[height=\textwidth]{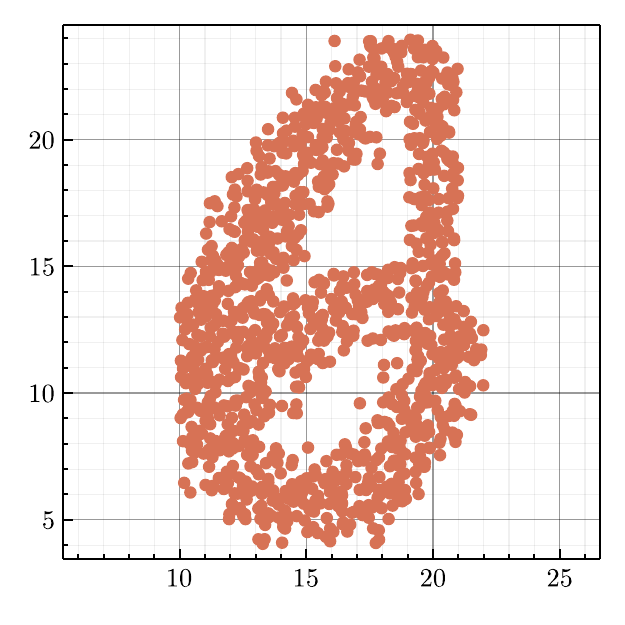}
        \caption{$\Xn$}
    \end{subfigure}
    \quad
    \begin{subfigure}[c]{0.315\textwidth}
        \includegraphics[height=\textwidth]{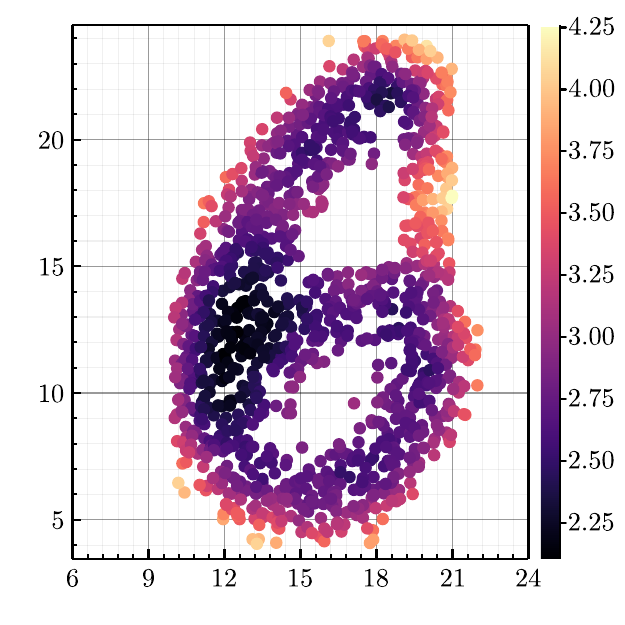}
        \caption{$\Xn$ colored by $\dnq$}
    \end{subfigure}
    \begin{subfigure}[c]{0.32\textwidth}
        \includegraphics[height=\textwidth]{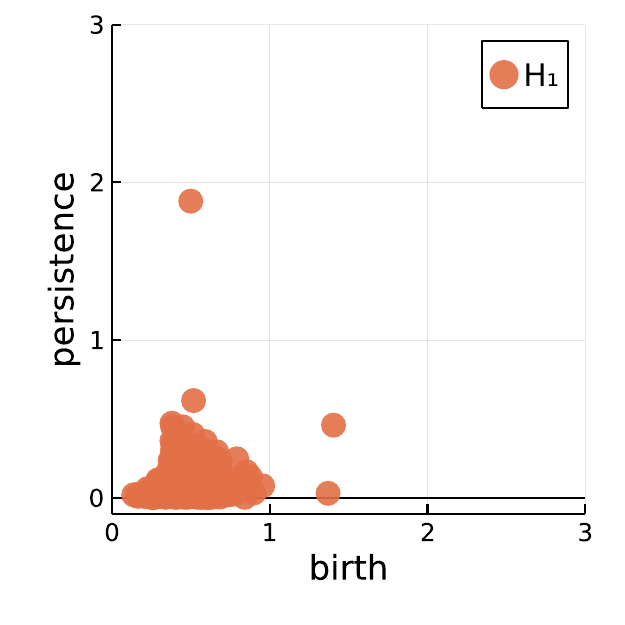}
        \caption{$\dgm\qty( \bbv[\Xnm] )$}
    \end{subfigure}
    \begin{subfigure}[c]{0.32\textwidth}
        \includegraphics[height=\textwidth]{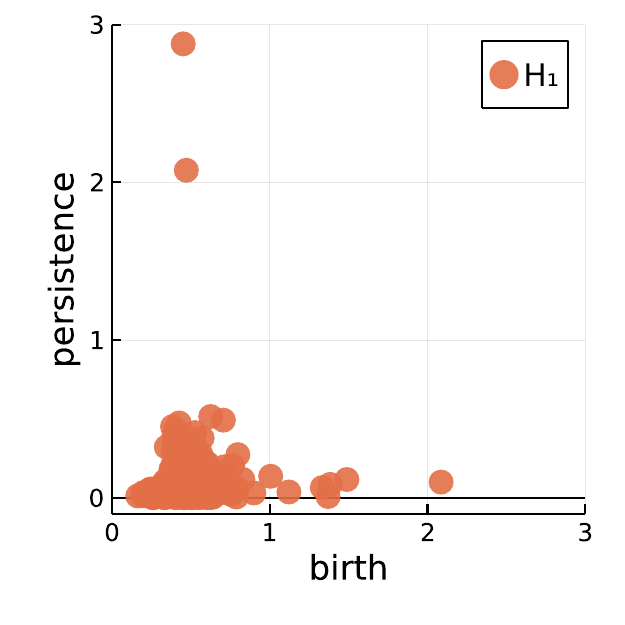}
        \caption{$\dgm\qty( \bbv[\Xn] )$}
    \end{subfigure}
    \begin{subfigure}[c]{0.32\textwidth}
        \includegraphics[height=\textwidth]{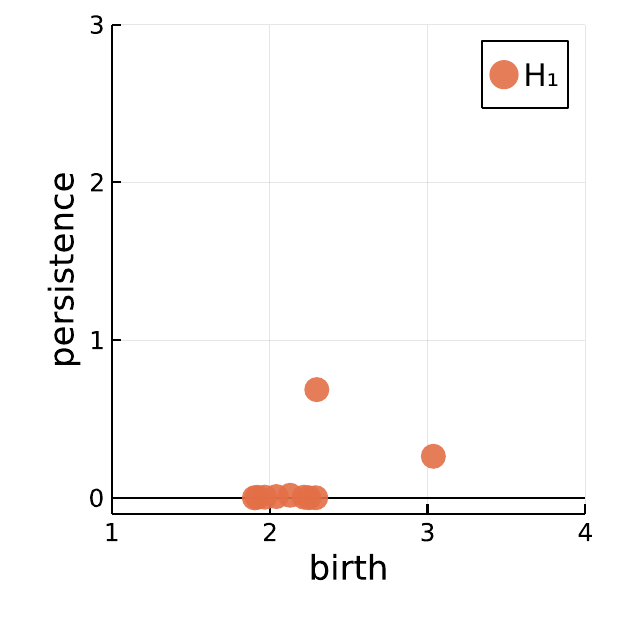}
        \caption{$\dgm\qty( \bbv[\Xn, \dnq] )$}
    \end{subfigure}\hspace{10pt}
    \caption{Recovering the topological information underlying the signal in the presence of adversarial contamination.}
    \label{fig:mnist}
\end{figure}

\endgroup

\subsection{Comparison with existing methods}
\label{exp:pathological}

{
For this experiment, we examine the robustness of several filtrations as the sample size $n$ grows and the number of outliers increases as $n=m^\epsilon$ for $\epsilon \in (0, 1)$. In particular, for each $n \in \qty{250,\cdots, 2000}$, $\Xnm \subset \R^2$  is sampled from the unit circle with additive Gaussian noise with $\sigma=0.1$, and compute the persistence diagram $\D[n]^*$. Outliers $\Ym$ are sampled from a Mat\'{e}rn cluster process with $m = n^\epsilon$ for $\epsilon=0.5$. For the composite sample $\Xn = \Xnm \cup \Ym$, persistence diagrams are constructed using the MoMDist, DTM, \kpdtm{} and the trimmed variant of the \kpdtm{}. For the MoMDist, we set $Q=2m+1$, and for the DTM we found that $k=10$ worked best. This is in line with the recommendation from \citet[Section~4.1]{brecheteau2018k}. For the \kpdtm{}, we set $N=n/5$ so as to not aggressively bias the $K$-point approximation, and for the trimmed version of the \kpdtm{}, we set $\alpha = (n-m)/n$ so that a total of $m = n(1-\alpha)$ points are considered outliers and trimmed \citet[Section~4.2]{brecheteau2018k}. The results are reported in Figure~\ref{fig:pathological}. For each persistence diagram $\D[n]$ we compare $\D[n]$ to $\D[n-m]^*$ in bottleneck distance $\winf\qty(\D[n], \D[n-m]^*)$. In addition, we also assess the quality of the signal in $\D[n]$ by computing 
$$
\Delta\text{-Lifetime}\qty(\D[n], \D[n-m]^*) := \frac{f(\D[n])}{f(\D[n-m]^*)} \qq{where} f(\D[]) = \frac{\text{pers}_{\text{max}}(\D[]) - \text{pers}_{\text{max}-1}(\D[])}{\text{pers}_{\text{max}}(\D[])},
$$
and $\text{pers}_{\max}(\D[])$ is the persistence of the longest bar in the first-order diagram of $\D[]$. Intuitively, $f(\D[])$ measures the excess persistence in the longest bar relative to the second longest bar in $\D[]$, and the ratio $\Delta\text{-Lifetime}(\D)$ measures how closely this matches the $f(\D[n-m]^*)$ in the signal. Ideally, we would like $\Delta\text{-Lifetime}\qty(\D[n], \D[n-m]^*)$ to be close to $1$ as this would indicate that the longest bar in $\D[n]$ is separated from the second longest bar to the same extent as that in $\D[n-m]^*$.

As corroborated by Theorem~\ref{theorem:momdist-consistency}, when $m = o(n)$ the bottleneck distance $\winf\qty(\D[n], \D[n-m]^*)$ for the MoMDist is decreasing as seen in Figure~\ref{fig:pathological}\,(a), and $\D[n]$ recovers the signal for moderately large $n$ as shown in Figure~\ref{fig:pathological}. On the other hand, the persistence diagrams for DTM, \kpdtm{}, and its trimmed variant are more influenced by the outliers. This can be attributed, to some extent, to the fact that the DTM is essentially inversely proportional to the density of the point cloud \citet{biau2011weighted}; therefore, the clusters of outliers in $\Ym$ have a noticeable impact on the DTM. By virtue of this, the \kpdtm{} and its trimmed variant are also influenced by the outliers. Interestingly, in most cases, the quantized \kpdtm{} performed better than the trimmed \kpdtm{}, as seen in Figure~\ref{fig:pathological}\,(b). In essence, when the outliers are particularly troublesome, even a few outliers misclassified as signal points can have a significant impact on the trimmed \kpdtm{}. We also note that, unlike the MoMDist, under certain circumstances, e.g., when the noise $\Ym$ is uniformly distributed in the ambient space and the signal is $(a, b)$-standard, the \kpdtm{} remains robust to outliers even when $m > n/2$. We refer the readers to \citet[Section~4.3]{brecheteau2018k} for experiments in this setting.

Lastly, in Table~\ref{tab:summary}, we provide a summary of the time and memory usage for the MoMDist, DTM, RKDE, and the \kpdtm{} for the same setup. We exclude the trimmed \kpdtm{} from the summary since it always takes longer than \kpdtm{}. In all cases, we use the publicly available implementation.\footnote{\url{https://github.com/GUDHI/TDA-tutorial/blob/master/Tuto-GUDHI-kPDTM-kPLM.ipynb}} We note that MoMDist and DTM are significantly faster and use less memory despite the fact that the \kpdtm{} constructs persistence diagrams using a smaller number of quantized points. The main computational bottleneck for the \kpdtm{} stems from the preprocessing step which uses a variant of Lloyd's algorithm, and is likely to be more efficient in settings where $n$ is very large. 
}

\begin{table}
    \centering
    \caption{Combined Summary of Time and Memory Usage}
    \label{tab:summary}
    \resizebox{\linewidth}{!}{
    \begin{tabular}{@{}lccccc|rrrrr@{}}
        \toprule
        & \multicolumn{5}{c}{Time ($\textup{seconds}$)} & \multicolumn{5}{c}{Memory ($\textup{MB}$)} \\ 
        \cmidrule(lr){2-6} \cmidrule(lr){7-11}
        \qquad$n$ & 500 & 1000 & 1500 & 2000 & 2500 & 500 & 1000 & 1500 & 2000 & 2500 \\
        \midrule
        MoMDist   & $0.047$ & $0.237$ & $0.570$ & $1.166$ & $2.105$ & $30.94$ & $109.79$ & $221.43$ & $426.67$ & $649.88$ \\
        DTM   & $0.054$ & $0.240$ & $0.593$ & $1.373$ & $2.177$ & $29.61$ & $106.99$ & $190.41$ & $419.92$ & $634.20$ \\
        \kpdtm{} & $0.539$ & $0.577$ & $2.523$ & $5.755$ & $5.838$ & $558.03$ & $2806.43$ & $5511.02$ & $14882.10$ & $21109.80$ \\
        RKDE  & $0.178$ & $1.216$ & $1.873$ & $3.961$ & $7.074$ & $93.49$ & $368.59$ & $659.98$ & $1202.16$ & $1968.14$ \\
        \bottomrule
    \end{tabular}
    }
    \end{table}

\begin{figure}
    \centering
    \includegraphics[width=\linewidth]{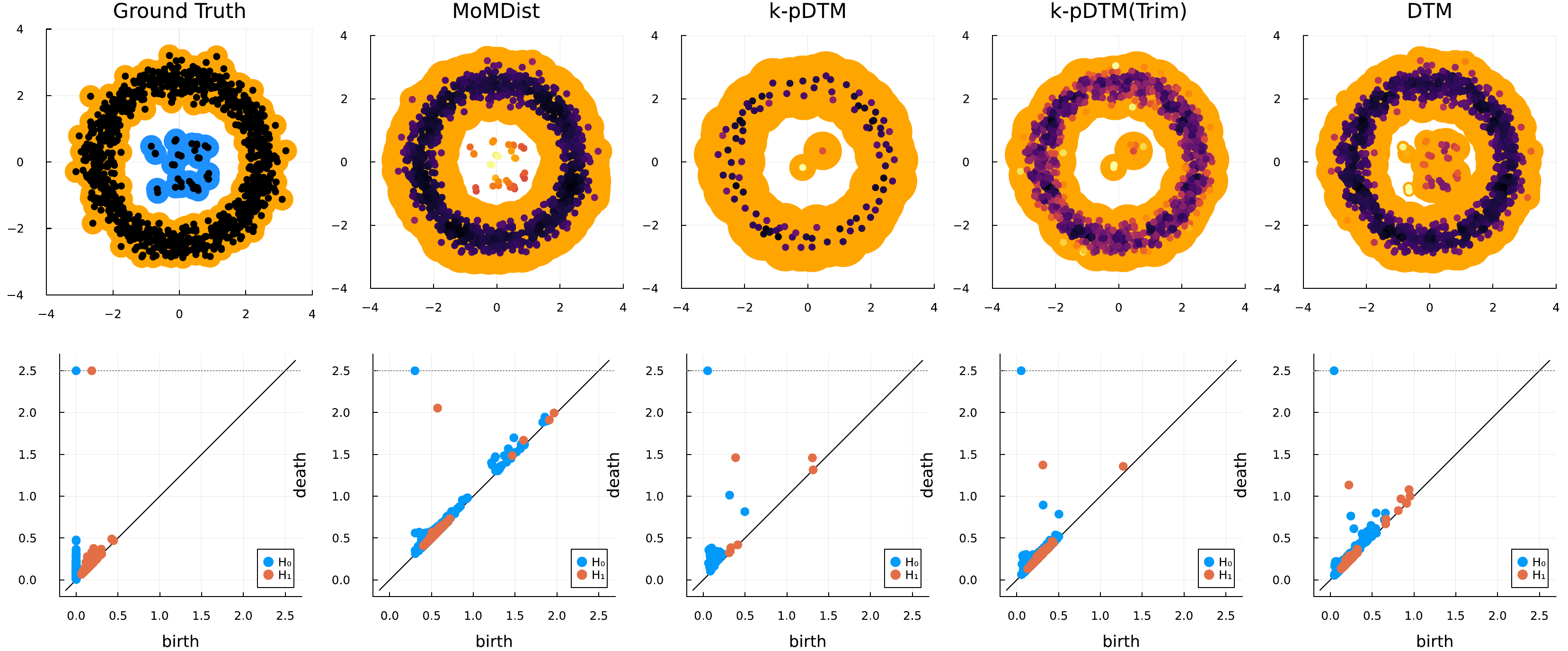}
    \caption{Illustration of the offsets $V^t_{n+m}$ and the persistence diagrams $\D[n+m]$ for the MoMDist, \kpdtm{}, the trimmed \kpdtm{}, and the DTM filtration, where $n-m$ points are sampled from a circle with additive Gaussian noise and $m$ outliers are sampled from a Mat\'{e}rn cluster process.}
    \label{fig:illustration}
\end{figure}

\begin{figure}
    \centering
    \small
    \begin{subfigure}[b]{0.44\linewidth}
        \includegraphics[height=\textwidth]{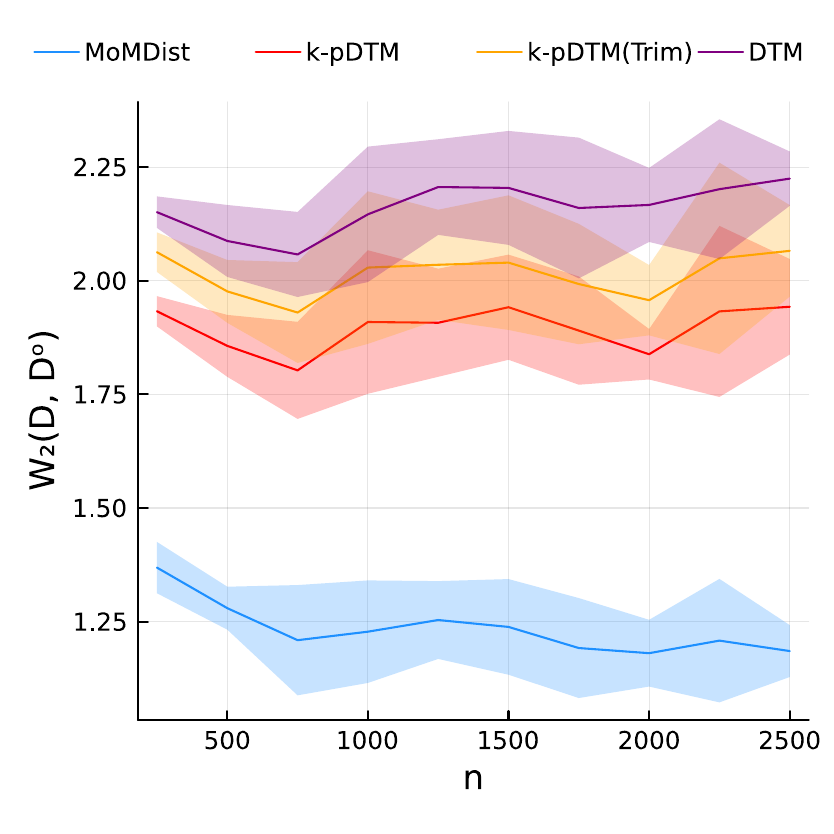}
        \caption{$\Winf\qty(\D[n], \D[n-m]^*)$}
    \end{subfigure}
    \quad
    \begin{subfigure}[b]{0.44\linewidth}
        \includegraphics[height=\textwidth]{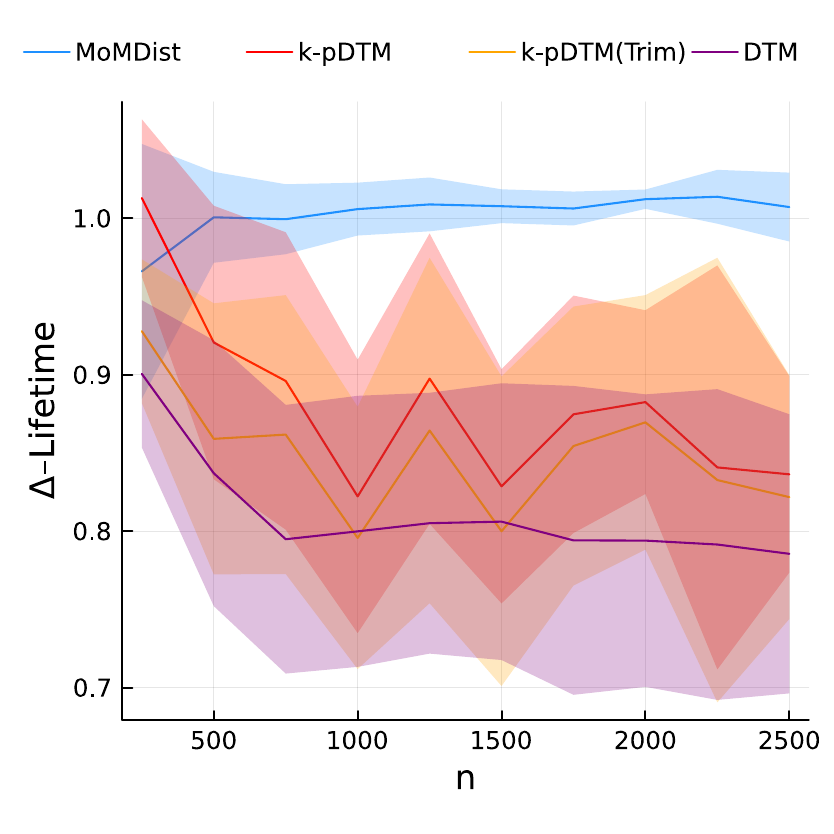}
        \caption{$\Delta$-Lifetime$\qty(\D[n], \D[n-m]^*)$}
    \end{subfigure}
    \caption{Comparison of MoMDist, DTM, \kpdtm{} and the trimmed \kpdtm{} for the experiment in \cref{exp:pathological}.}
    \label{fig:pathological}
\end{figure}



\section{Conclusion \& Discussion}
\label{sec:discussion}

In this paper, we introduce a methodology for constructing filtrations that are computationally efficient, provably robust, and statistically (near) optimal even in the presence of outliers. 

To elaborate, we introduced \md{}, $\dnq$, as a computationally efficient and outlier-robust variant of the distance function based on the median-of-means principle, and established some of its theoretical properties. In particular, when the samples contain outliers in the adversarial contamination setting, we (i) showed that the $\dnq$-weighted filtrations are statistically (near) optimal estimators of the true (uncontaminated) population counterpart, (ii) characterized its convergence rate in the bottleneck metric, and (iii) provided uniform confidence bands in the space of persistence diagrams. Furthermore, we used an empirical influence analysis framework to quantify the robustness of the $\dnq-$filtrations, and provide a framework for selecting the parameter $Q$. 

Topological inference in the presence of outliers is a topic that has received considerable attention in recent years and with good reason. We would like to highlight that the objective of this paper has been to develop a framework of topological inference in which the population target is the persistence diagram $\dgm\pa{\bbv[\Xb]}$. Therefore, the proposed methodology disregards, to a large extent, the distribution of mass on the support. As a future direction, we would like to explore a framework of inference that incorporates information from, both, the geometry of the underlying space and the structure of the probability measure generating the data. As noted in \citet[Section~5]{anai2019dtm}, their results follow only from a few simple properties of the distance-to-measure. We build off their foundation to provide some useful generalizations which we hope will be useful in the analysis of other estimators using this framework. 



\section*{Acknowledgements}
BKS is partially supported by the National Science Foundation (NSF) CAREER Award DMS-1945396. SK is partially supported by JSPS KAKENHI Grant Number 21H03403.  KF and SK are supported by JST, CREST Grant Number JPMJCR15D3, Japan.

\bibliographystyle{plainnat}
\bibliography{src/MoMRefs}
\clearpage

\appendix
\section*{Appendix}



\section{Background on Persistent Homology}
\label{sec:persistent-homology-appendix}

Given a compact set $\bX$, the building block of any topological data analysis pipeline to extract meaningful information from $\bX$ begins with a nested sequence of filtered topological spaces called a filtration, simply denoted by $V$. The sequence of spaces {is} parametrized by a resolution parameter $t$. There are several approaches for constructing a filtration using $\Xb$. One approach is to consider the collection of offsets built on top of $\Xb$, i.e., $V^t = V^t[\bX] = \Xb(t)$.
For $s < t$, the offsets are nested $V^s \subseteq V^t$, and $V[\bX] \defeq \pb{ V^t[\Xb] : t \in \R }$ is a nested sequence of topological spaces and defines the filtration built using the offsets of $\Xb$. 

The second approach to constructing a filtration is using a filter function $f_{\Xb}: \R^d \rightarrow \R$ which carries the topological information underlying $\Xb$. In this scenario, one typically constructs the filtration from the sublevel sets associated with $f_\Xb$, given by $V^t = f\inv_{\Xb}\qty\big({ (-\infty, t] })$ for each resolution $t$. Again, for $s<t$, {$V^s[f_\Xb] \subseteq V^t[f_\Xb]$} and the sequence $V[f_\Xb] = \pb{ V^t[f_{\Xb}] : t \in \R }$ constitutes the sublevel filtration from $f_{\Xb}$. Mutatis mutandis a similar notion holds for the superlevel filtration.

In general, the filtration $V[\Xb]$ can be very different from $V[f_{\Xb}]$, although the prevailing objective is for $V[f_{\Xb}]$ to encode the same information as in $V[\Xb]$. In this context, the distance function $\dx$ plays a special role owing to the fact that its sublevel filtration is the same filtration associated with the offsets, i.e., $V[\dx] = V[\bX]$. This fact plays an important role in motivating the \md{} estimator introduced in Section~\ref{sec:proposal}, and follows by noting that for every resolution $t > 0$,
\eq{
\dsf_{\Xb}\inv\qty\Big({ (-\infty, t] }) = \pb{\xv \in \R^d : \dx(\xv) \le t} = \bigcup_{\xv \in \Xb}B(\xv, t).\nn
}
Let $V = \pb{V^t : t \in \R}$ denote a generic filtration and let $\iota_{s}^t: V^s \hookrightarrow V^t$ denote the inclusion map between the filtered spaces at resolutions $s < t$. For each resolution $t$, let $\bbv^t = \textup{H}_*\pa{V^t; \bF}$ be the homology\footnote{Where, as per convention, the order of homology, denoted by $*$, is an arbitrary non-negative integer.} of $V^t$ with coefficients in a field $\bF$. As the resolution $t$ varies, the evolution of topological features is captured by $V$. Roughly speaking, new cycles (i.e., connected components, loops, holes, and higher dimensional analogues) are born, or existing cycles {can disappear.} The collection of cycles in $V^t$ at each resolution $t$ is encoded as a vector space in $\bbv^t$. The inclusion maps $\iota_{s}^t: V^s \hookrightarrow V^t$ induce linear maps $\phi_s^t: \bbv^s \rightarrow \bbv^t$ between the vector spaces $\bbv^s$ and $\bbv^t$. 

As such, the collection $V$ can be described more succinctly as the \emph{category} $V = \qty{V^t, \iota_s^t : s \le t}$ with the inclusion maps $\iota_s^t$ \mbox{representing the morphisms for ${s\le t}$}. The image of $V$ under the \emph{homology functor} $\mathbf{Hom}_* : V \mapsto \bbv$, gives us the \emph{persistence module}
\eq{
    \bbv \defeq \qty\Big{\bbv^t, \phi_s^t: {s\le t}},\nonumber
}
where the induced maps $\phi_s^t: \bbv^s \rightarrow \bbv^t$ are homomorphisms between two vector spaces. For $r < s < t$, the persistence module can equivalently be represented as
\begin{figure}[H]
    \includegraphics[]{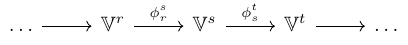}
    \vspace{-10pt}
\end{figure}
Informally, a new topological feature is born at resolution $b \in \R$ if the cycle associated with that feature is not present in $\bbv^{b-\e}$ for all $\e > 0$. The same feature is said to die at resolution $d>b$ if the cycle associated with this feature disappears from $\bbv^{d+\e}$ for all $\e > 0$, resulting in the (ordered) persistence pair $(b,d)$. By collecting all the persistence pairs, the persistence module $\bbv$ may be succinctly represented by a \textit{persistence diagram},
\eq{
    \dgm\pa{\bbv} \defeq \pb{ (b,d) \in \R^2: b \le d \le \infty }.\nn
}

\subsection{Interleaving of Persistence Modules}
\label{sec:interleaving}

Given two persistence modules $\bV = \pb{\bV[][t], \phi_s^t}_{s\le t}$ and $\bW = \pb{\bW[][t], \psi_s^t}_{s \le t}$, they are said to be equivalent (or isomorphic) if there exists a family of linear maps $\pb{\xi_t}_{t \in \R}$ such that each $\xi^t: \bV[][t] \rightarrow \bW[][t]$ is an isomorphism. This notion can be extended to define two collections of maps $\pb{\alpha_t : t \in \R}$ and $\pb{\beta_t : t \in \R}$ which weave the two persistence modules together.

\bigskip

\begin{definition}[Interleaving of persistence modules] Given two persistence modules $\bV$ and $\bW$, and two monotone increasing maps ${\alpha,\beta: \R \rightarrow \R}$, $\bV$ and $\bW$ are said to be $\pa{\alpha,\beta}$--interleaved if the following diagrams commute for all $s \le t$
    \begin{figure}[H]
        \includegraphics[]{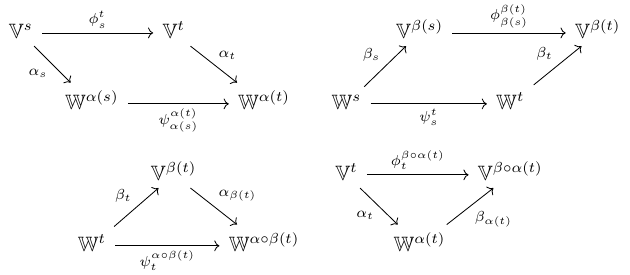}
    \end{figure}
    \label{def:interleaving}
\end{definition}
\begin{remark}
    The persistence modules $\bbv$ and $\bbw$ are purely algebraic objects, and their underlying filtrations $V$ and $W$ are not necessarily compatible. However, when the filtrations $V$ and $W$ arise as filtered subsets of the same underlying space (e.g., $\R^d$), we can similarly define an $(\alpha,\beta)-$interleaving between the filtrations $V$ and $W$ by replacing all linear maps in Definition~\ref{def:interleaving} by inclusion maps.
\end{remark}

The resulting persistence diagrams $\dgm\pa{\bbv}$ and $\dgm\pa{\bbw}$ are elements of the space of persistence diagrams $\Omega = \pb{ (x,y) : x \le y }$ endowed with the family of $q$-Wasserstein metrics $W_q(\cdot, \cdot)$ for $1 \le q \le \infty$. We refer the reader to \cite{edelsbrunner2010computational,mileyko2011probability} for more details. In the special case of $q = \infty$, the resulting metric $\winf$ is commonly referred to as the \textit{bottleneck distance}, and is given as follows. 

\begin{definition}[Bottleneck distance]
    Given two persistence diagrams $D_1, D_2 \in \Omega$, the bottleneck distance is given by
    \eq{
        \winf\pa{D_1, D_2} \defeq \inf_{\gamma \in \Gamma}\sup_{p \in D_1 \cup \Delta} \norminf{ p - \gamma(p) },\nonumber
    }
    where $\Gamma = \pb{\gamma : D_1 \cup \Delta \rightarrow D_2 \cup \Delta}$ is the set of all multi-bijections from $D_1$ to $D_2$ including the diagonal $\Delta = \pb{(x,y) : x=y}$ with infinite multiplicity\footnote{In order to ensure that both persistence diagrams have the same cardinality.}.
\end{definition}

Although the space of persistence diagrams $(\Omega, W_q)$, together with the $q-$Wasserstein distance, presents a challenging mathematical structure for refined statistical analyses \citep{mileyko2011probability,turner2014frechet}, the stability of persistence diagrams \citep{cohen2007stability,chazal2016structure} provides a handle on this space by allowing us to directly work on the space generating the filtrations. 

\begin{lemma}[Stability of persistence diagrams]
    Given two compact sets $\bX, \bY \subset \R^d$,
    $$
        \winf\qty\Big({ \dgm\pa{ \bbv[\Xb] }, \dgm\pa{ \bbv[\Yb]} }) \le \haus{\Xb, \Yb}[].\nn
    $$
    Alternatively, for two filter functions $f,g : \R^d \rightarrow \R$,
    $$
        \winf\qty\Big({ \dgm\pa{ \bbv[f] }, \dgm\pa{ \bbv[g]} }) \le \norminf{f-g}.\nn
    $$
\end{lemma}

\begin{remark}\label{remark:interleaving-1} Given two persistence modules $\bV$ and $\bW$ and their associated persistence diagrams $\dgm\pa{\bV}$ and $\dgm\pa{\bW}$, the following relationships hold:
    \begin{enumerate}[label=\textup{{(\roman*)}}, itemindent=*]
        \item When the interleaving maps $\pa{\alpha,\beta}$ are additive, i.e., of the form ${\alpha: t \mapsto t + \epsilon}$ and ${\beta: t \mapsto t+\delta}$, then persistence diagrams $\dgm\pa{\bV}$ and $\dgm\pa{\bW}$ obtained from the persistence modules  satisfy the following relationships:
            \eq{
                \dgm\pa{\bV} \in \dgm\pa{\bW} \oplus \pc{-\delta, \e}^2 \hspace{1em} \text{ and } \hspace{1em} \dgm\pa{\bW} \in \dgm\pa{\bV} \oplus \pc{-\e, \delta}^2,\nn
            }
            where $\oplus$ denotes the Minkowski sum in $\R^2$. A \emph{coarser} bound is obtained from the stability theorem, which guarantees that
            \eq{
                \Winf\qty\big(\dgm\pa{\bV}, \dgm\pa{\bW}) \le \max\pb{\e,\delta}.\nn
            }
        \item\label{remark:isometry} Furthermore, when the interleaving maps are identical, i.e., $\alpha \equiv \beta: t \mapsto t + \epsilon$, this notion can be extended to define an \textit{interleaving pseudo-distance} between persistence modules,
        \eq{
            d_{\mathcal{I}}\pa{\bV, \bW} \defeq \inf\qty\Big{ \e > 0 : \bV \text{ and } \bW \text{ are } (\alpha,\alpha)\text{--interleaved for } \alpha: t \mapsto t + \e }.\nn
        }
        From the isometry theorem \citep{chazal2016structure} the \emph{interleaving distance} is identical to the \emph{bottleneck distance}, i.e., $\Winf\qty\big(\dgm\pa{\bV}, \dgm\pa{\bW})  = d_{\mathcal{I}}\pa{\bV, \bW}$. Therefore, with a slight abuse of notation, for the proofs in \cref{sec:proofs} we use $\Winf(\bV, \bW)$ to denote the interleaving distance between the persistence modules $\bV$ and $\bW$. In such cases, it is equivalent to say that $\bV$ and $\bW$ are $\pa{\alpha,\alpha}$--interleaved or ${d_{\mathcal{I}}\pa{\bV, \bW} \le \e}$. Similarly, for filtrations $V$ and $W$ comprising of subsets of $\R^d$,
        \eq{
            d_{\mathcal{I}}\pa{V, W} \defeq \inf\qty\Big{ \e > 0 : V^t \subseteq W^{t+\e} \quad \text{and} \quad W^t \subseteq V^{t+\e} }.
            \label{eq:interleaving-filtration}
        }
        Moreover, $d_{\mathcal{I}}\pa{V, W} \le \e \Longrightarrow d_{\mathcal{I}}\pa{\bV, \bW} \le \e \Longrightarrow \Winf\qty\big(\dgm\pa{\bV}, \dgm\pa{\bW}) \le \e$.
    \end{enumerate}
\end{remark}

\subsection{Weighted Rips Filtrations}
\label{sec:weightedrips}

In practice, given a compact set $\Xb \subset \R^d$ or a filter function $f$, the persistence modules $\bbv[\Xb]$ and $\bbv[f]$ are computed using simplicial complexes. In particular:
\begin{enumerate}[label=(\roman*), itemindent=*]
    \item For each $t \in \R$, one may use the \cech{} or Alpha complex to compute the nerve of the cover, $\text{nerve}\pb{ B(\xv, t) : \xv \in \Xb }$. Since the Nerve lemma \citep{edelsbrunner2010computational} guarantees that ${V^t[\Xb] \cong \text{nerve}\pb{ B(\xv, t) : \xv \in \Xb }}$, the resulting persistence module $\bbv[\bX]$ may be computed exactly using simplicial homology. 
    \item In the case of $\bbv\pc{f}$, this is typically achieved by choosing a grid resolution parameter $\e$, and constructing a cubical complex $\k_\e$ on the underlying space. The function $f:\R^d \rightarrow \R$ may be extended to define $f: \k_\e \rightarrow \R$, and at each resolution $t \in \R$, the sublevel sets $V^t[f_\Xb]$ can be approximated using the lower-star filtration $\k_\e^t  = \pb{ \sigma \in \k_\e : \max_{\xv \in \sigma}f(\xv) \le t }$. Therefore, the filtration $\bbv[f]$ can be approximated by the filtration $\pb{\k_\e^t : t \in \R}$, and the resulting persistence module is computed using cubical homology. 
\end{enumerate}

Note that (i) is able to compute the exact persistence module in practice, but is unable to weight points according to $f$. On the other hand, (ii) is only an approximate computation and depends on the nuisance parameter $\e$. Furthermore, the size of the cubical complex is $\abs{\k_\e} = O(\e^{-d})$, making it scale poorly in high dimensions. To overcome this limitation, \cite{buchet2016efficient} proposed the $f$-weighted filtrations, which was subsequently generalized by \cite{anai2019dtm}. 

Given a non-negative \emph{weight function} $f: \R^d \rightarrow \R_{\ge 0}$ and \textit{power} $1 \le p \le \infty$, the \emph{weighted radius function}  of resolution $t>0$ at $\xv$ is given by
\begin{equation}
    \rfx \defeq \begin{cases}
        \pa{t^p - f(\xv)^p}^{1/p} & \text{ if } t \ge f(\xv) \\
        -\infty                   & \text{ if } t < f(\xv).
    \end{cases}\nonumber
\end{equation}
Consequently, $\Bfx[f,\rho][][][]$ is the \emph{weighted ball of resolution $t$ at $\xv$} w.r.t.~the metric $\rho$, which is illustrated in Figure~\ref{fig:ball}, and is given by
\eq{
    \Bfx[f,\rho][][][] \defeq B_{\rho}\pa{\xv, \rfx} = \pb{\yv \in \R^d: \rho(\xv,\yv) \le \rfx}.\nn
}

\begin{figure}
    \centering
    \begin{subfigure}[b]{0.32\textwidth}
        \includegraphics[width=\textwidth]{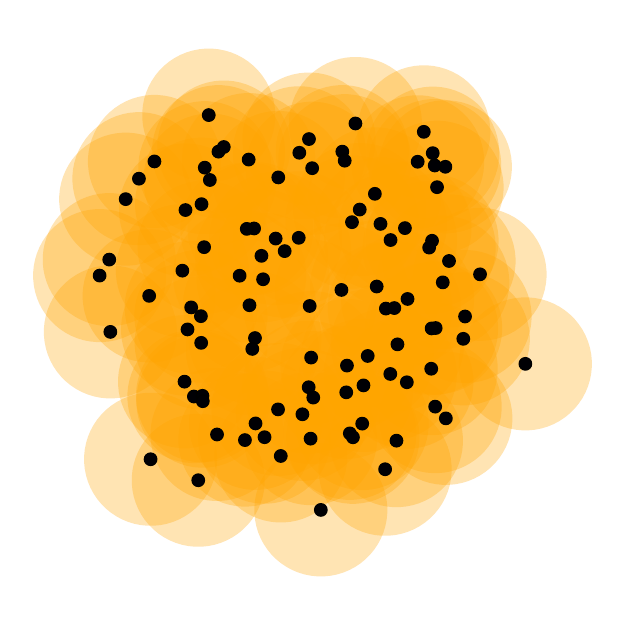}
        \caption{$V^t[\Xn]$ unweighted}
    \end{subfigure}
    \begin{subfigure}[b]{0.32\textwidth}
        \includegraphics[width=\textwidth]{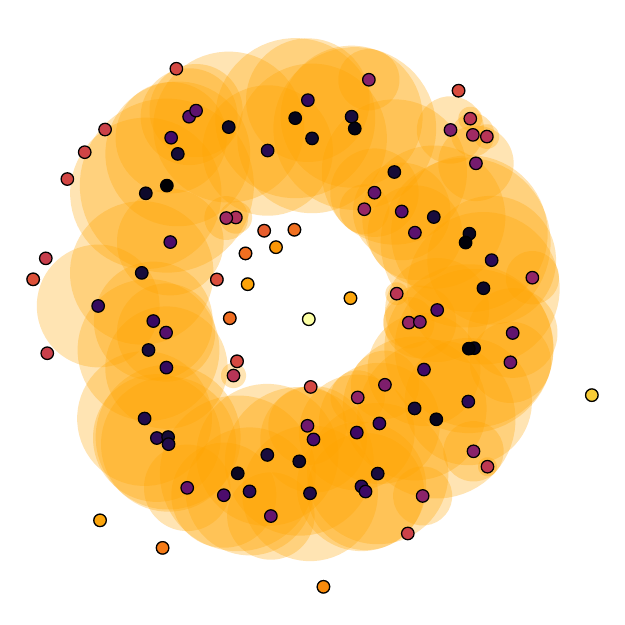}
        \caption{$V^t[\Xn, f]$ for $p=1$}
    \end{subfigure}
    \begin{subfigure}[b]{0.32\textwidth}
        \includegraphics[width=\textwidth]{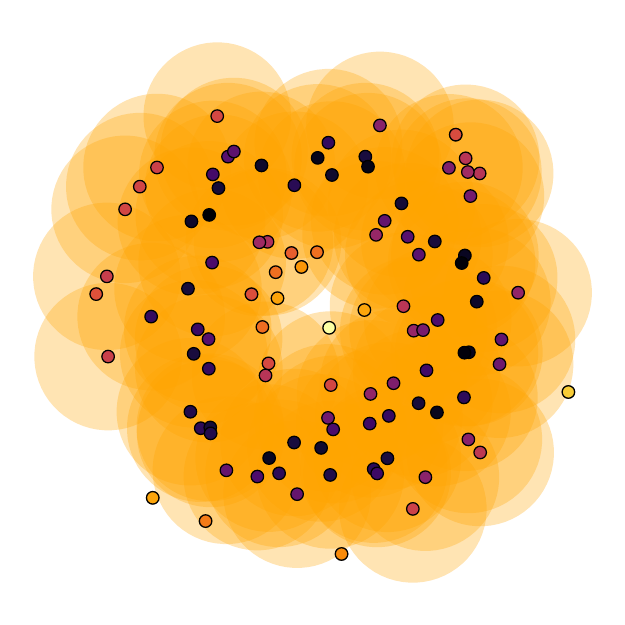}
        \caption{$V^t[\Xn, f]$ for $p=\infty$}
    \end{subfigure}
    \caption{Illustration of offsets for $t=0.5$ and $f(\xv) = \inf_{\yv \in \mathbb{S}^1}\norm{\xv-\yv}$.}
    \label{fig:ball}
\end{figure}
Given $\bX \subseteq \R^d$, the collection of weighted balls ${\VVt[][\bX, f][] = \pB{\Bfx[f][\xv][]: \xv \in \bX}}$, is called the \emph{weighted cover} of $\Xn$. The $f$-weighted offset at resolution $t$ is given by the union of balls in $\VVt[][\bX, f][]$,
\eq{
    \Vt[][\bX,f][] \defeq \bigcup_{\xv \in \bX} \Bfx[f][\xv][].\nonumber
}
Together with the inclusion maps $\iota_s^t:\Vt[s][\bX,f][] \hookrightarrow \Vt[][\bX,f][]$, the $f-$\emph{weighted filtration} is given by 
\eq{
    V[\bX,f] \defeq \pB{\Vt[][\bX,f][], \iota_s^t : {s \le t}}.\nonumber
}
The image of $V[\bX,f]$ under the \emph{homology functor} $\mathbf{Hom}_* : V[\bX,f] \mapsto \bV[][][\bX,f]$, results in the \emph{weighted persistence module} $\bV[][][\bX,f] \defeq \qty{\bVt[][\bX,f][], \phi_s^t : {s\le t}}$, where the induced maps ${\phi_s^t: \bVt[s][\bX,f][] \rightarrow \bVt[][\bX,f][]}$ are linear maps between vector spaces. The weighted-simplicial complexes
\eq{
    \Ct[][\bX,f][] = \text{nerve}\pb{\VVt[][\bX,f][]} \hspace{2mm} \text{ and } \hspace{2mm} \Rt[][\bX,f][] = \text{Rips}\pb{\VVt[][\bX,f][]}\nn
}
denote the weighted-\cech{} complex and weighted-Rips complex associated with the weighted cover $\VVt[][][]$ respectively. Without loss of generality $\bVt[][\bX,f][] = \textup{H}_*\pa{\Vt[][\bX,f][]}$ is the homology of the offset $\Vt[][\bX,f][]$, which, by the nerve lemma, is the same as the homology of the weighted-\cech{} complex. Furthermore, if $f(\xv) \equiv 0$ for all $\xv \in \R^d$ then the resulting filtrations are the usual unweighted filtrations. In particular, $V[\Xn] \cong \Ct[ ][\Xn,f][ ]$ and $\Rt[ ][\Xn, f][ ]$ correspond to \cech{ and Rips} filtrations, respectively. The following structural results appear in \cite{anai2019dtm}, and serve as analogues of the stability result for $f$-weighted filtrations.

\begin{lemma}[{\citealp[Propositions 3.2 \& 3.3]{anai2019dtm}}]
    Given $\bX \subset \R^d$ and $f,g : \bX \rightarrow \R$
    \begin{enumerate}[label=\textup{(\roman*)}]
        \item $\bbv[\bX,f]$ and $\bbv[\bX, g]$ are $(\alpha,\alpha)$--interleaved for $\alpha: t \mapsto t + \norminf{f-g}$.
    \end{enumerate}
    Additionally, given $\bY \subset \R^d$ and $h: \bX \cup \bY \rightarrow \R_+$, if $h$ is $L$--Lipschitz and $\haus{\bX,\bY} \le \e$, then
    \begin{enumerate}[label=\textup{(\roman*)}, resume]
        \item $\bbv[\bX,h]$ and $\bbv[\bY, h]$ are $(\beta,\beta)$--interleaved for $\beta: t \mapsto t + \e \pa{1+L^p}^{1/p}$.
    \end{enumerate}
\end{lemma}



\begin{table}[h!]
    \caption{{\small Glossary of Notation}}\label{tab:glossary}
    \vspace*{-1em}
    \resizebox{0.99\columnwidth}{!}{
    \begin{tabular}{ll}
        \toprule[2pt]
        \textsc{Notation} & \textsc{Description}\\
        \midrule
        $\mathsf{H}_\rho(\bX, \bY)$ & Hausdorff distance between $\bX \subseteq \M$ and $\bY \subseteq \M$ measured w.r.t. metric $\rho$.\\
        $V^t[f]$ & Sublevel set of $f$ at level $t$ given by $\pb{\xv \in \R^d : f(\xv) \le t}$\\
        $V[f]$ and $\bbv[f]$ & Sublevel filtration $\pb{V^t[f]: t \in \R}$ and its persistence module $\textbf{Hom}(V[f])$\\
        $\rfx$ & The $f$--weighted radius function of resolution $t$ at $\xv$. $\rfx=\pa{t^p - f(\xv)^p}^{\f1p}$\\
        $\Bfx[f,\rho][\xv][t]$ & $f$--weighted ball at $\xv$ with radius $\rfx$ w.r.t the metric $\rho$.\\
        $\Vt[][][]$ & The $f$--weighted offset of $\bX$ at resolution $t$ given by $\Vt[][][] = \bigcup_{\xv\in\bX}\Bfx[f,\rho][\xv][t]$\\
        $\Vt[ ][][]$ & $f$--weighted filtration $\qty{\Vt : t \in \R}$ s.t. $\Vt[s][][] \subseteq \Vt$ for all $s \le t$.\\
        $\bVt[ ][][]$ & $f$--weighted persistence module, i.e., $\bVt[ ][][] = \qty{\bbv^t[f], \phi_s^t: s \le t}$ for linear maps $\phi_s^t$.\\
        $\dgm\pa{\bbv}$ & Persistence diagram associated with the persistence module $\bbv$\\
        $\hat\theta_{n,Q}$ & MoM-estimator, $\median\{\hat\theta_{1},\dots, \hat\theta_Q\}$, where $\hat\theta_q$ is the estimator from block $S_q$.\\
        $\dnq$ & \md{} function given by $\dnq(\xv) = \median\pb{\inf_{\yv \in S_q}\norm{\xv-\yv} : q \in [Q]}$\\
        $\dx$ & Distance function to a compact set $\Xb$ given by $\dx(\yv) = \inf_{\xv \in \Xb}\norm{\xv-\yv}$\\
        \bottomrule[2pt]
    \end{tabular}
    }
\end{table}



\section{Proofs}
\label{sec:proofs}

\textbf{Notation.}\quad We adopt the notations introduced in the main text. A glossary of notation is provided in Table~\ref{tab:glossary} for the reader's convenience. In the proofs, we use $\id: b \mapsto b$ to denote the identity map, and we also use $a \lesssim b$, equivalently $b \gtrsim a$, if there exists a constant $C>0$ such that $a \le Cb$. Throughout, we use $C, c, C_1, C_2, \dots$ to denote constants whose values may change from line to line, but are independent of the other parameters of the problem. With a slight abuse of notation and in the interest of brevity, we use $\Winf(\bV, \bW)$ to denote $\Winf(\dgm(\bV), \dgm(\bW))$ which is equivalent to the interleaving distance between the persistence modules $\bV$ and $\bW$ from Remark~\ref{remark:interleaving-1}\,\ref{remark:isometry}.


\allowdisplaybreaks

\providecommand{\ut}[1]{U^{#1}}
\providecommand{\vt}[1]{V^{#1}}
\providecommand{\wt}[1]{W^{#1}}
\providecommand{\but}[1]{\mathbb{U}^{#1}}
\providecommand{\bvt}[1]{\mathbb{V}^{#1}}
\providecommand{\bwt}[1]{\mathbb{W}^{#1}}
\providecommand{\ball}[1]{B_{f\!, \rho}\pa{#1}}
\providecommand{\xvy}{\xv^*_{\yv}}
\renewcommand{\tv}{\mathsf{TV}}


\subsection{Proof for Theorem~\ref{thm:minimax}}
\label{proof:thm:minimax}
\newcommand{\simiid}{\overset{\text{iid}}{\sim}}
\newcommand{\risk}{\mathfrak{R}}
\renewcommand{\P}{\mathcal{P}}
\newcommand{\Dist}{\mathcal{M}}
\renewcommand{\tv}{\mathsf{TV}}

\textbf{Notation.}\quad First, we introduce some notation in order to simplify the presentation of the proof. Fix $\P = \P(\Xb, a, b)$ and $\pi\defeq m/n$, and for samples $\Xn = \{\Xv_1, \dots, \Xv_n\}$, let $\pr_n$ denote its empirical distribution. Let $\Dist^{HC}(\P, \pi)$ denote the set of all empirical distributions satisfying the Huber contamination model (\eref{eq:huber} in \cref{sec:auxiliary}) with fraction $\pi$, i.e., $\pr_n \in \Dist^{\text{HC}}(\P, \pi)$ if there exists $\pr \in \P$ and a probability distribution $\qr$ such that
\begin{align}
    \qty{\Xv_i: i \in [n]} \simiid \pr_{\pi, \qr} \qq{where} \pr_{\pi, \qr} = (1-\pi)\pr + \pi\qr.\label{eq:huber-sampling}
\end{align}
Following \citet[Section~2.2]{bateni2020confidence}, let $\Dist^{\text{HDC}}(\P, \pi)$ denote the set of all empirical distributions satisfying the \textit{Huber deterministic contamination} model, where ${\pr_n \in \Dist^{\text{HDC}}(\P, \pi)}$ if there exists $\pr \in \P$, a distribution $\qr$, and a set $\mathcal{O} \subset [n]$ such that
\begin{align}
    \abs{\mathcal{O}} = n\pi\qc{} 
    \Xnm = \qty{\Xv_i: i \in \mathcal{O}^c} \simiid \pr,\qq{and} 
    \Ym = \qty{\Xv_i: i \in \mathcal{O}} \simiid \qr.\label{eq:huber-deterministic-sampling}
\end{align}
Finally, let $\Dist^{\,\scr{S}}(\P, \pi)$ denote the set of all empirical distributions satisfying sampling setting \ref{setting}, where $\pr_n \in \Dist^{\,\scr{S}}(\P, \pi)$ if there exists $\pr \in \P$ and a set $\mathcal{O} \subset [n]$ such that
\begin{align}
    \abs{\mathcal{O}} = n\pi\qc{} 
    \Xnm = \qty{\Xv_i: i \in \mathcal{O}^c} \simiid \pr,\qq{and} 
    \Ym= \qty{\Xv_i: i \in \mathcal{O}} \sim \qr_{\mathcal{O}},\label{eq:sampling}
\end{align}
where $\qr_{\mathcal{O}}$ is an arbitrary joint distribution on $\Ym$. Let $\widehat\D_n$ be an estimator of the persistence diagram based on the sample $\Xn$, and let $\risk\qty\big(\widehat{\D}_n, \Dist(\P, \pi))$ be the worst-case risk under the contamination model $\Dist(\P, \pi)$, given by
$$
\risk\qty\big(\widehat\D_n, \Dist(\P, \pi)) \defeq \sup_{\pr_n \in \Dist(\P, \pi)} \E_{\pr} \qty[\winf\qty(  \widehat\D_n, \dgm(\bbv[\Xb]) )].
$$
Note that $\risk_{n, m}(\mathcal{P}) = \inf_{\widehat{\D}_n}\risk\qty\big(\widehat\D_n, \Dist^{\scr{S}}(\P, \pi))$. The proof of Theorem~\ref{thm:minimax} requires the following bound on the minimax risk under the Huber contamination model.
\begin{lemma}\label{lem:lower-bound}
    For $\pi \in (0, 1)$ and $\P = \P(\bX, a, b)$, let $\Dist^{\text{HC}}(\P, \pi)$ be the Huber contamination model in \eref{eq:huber-sampling}. Then, there exist absolute constants $C, c > 0$ such that
    \begin{align}
        \inf_{\widehat\D_n}\sup_{\pr_n \in \Dist^{\text{HC}}(\P, \frac{\pi}{2})} \pr\qty{ \winf\qty(  \widehat\D_n, \dgm(\bbv[\Xb]) ) \ge C \qty(\qty(\frac{\pi}{2-\pi})^{1/b} \vee \qty(\frac{\log{n}}{n})^{1/b}) } \ge c.\label{eq:lower-bound}
    \end{align}
\end{lemma}
The proof of \cref{lem:lower-bound} is deferred to \cref{proof:lem:lower-bound}. We now proceed with the proof of Theorem~\ref{thm:minimax}.

\textbf{Proof.}\quad Let $\risk_n^\circ(\P) = \risk_{n, 0}(\P)$ be the minimax lower bound in the absence of any contamination, i.e., when $m=0$. From Theorem~4 and Theorem~5 of \cite{chazal2015convergence},
\begin{align}
    \risk_n^{\circ}(\P) \gtrsim \qty(\frac{\log{n}}{n})^{1/b}.\label{eq:minimax-chazal}
\end{align}
For the claim in \eref{eq:minimax-n}, we want to show that
\begin{align}
    \risk_{n,m}(\P) \gtrsim \qty(\frac{\pi}{2-\pi})^{1/b} \vee \qty(\frac{\log{n}}{n})^{1/b} =: r_n(\pi).\label{eq:minimax-0}
\end{align}
We first note that it suffices to show \eref{eq:minimax-0} holds for all $\pi > \log{n}/n$. Indeed, if $\pi \le \log{n}/n$, then 
\begin{align}
    \frac{\pi}{2-\pi} \le \frac{\log{n}}{2n- \log{n}} \le \frac{\log{n}}{n}.\label{eq:pi-logn}
\end{align}
Moreover, since the sampling setting \ref{setting} encompasses the scenario where the samples $\Xn$ are obtained i.i.d. from $\pr \in \P$ (i.e., $\qty{\Xv_i: i \in \mathcal{O}} \sim \pr$), when $\pi \le \log{n}/n$, the bound in \eref{eq:minimax-0} follows trivially from \eref{eq:minimax-chazal} and \eref{eq:pi-logn}, i.e.,
\begin{align}
    \risk_{n,m}(\P) \gtrsim \risk_n^\circ(\P) \gtrsim \qty(\frac{\log{n}}{n})^{1/b} = r_n(\pi).\label{eq:minimax-small-pi}
\end{align}
Therefore, we assume $\log{n}/n < \pi < 1/2$ in the remainder of the proof.  

From \eref{eq:huber-deterministic-sampling} and \eref{eq:sampling} it is clear that $\Dist^{\text{HDC}}(\P, \pi) \subseteq \Dist^{\,\scr{S}}(\P, \pi)$, and
$$
\risk_{n, m}(\mathcal{P}) = \inf_{\widehat\D_n}\risk\qty\big(\widehat\D_n, \Dist^{\,\scr{S}}(\P, \pi)) \ge \inf_{\widehat\D_n}\risk\qty\big(\widehat\D_n, \Dist^{\text{HDC}}(\P, \pi)).
$$ 
Furthermore, from Proposition~1 of \cite{bateni2020confidence}, for all $r > 0$ we have
\begin{align}
    \inf_{\widehat\D_n}\risk\qty\big(\widehat\D_n, \Dist^{\text{HDC}}(\P, \pi)) \ge \qty(\inf_{\widehat\D_n}\sup_{{\pr_n\in\!\Dist^{\text{HC}}(\P, \frac{\pi}{2})}} r\pr\qty{ \winf\qty(  \widehat\D_n, \dgm(\bbv[\Xb]) ) > r }) - re^{-{n\pi}/{6}},\label{eq:minimax-1}
\end{align}
Note that first term on the right hand side of \eref{eq:minimax-1} is lower bounded by Lemma~\ref{lem:lower-bound}. Therefore, setting $r = C r_n(\pi)$ in \eref{eq:minimax-1} we get
\begin{align}
    \risk_{n,m}(\P) \ge Cr_n(\pi)\qty(c - e^{-n\pi/6}) \ge Cr_n(\pi)\qty(c - e^{-\log{n}/6}),\label{eq:minimax-2}
\end{align}
where the final inequality follows by noting that $e^{-n\pi/6} \le e^{-\log{n}/6}$ for $\pi > \log{n}/n$. For sufficiently large\footnote{It suffices for $n \ge (2/c)^{6}$.}~$n$, it follows that
\begin{align}
    \risk_{n,m}(\P) \ge \frac{C}{2} \cdot r_n(\pi) = \frac{C}{2} \cdot \qty(\frac{\pi}{2-\pi})^{1/b} \vee \qty(\frac{\log{n}}{n})^{1/b}.\label{eq:minimax-large-pi}
\end{align}
Combining the bound in \eref{eq:minimax-small-pi} for $\pi \le \log{n}/n$ and the bound in \eref{eq:minimax-large-pi} for $\pi > \log{n}/n$ gives the desired lower bound in \eref{eq:minimax-0}.

For the bound in \eref{eq:minimax-rate}, note that when $m = cn^\e$ for $0 < \e < 1$ and $c > 0$, we have $2n - m \gtrsim n$, from which it follows that $m/(2n-m) \lesssim (1/n)^{1-\e}$. Therefore, for $\e > 0$ we have $(1/n)^{1-\e} \gg \log{n}/n$, and
\begin{align}
    \risk_{n,m}(\P) \gtrsim \qty(\frac{1}{n^{1-\epsilon}})^{1/b}.\label{eq:minimax-rate-1}
\end{align}
This completes the proof of Theorem~\ref{thm:minimax}. \QED

\subsubsection{Proof of Lemma~\ref{lem:lower-bound}}
\label{proof:lem:lower-bound}

\begin{proof}
    From Theorem~5.1 of \cite{chen2018robust} (see \cref{thm:chen} in \cref{sec:auxiliary}), there exist absolute constants $C> 0$ and $c \in (0, 1)$ such that
    \begin{align}
        \inf_{\widehat\D_n}\sup_{\pr_n \in \Dist^{\text{HC}}(\P, \frac{\pi}{2})} \pr\qty{ \winf\qty(  \widehat\D_n, \dgm(\bbv[\Xb]) ) \gtrsim r_n(0) \vee \omega\qty\big(\pi/2, \P, \Winf) } \ge c,
    \end{align}
    where $r_n(0)$ is the minimax rate of convergence in the absence of any contamination, and $\omega(\pi/2, \P, \Winf)$ is the total-variation modulus of continuity of the bottleneck distance $\Winf$, given~by
    \eq{
        \omega\qty\big(\pi/2, \P, \Winf) = \sup\qty{ \Winf( \dgm(\Xb_1), \dgm(\Xb_2) ) : \tv(\pr_1, \pr_2) \le \frac{{\pi}/{2}}{1 - {\pi}/{2}} },\nn
    }
    for $\Xb_i = \supp(\pr_i)$. See, also, \cref{def:modulus-of-continuity}. From \eref{eq:minimax-chazal} and by standard protocol for reduction of the minimax risk using Markov's inequality (see, e.g., Chapter~2.2 of \citealp{tsybakov2008nonparametric}), we have ${r_n(0) = \risk_n^\circ(\P) \gtrsim (\log{n}/n)^{1/b}}$, and it remains to establish a lower bound for $\omega(\pi/2, \P, \Winf)$. 
    
    To this end, for $b \le d$ let $B_b(r)$ denote the $b$-dimensional closed ball of radius $r$ embedded in $\R^d$ and centered at the origin, i.e.,
    $$
    B_b(0, r) = \qty{ \xv = (\yv, \zerov_{d-b}) \in \R^d: \norm{\yv} \le r}.
    $$ 
    Further, let $\pr_1 = \textup{Unif}(B_b(1))$ and $\pr_2 = \textup{Unif}( B_b(1) \setminus B_b(r) )$ for some $r \in (0, 1)$, i.e., $\pr_2$ is the uniform distribution on the spherical shell. The total variation distance between $\pr_1$ and $\pr_2$ is given by
    \eq{
        \tv(\pr_1, \pr_2) 
        &= \int_{B_b(1)} \abs{ \frac{\mathbf{1}\qty{\xv \in B_b(1)}}{\text{vol}(B_b(1))} -  \frac{\mathbf{1}\qty{\xv \in B_b(1)\setminus B_b(r)}}{\text{vol}\qty(B_b(1)\setminus B_b(r))}} d\xv \nn\\
        &= r^b \cdot \abs\Big{1 - 0} + (1-r^b) \cdot \abs{\frac{1}{1-r^b} - 1} = r^b.\nn
    }
    Moreover, by construction, $\Xb_2 = B_b(1) \setminus B_b(0)$ has a $b$-dimensional hole of radius $r$ in the center whereas $\Xb_1 = B_b(1)$ has trivial order-$b$ homology. In other words, the $b$-th order persistence diagram $\dgm_b(\Xb_1)$ is empty and $\dgm_b(\Xb_2)$ has a single point at $(0, r)$, and, 
    $$
    \Winf(\dgm_b(\Xb_1), \dgm_b(\Xb_2)) = r.
    $$ 
    Therefore,
    \eq{
        \omega\qty\big(\pi/2, \P, \Winf) = \sup\qty{ r: r^b \le \frac{\pi/2}{1-\pi/2}} = \qty(\frac{\pi/2}{1 - \pi/2})^{1/b} = \qty( \frac{\pi}{2-\pi} )^{1/b}.\nn
    }
\end{proof}


\subsection{Proof for Lemma~\ref{lemma:lipschitz}}
\label{proof:lemma:lipschitz}

We begin by noting that for each $q \in [Q]$, the distance function $\dsf_{n, S_q}$ associated with the block $S_q$ is $1-$Lipschitz \cite[Chapter~9.1]{boissonnat2018geometric}. Thus, for each $q \in [Q]$ and for all $\xv, \yv \in \R^d$ we have that
\eq{
    0 \le \dsf_{n,q}(\xv) \le \dsf_{n,q}(\yv) + \norm{\xv-\yv},\nn
}
and, therefore, it follows that
\eq{
    \median\pb{ \dsf_{n,q}(\xv) : q \in [Q] } \le \median\pb{ \dsf_{n,q}(\yv) : q \in [Q] } + \norm{\xv-\yv}.\nn
}
As a result, we obtain that $\dnq(\xv) \le \dnq(\yv) + \norm{\xv-\yv}$. Exchanging $\xv$ and $\yv$ in the steps above yields the desired result. \QED


\renewcommand{\rqt}{\xi_{q}({2}t; n, Q)}

\subsection{Proof of Theorem~\ref{theorem:momdist-sublevel}}
\label{proof:theorem:momdist-sublevel}

For the sake of brevity, we use $\dgm\pa{f}$ to denote $\dgm(\bV[][][f])$. First, we note from the stability of persistence diagrams that,
\eq{
    \pr\qty\bigg{\Winf\qty\Big({\dgm\pa{\dnq}, \dgm\pa{\dx}}) > {2}t} \le \pr\pb{\norminf{\dnq - \dx} > {2}t}.
    \label{mom-concentration-stability}
}
Therefore, it suffices to control the probability of the event $\pb{\norminf{\dnq - \dx} > {2}t}$. To this end, let $A = \pb{q\in [Q]: S_q \cap \Ym = \varnothing}$ be the blocks which contain no outliers. From the assumption on $Q$, i.e., $2m < Q < n$, it follows that, and $\abs{A}  > Q/2$. For $q \in [Q]$, let $\rqt$ be given by
\eq{
    \rqt = \mathds{1}\qty\Big( { \norminf{\mathsf{d}_{n, q} - \dx} } > {2}t ).\nn
}
On application of Lemma~\ref{lemma:mom} to the estimator $\dnq$, it follows that
\eq{
    \pr\qty\bigg{{ \norminf{\mathsf{d}_{n, Q} - \dx} } > {2}t} \le \pr\pa{\sum_{q \in A}\rqt > \f Q 2 - m}.
    \label{eq:momdist-sublevel1}
}

Since $S_q \subseteq \Xnm$ for all $q \in A$, it follows that $\qty{\rqt: q \in A}$ are i.i.d.~$\textup{Ber}\qty\big(p({2}t; n, Q))$ random variables, where 
\eq{
    p({2}t; n, Q) = \E\qty(\rqt) = \pr\qty\Big( { \norminf{\mathsf{d}_{n, q} - \dx} } > {2}t ).\nn
}
For the remainder of the proof we need two key ingredients: (i) we need an upper bound for  $\E(\rqt)$, and (ii) we need a tight bound for the binomial tail probability in \cref{eq:momdist-sublevel1}.

\textbf{Bound for $p({2}t; n, Q)$.} From \citet[Theorem~2]{chazal2015convergence}, under the $(a, b)-$standard condition it follows that
\eq{
    p({2}t; n, Q) \le \f{2^b}{at^b}\exp({ -\nq at^b }) = \exp({ -\nq at^b - \log\qty\big(at^b) + b\log2 }).
    \label{eq:momdist-pt}
}

\textbf{Binomial tail probability bound.} For $0 < \e < 1$, using the Chernoff-Hoeffding bound from Lemma~\ref{lemma:chernoff-hoeffding} yields,
\eq{
    \pr\pa{\f{1}{\abs{A}}\sum_{q \in A}\rqt > \e} \le \exp\qty\Bigg( \abs{A} \pa{\f{2}{e} + \e \log p({2}t; n, Q)}).\nn
}
Using the bound for $p({2}t; n, Q)$ from \cref{eq:momdist-pt}, we obtain
\eq{
    \pr\pa{\f{1}{\abs{A}}\sum_{q \in A}\rqt > \e} &\le \exp\qty\Bigg( \abs{A} \qty\Big( \f{2}{e} + b\e\log2 - \e \nq at^b - \e \log\!\qty\big(at^b)) )\n
    &\le \exp\qty\Bigg( \abs{A} \qty\Big( 1 + b\e - \e \nq at^b - \e \log\!\qty\big(at^b)) )\n
    &\le \exp\qty\bigg( \abs{A} \qty\Big( 1 + b\e - \e\Ot) ),\nn
}
where, in the last line we use $\Ot \defeq (n/Q)at^b + \log(at^b)$ for brevity. When $t$ satisfies the condition that
\eq{
    \Ot \ge \f{2(1+b\e)}{\e},
    \label{eq:Ot-condition1}
}
then it implies that
\eq{
    1 + b\e - \e\Ot \le -\f\e2 \Ot,\nonumber
}
and we get
\eq{
    \pr\pa{\f{1}{\abs{A}}\sum_{q \in A}\rqt > \e} \le \exp\qty\Bigg( -\f{\abs{A}\e}{2} \Ot ).\nn
}

By setting $\delta$ equal to the r.h.s. of the inequality above, we obtain
\eq{
    \Ot = \f{2\log(1/\delta)}{\abs{A}\e}.
    \label{eq:Ot-Ae}
}
When $\delta \le e^{-(1+b)Q}$, using the fact that $Q > \abs{A}$ and $0 < \e < 1$, it follows that
\eq{
    \Ot = \f{2\log(1/\delta)}{\abs{A}\e} \ge \f{2(1+b)Q}{\abs{A}\e} \ge \f{2(1+b\e)}{\e},\nn
}
and, therefore, the condition in \cref{eq:Ot-condition1} is satisfied. Consequently, for $\delta \le e^{-(1+b)Q}$, on rearranging the terms in \cref{eq:Ot-Ae} we obtain 
\eq{
    \pr\pa{\sum_{q \in A}\rqt > \f{2\log(1/\delta)}{\Ot}} \le \delta. 
    \label{eq:momdist-sublevel-bound1}
}

Comparing \cref{eq:momdist-sublevel1} with \cref{eq:momdist-sublevel-bound1} we conclude that
\eq{
    \pr\pa{\sum_{q \in A}\rqt > \f{Q-2m}{2}} = \pr\pa{\sum_{q \in A}\rqt > \f{2\log(1/\delta)}{\Ot}} \le \delta,\nn
}
by setting
\eq{
    \f{2\log(1/\delta)}{\Ot} = \f{Q-2m}{2} \Longleftrightarrow \Ot = \f{4{\log(1/\delta)}}{Q-2m}.\nn
}
Since $\Ot = \nq at^b + \log(at^b)$, this is equivalent to
\eq{
    \exp( \nq a t^b ) \nq at^b = \nq \exp({ \f{4{\log(1/\delta)}}{Q-2m} }).\nn
}
Moreover, using the fact that the Lambert $\wo$ function satisfies $\wo(x)e^{\wo(x)} = x$ \citep{hassani2005approximation}, we obtain that
\eq{
    t = \qty\Bigg( \f{Q}{an} \wo\qty( \nq \exp{ \f{4{\log(1/\delta)}}{Q-2m} })) ^{1/b}.
    \label{eq:momdist-t-constraint}
}
Therefore, from \cref{mom-concentration-stability} and (\ref{eq:momdist-sublevel1}), for $t$ satisfying \cref{eq:momdist-t-constraint} and for all $\tau \ge t$ we have
\eq{
    \pr\qty\bigg{\Winf\qty\Big({\dgm\pa{\dnq}, \dgm\pa{\dx}}) > {2}\tau} \le \pr\qty\bigg{\Winf\qty\Big({\dgm\pa{\dnq}, \dgm\pa{\dx}}) > {2}t} \le \delta.
    \label{eq:momdist-sublevel2}
}
Since $\delta \le e^{-(1+b)Q}$, observe that
\eq{
    \f{4\log(1/\delta)}{Q-2m} \ge \f{4(1+b)Q}{Q-2m} \ge 4(1+b) > 1.\nonumber
}
Furthermore, using the fact that $\wo(z) \le \log(z)$ for $z > e$ \citep{hassani2005approximation}, we may take $\tau$ to be 
\eq{
    t = \qty\Bigg( \f{Q}{an} \wo\qty( \nq \exp{ \f{4{\log(1/\delta)}}{Q-2m} })) ^{1/b} &\le \qty\Bigg( \f{Q}{an} \log\qty( \nq \exp{ \f{4{\log(1/\delta)}}{Q-2m} })) ^{1/b}\n
    &= \qty\Bigg( \f{Q \log(n/Q)}{an} + \f{4Q{\log(1/\delta)}}{a(Q-2m)n} ) ^{1/b} \defeq \tau.\nn
}
Plugging this into \cref{eq:momdist-sublevel2}, we obtain the desired result. 

For the second claim in the theorem, by inverting the relationship between $t$ and $\delta$ in \cref{eq:momdist-t-constraint} and using the fact that $\wo(z)$ is an increasing function for $z > 0$, observe that the constraint on $\delta$ equivalently specifies a constraint on $t$, i.e.,
\eq{
    \delta \le e^{-(1+b)Q} \Longleftrightarrow t \ge \qty\Bigg( \f{Q}{an} \wo\qty( \nq \exp{ \f{4{(1+b)Q}}{Q-2m} })) ^{1/b}.\nn
}
A sufficient condition for this to hold is that
\eq{\label{eq:t-bound}
    t \ge t(n, Q) \defeq \qty\Bigg( \f{Q \log(n/Q)}{an} + \f{4(1+b)Q^2}{a(Q-2m)n} )^{1/b}.\nonumber
}
Therefore, from \cref{eq:momdist-sublevel2} we have that for all $t \ge t(n, Q)$
\eq{
    \pr\qty\bigg{\Winf\qty\Big({\dgm\pa{\dnq}, \dgm\pa{\dx}}) > {2}t} \le \exp( - \qty(\f{Q-2m}{4})\Ot ).\nn
}
By taking $\pr\qty\big{\Winf\qty\Big({\dgm\pa{\dnq}, \dgm\pa{\dx}}) > {2}t}$ to be its maximum value of $1$ in the interval $[0, {2}t(n, Q)]$ we have 
\eq{
    \E\qty[ \Winf\qty\Big({\dgm\pa{\dnq}, \dgm\pa{\dx}}) ] &= {2}\int\limits_0^\infty\pr\qty\bigg{\Winf\qty\Big({\dgm\pa{\dnq}, \dgm\pa{\dx}}) > {2}t} dt\n
    &\le {2}t(n, Q) + {2}\int\limits_{{2}t(n, Q)}^\infty  \exp( - \qty(\f{Q-2m}{4})\Ot ) dt.\n
    &\le {2}t(n, Q) + {2}\int\limits_{{2}t(n, Q)}^\infty  \exp( - \f{\Ot}{4} ) dt.\nonumber
}

{
Let $w = \Ot$, i.e., $t = \qty(\frac{Q}{an}\wo(\frac{n}{Q}e^w))^{1/b}$. By noting that $\Omega'(t, n/Q) = b(w+1)/t$, and that $t \ge 2t(n, Q)$ implies that $w = \Ot \ge {8}(1+b)/(Q-2m)$ by \cref{eq:t-bound}, the change of variables from $t$ to $w$ in the integral gives
\eq{
    \int\limits_{t \ge {2}t(n, Q)} \exp( - \f{\Ot}{4} ) dt 
    &\le \int\limits_{w \ge \frac{{8}(1+b)}{(Q-2m)}} e^{-w/4} \cdot \frac{\qty(\frac{Q}{an}\wo(\frac{n}{Q}e^w))^{1/b}}{b(w+1)} dw\n
    &\lesssim \pa{\f{Q}{n}}^{1/b} \int\limits_{r_n}^{\infty} \f{e^{-w/4}\pa{\wo\pa{ \f{n}{Q} e^w }}^{1/b}}{w+1}dw,
}
where $r_n \defeq {8}(1+b)/(Q-2m)$ and the factor $ba^{1/b}$ is absorbed into the symbol $\lesssim$. Therefore, we have that
}
\eq{\label{eq:momdist-ebound1}
    &\E\qty\Big[ \Winf\qty\big({\dgm\pa{\dnq}, \dgm\pa{\dx}}) ] \n
    &\qq{} \qq{} \lesssim t(n, Q) + \pa{\f{Q}{n}}^{1/b} \int\limits_{r_n}^{\infty} \f{e^{-w/4}\pa{\wo\pa{ \f{n}{Q} e^w }}^{1/b}}{w+1}dw\nn\\[5pt]
    &\qq{} \qq{} \num{\lesssim}{ii} t(n, Q)%
    + \qty( \f{\log(n/Q)}{n/Q})^{1/b} \underbrace{\int\limits_{r_n}^\infty\f{e^{-w/4}}{w+1}dw}_{\circled{a}} %
    + \pa{\f{Q}{n}}^{1/b} \underbrace{\int\limits_{r_n}^\infty\f{e^{-w/4}w^{1/b}}{w+1}dw}_{\circled{b}},
}
where (ii) follows from the fact that $\wo(z) \le \log(z)$ for $z > e$ together with, either, an application of Lemma~\ref{lemma:useful-inequalities}~(iii) when $b\ge 1$, or Lemma~\ref{lemma:useful-inequalities}~(i) with the additional factor $2^{1/b - 1}$ being absorbed in the symbol $\lesssim$ when $b<1$. The term $\circled{a}$ can be bounded above using the incomplete $\Gamma$ function as,
\eq{
    \circled{a} = \int\limits_{r_n}^\infty\f{e^{-w/4}}{w+1}dw = e^{1/4} \int_{(r_n+1)/4}^\infty v^{-1}e^{-v}dv = e^{1/4} \Gamma\qty\big(0, (r_n+1)/4) < \infty. \nn
}
Similarly, using the fact that $w+1 > 1$, the term $\circled{b}$ may be bounded above as,

{
\eq{
    \circled{b} = \int\limits_{r_n}^\infty\f{e^{-w/4}w^{1/b}}{w+1}dw \le \int\limits_{r_n}^\infty{e^{-w/4}w^{1/b}}dw \le \int\limits_{0}^\infty{e^{-w/4}w^{1/b}}dw \le \f{\Gamma(1 + b\inv)}{4^{1+1/b}} < \infty,\nn
}
}
Therefore, the inequality in \cref{eq:momdist-ebound1} becomes
\eq{
    \E\qty[ \Winf\qty\Big({\dgm\pa{\dnq}, \dgm\pa{\dx}}) ] \lesssim \qty\bigg(\f{\log(n/Q)}{n/Q} + \f{Q^2}{(Q-2m)n} )^{1/b} + \qty\bigg(\f{Q}{n})^{1/b}.\nn
}
When the number of outliers grows with $n$ as $m_n = cn^\e$ where $0 \le \e < 1$, {let the number of blocks be $Q_n = 3c n^\beta$,} where $\e \le \beta < 1$. Therefore, 
\eq{
    \E\qty[ \Winf\qty\Big({\dgm\pa{\dnq}, \dgm\pa{\dx}}) ] &\lesssim \inf_{\e \le \beta < 1} \qty\bigg(\f{\log(n/n^\beta)}{n/n^\beta} + \f{n^{2\beta}}{(3n^\beta-2n^\e)n} )^{1/b} + \qty\bigg(\f{n^\beta}{n})^{1/b}\n[5pt]
    &\lesssim \qty( \f{\log n}{n^{1-\e}} )^{1/b},\nn
}
which gives us the desired result when $Q_n = 3cn^\e$. \QED


\subsection{Proof of Theorem~\ref{theorem:momdist-stability}}
\label{proof:theorem:momdist-stability}

We begin by establishing the following result:
\eq{
    \Winf\qty\bigg( \bVt[  ][\Xn, \dnq], \bVt[  ][\Xnm, \dnq] ) \le \adjustlimits\sup_{\xv \in \Xnm}\dnq(\xv) + \qty( 1 - \f1p )t(\Xnm).\nn
}
Observe that from Lemma~\ref{lemma:ab-filtration} and Lemma~\ref{lemma:ab-module}, it suffices to show that for every $\yv \in \Ym$ the MoM-Dist function $\dnq$ satisfies the property that
\eq{
    \inf_{\xv \in \Xnm}\norm{\xv - \yv} \le \dnq(\yv).\nn
}
To this end, let $A = \pb{ q \in [Q] : S_q \cap \Ym = \varnothing }$ be the blocks containing no outliers. For $\yv \in \Ym$ and every $q \in A$, we have that $S_q \subseteq \Xnm$, and therefore
\eq{
    \inf_{\xv \in \Xnm} \norm{\xv - \yv} \le \inf_{\xv \in S_q}\norm{\xv - \yv} = \mathsf{d}_{n,q}(\yv).\nn
}
Since this holds for every $q \in A$, taking the infimum on the right-hand side over $A$ yields
\eq{
    \inf_{\xv \in \Xnm} \norm{\xv - \yv} \le \inf_{q \in A} \mathsf{d}_{n,q}(\yv).\nn
}
Since $2m < Q$ by assumption, using the pigeonhole principle we further have that
{
\eq{
    \inf_{q \in A} \mathsf{d}_{n,q}(\yv) \le \med\qty\Big{ \mathsf{d}_{n, q}(\yv) : q \in [Q] },\nn
}
}
which implies that $\inf_{\xv \in \Xnm}\norm{\xv - \yv} \le \dnq(\yv)$ for every $\yv \in \Ym$. Therefore, taking $a=0$ in Lemma~\ref{lemma:ab-filtration} and Lemma~\ref{lemma:ab-module} we obtain
\eq{
    \Winf\qty\bigg( \bVt[  ][\Xn, \dnq], \bVt[  ][\Xnm, \dnq] ) \le \adjustlimits\sup_{\xv \in \Xnm}\dnq(\xv) + \qty( 1 - \f1p )t(\Xnm).
    \label{eq:stability-1}
}

Turning our attention to the quantity appearing in the statement of the theorem, note that an application of the triangle inequality yields
\eq{
    \Winf\qty\bigg( \bVt[  ][\Xn, \dnq], \bVt[  ][\Xnm, \dsf_{n-m}] ) &\le \Winf\qty\bigg( \bVt[  ][\Xn, \dnq], \bVt[  ][\Xnm, \dnq] ) \n
    &\quad\quad+ \Winf\qty\bigg( \bVt[  ][\Xnm, \dnq], \bVt[  ][\Xnm, \dsf_{n-m}] )\n
    &\num{\le}{$\star$} \adjustlimits\sup_{\mathclap{\xv \in \Xnm}}\dnq(\xv) + \qty( 1 - \f1p )t(\Xnm) + \norminf{\dnq - \dsf_{n-m}},\nn
}
where the first term in ($\star$) follows from \cref{eq:stability-1} and the last term follows from Proposition~\ref{lemma:anai-et-al}. This gives us the desired result. Furthermore, when $p=1$ note that $1 - 1/p = 0$, giving us the tighter bound in this case. \QED


\subsection{Proof of Theorem~\ref{theorem:momdist-consistency}}
\label{proof:theorem:momdist-consistency}

We begin by noting that $\bvt{}[\bX] = \bvt{}[\bX, \dx]$. Indeed, the distance function $\dx(\xv) = 0$ for all $\xv \in \bX$. We may further conclude that
\eq{
    \sup_{\xv \in \bX}\dx(\xv) = 0.
    \label{eq:supdist=0}
}
The bottleneck distance between $\bVt[  ][\Xnm \cup \Ym, \dnq]$ and $\bvt{}[\bX]$ may be bounded above as
\eq{
    \Winf\qty\bigg( \bVt[  ][\Xnm \cup \Ym, \dnq], \bvt{}[\bX] ) &\le \Winf\qty\bigg( \bVt[  ][\Xnm \cup \Ym, \dnq], \bVt[  ][\Xnm, \dnq] ) && =:\circled{a} \n
    &\quad+ \Winf\qty\bigg( \bVt[  ][\Xnm, \dnq], \bVt[  ][\Xnm, \dx]) && =:\circled{b} \n
    &\quad+ \Winf\qty\bigg( \bVt[  ][\Xnm, \dx], \bVt[  ][\bX, \dx] ). && =:\circled{c} \nn
}

When $p=1$, the terms $\circled{b} \le \norminf{\dnq - \dx}$ and $\circled{c} \le \mathsf{H}(\Xnm, \bX)$ using Proposition~\ref{lemma:anai-et-al}. The term $\circled{a}$ is bounded above by taking $p=1$ in \cref{eq:stability-1} (from the proof of Theorem~\ref{theorem:momdist-stability}) to give
\eq{
    \circled{a} &= \sup_{\xv \in \Xnm}\dnq(\xv) \n
    &\num{\le}{$\star$} \sup_{\xv \in \bX}\dnq(\xv)\n
    &\num{\le}{$\dagger$} \norminf{\dnq - \dx} + \sup_{\xv \in \bX}\dx(\xv)\n
    &\num{=}{$\ddagger$} \norminf{\dnq - \dx},\nn
}
where ($\star$) follows from the fact $\Xnm \subset \bX$, ($\dagger$) uses the identity $f(\xv) \le \norminf{f - g} + g(\xv)$ for all $\xv \in \bX$, and ($\ddagger$) follows from \cref{eq:supdist=0}. Plugging in the bounds for the bottleneck distance we obtain
\eq{
    \Winf\qty\bigg( \bVt[  ][\Xnm \cup \Ym, \dnq], \bvt{}[\bX] ) \le 2\norminf{\dnq - \dx} + \mathsf{H}(\Xnm, \bX).\nn
}
By noting that the Hausdorff distance $\mathsf{H}(\Xnm, \bX) = \norminf{\dsf_{n-m} - \dx}$, for $t_1, t_2$ such that $t_1 + t_2 = t$ we may bound the tail probability for the bottleneck distance as follows.
\eq{
    &\pr\qty\Bigg{\Winf\qty\Big( \bVt[  ][\Xnm \cup \Ym, \dnq], \bvt{}[\bX] )  > {2}t}\\ 
    &\qquad\le \pr\qty\Big({ 2\norminf{\dnq - \dx} > {2}t_1 }) + \pr\qty\Big({ \norminf{\dsf_{n-m} - \dx} > {2}t_2 })\n
    &\qquad\le \delta_1 + \delta_2 = \delta,
    \label{eq:momdist-filt}
}
where the relationship between $\delta_1, \delta_2$ and $t_1, t_2$ is given by \cref{eq:momdist-t-constraint}, i.e., $\delta_1 \le e^{-(1+b)Q}$ from the condition in Theorem~\ref{theorem:momdist-sublevel}, $\delta_2 = \delta - \delta_1$,
\eq{
    t_1 = 2\qty\Bigg( \f{Q}{an} \wo\qty( \nq \exp{ \f{4{\log(1/\delta_1)}}{Q-2m} })) ^{1/b}, \ \ \text{ and } \ \ \ t_2 = \qty\Bigg( \f{1}{a{n-m}} \wo\qty\Big( {n-m} e^{ 4{\log(1/\delta_2)} })) ^{1/b}.\label{eq:t1_t2}
}
Furthermore, using the bound for the Lambert $\wo$ function ${\wo(z) \le \log z}$ for $z > e$, we have
\eq{
    t_1 \le 2\qty\bigg( \frac{Q\log(n / Q)}{an} + \frac{4Q \log(1/\delta_1)}{a(Q-2m)n} )^{1/b}, \ \ \text{  } \ \ \ t_2 \le \qty\bigg( \frac{\log (n-m)}{a (n-m)} + \frac{4 \log(1/\delta_2)}{a (n-m)} )^{1/b},\nonumber
}
and $t = t_1 + t_2 \le \mathfrak{f}(n, m, Q, \delta_1, \delta_2)$. Therefore, the bound in \cref{eq:momdist-filt} yields
{
    \eq{
    &\pr\Bigg\{\Winf\qty\Big( \bVt[  ][\Xnm \cup \Ym, \dnq], \bvt{}[\bX] )  \le {2}\mathfrak{f}(n, m, Q, a, b)\Bigg\} \n
    &\qquad \ge \pr\qty\Bigg{\Winf\qty\Big( \bVt[  ][\Xnm \cup \Ym, \dnq], \bvt{}[\bX] )  \le {2}t_1 + {2}t_2} \ge 1 - \delta,\nn
}
}
which gives the desired result. The second part of the theorem follows directly using the identical procedure as that used in the proof of Theorem~\ref{theorem:momdist-sublevel} in \cref{proof:theorem:momdist-sublevel}. \QED



\begin{remark}
    The result in Theorem~\ref{theorem:momdist-consistency} holds when $p=1$. For $p \ge 1$, from \cite[Proposition~3.6]{anai2019dtm} it follows that the number of points in the $0$th persistence diagram $\dgm(\bVt[  ][\Xn])$ is non-increasing in $p$, i.e., choosing $p > 1$ leads to sparser persistence diagrams, and has an appeal from a computational perspective. The following result characterizes the error incurred when using $V[\Xn, \dnq]$ to approximate the sublevel filtration $V[\dnq]$ for $p \ge 1$. In light of Remark~\ref{remark:stability}~(ii), the approximation error vanishes with increasing sample size. In contrast, the approximation error for the DTM-filtration is non-vanishing \cite[Proposition~4.6]{anai2019dtm}.
\end{remark}

\begin{proposition}\label{prop:sublevel-2}
    Let $p \ge 1$. For $\Xn = \Xnm \cup \Ym$ under sampling setting \ref{setting} and {${2m < Q < n}$,} the filtrations $V[\dnq]$ and $V[\Xn, \dnq]$ are $(\eta, \xi)-$interleaved, where
    \eq{
        \eta(t) =  2^{\ipfac}t + \sup_{\xv \in \Xnm}\dnq(\xv) \qquad\qq{and}\qquad \xi(t) = 2^{\ipfac}\eta(t).\nn
    }
\end{proposition}

\begin{proof}
We begin by noting from Lemma~\ref{lemma:sublevel-equivalence} that the filtrations $V[\dnq]$ and $V[\R^d, \dnq]$ are $(\id, \alpha)-$interleaved for ${\alpha : t \mapsto \tp t}$. Furthermore, consider the intermediate filtrations $V[\Xnm, \dnq]$. From Theorem~\ref{theorem:momdist-stability} and Lemma~\ref{lemma:ab-filtration} we have that $V[\R^d, \dnq]$ and $V[\Xnm, \dnq]$ are $(\eta, \id)-$ interleaved for 
$$
{\eta: t \mapsto \tp t + \sup_{\xv \in \Xnm}\dnq(\xv)}. 
$$
Using an identical argument, but reversing the order, we have that $V[\Xnm, \dnq]$ and $V[\Xn, \dnq]$ are $(\id,\eta)-$interleaved. We can now apply the ``\textit{triangle inequality}'' for generalized interleavings \cite[Proposition~3.11]{bubenik2015metrics} to obtain that $V[\dnq]$ and $V[\Xn, \dnq]$ are ${(\id \circ \eta \circ \id, \alpha \circ \id \circ \eta)-}$interleaved. On simplifying the interleaving maps, we obtain that for
$$
{\xi(t) = \alpha \circ \eta(t) = \tp\eta(t)},
$$
the two filtrations are $(\eta, \xi)-$interleaved.
\end{proof}


\subsection{Proof of Theorem~\ref{theorem:momdist-influence}}
\label{proof:theorem:momdist-influence}

We begin by observing that $\norminf{ \dsf_{n+m} - \dsf_n }$ can be bounded from below as follows:
\eq{
    \norminf{ \dsf_{n+m} - \dsf_n } &\ge \dsf_n(\xvo) - \dsf_{n+m}(\xvo)\n 
    &= \dsf_n(\xvo) - 0 \n
    &=\inf_{\xv \in \Xn}\norm{\xv - \xvo} \n
    &\ge\inf_{\xv \in \Xb}\norm{\xv - \xvo} = \dx(\xvo). \nn
}
Furthermore, for {$\delta/2 \le e^{-(1+b)Q}$} and $k \defeq \max\qty\big{1, {2^{\f{b-1}{b}}}}$, with probability greater than $1-\delta$,

{\eq{
    &\norminf{\dsf_{n+m,Q} - \dsf_n} \n 
    &\qq{} \num{\le}{i} \norminf{\dsf_{n+m,Q} - \dx} + \norminf{\dsf_n - \dx}\n
    &\qq{} \num{\le}{ii} {2}\qty\Bigg[ \f{1}{a\nQ}\wo\qty( \nQ \exp\qty{ \f{4\log(2/\delta)}{Q-2m} } ) ]^{1/b} + {2}\qty\Bigg[ \f{1}{an}\wo\qty( n \exp\qty{ 4\log(2/\delta) } ) ]^{1/b}\n
    &\qq{} \num{\le}{iii} \f{{2}k}{a^{1/b}}\qty[{ \f{1}{\nQ} \wo\qty( \nQ \exp\qty{ \f{4\log(2/\delta)}{Q-2m} } ) + \f{1}{n} \wo\qty( n \exp\qty{ 4\log(2/\delta) } )}]^{1/b}\n
    &\qq{} \num{\le}{iv} \f{{2}k}{a^{1/b}}\qty[{\f{\log\nQ}{\nQ}} + \f{\log n}{n} + 4\log(2/\delta)\qty( \f{1}{\nQ(Q-2m)} + \f{1}{n} ) ]^{1/b}\n
    &\qq{} \num{\le}{v} \f{{2}k2^{1/b}}{a^{1/b}}\qty(\f{\log\nQ + 4\log(2/\delta)}{\nQ})^{1/b} \defeq \et,\nn
}
}

where, for $\nQ = (n+m)/Q$, (i) is a consequence of the triangle inequality and (ii) follows from the proofs of Theorem~\ref{theorem:momdist-sublevel} and Theorem~\ref{theorem:momdist-consistency}, (iii) uses Lemma~\ref{lemma:useful-inequalities}, (iv) follows from the fact that $\wo(z) < \log(z)$ for $z > e$, and (v) uses the fact that $\nQ < n$ and $(Q-2m)\inv \le 1$ for {$Q > 2m$}.

Observe that if $2\et \le \dx(\xvo)$, then with probability greater than $1-\delta$,
\eq{
    \norminf{ \dsf_{n+m} - \dsf_n } - \norminf{\dsf_{n+m,Q} - \dsf_n} \ge \dx(\xvo) - \et \ge \et,
    \label{eq:momdist-influence-triangle}
}
and the result follows. Therefore, in order to establish the claim for the second part it suffices to check that ${2\et \le \dx(\xvo)}$ under conditions (I) and (II). To this end, note that
{\eq{
    \dx(\xvo) \ge 2\et \quad \Longleftrightarrow \quad \vp \ge \f{\log\nQ + 4\log(2/\delta)}{\nQ},\nn
}
}
which is satisfied whenever $\delta$ satisfies  the r.h.s. of condition (II), i.e.,
{
\eq{
    \log(2/\delta) \le \f{\nQ\vp - \log\nQ}{4}.\nn
}
}
Furthermore, the l.h.s. of condition (II), i.e., {$\delta \le 2e^{-(1+b)Q}$}, is satisfied only when
{
\eq{
    (1+b) Q \le \f{\nQ\vp - \log\nQ}{4},\nn
}
}
or, equivalently, when condition (I) is satisfied:

{
\eq{
    \vp \ge \f{\log\nQ}{\nQ} + \f{4(1+b)Q}{\nQ}.\nn
}
}
The result now follows from \cref{eq:momdist-influence-triangle}. \QED


\begingroup
\providecommand{\hQ}{\widehat{Q}}
\providecommand{\hm}{\widehat{m}}
\renewcommand{\ms}{m^*}
\providecommand{\h}{\mathfrak{h}}

\subsection{Proof of Theorem~\ref{theorem:lepski}}
\label{proof:theorem:lepski}

Let {$\js = \min\pb{ j \in \J : m(j) \ge \ms }$.} By definition of $\J$ we have that $\abs{\J} \le 1 + \log_\theta(\mmax/\mmin)$ and $m(\js) < \theta\ms$ for $\theta > 1$. The outline of the proof is as follows. First, we show that $\mathfrak{h}(n,m,\delta)$ is non-decreasing in $m$, from which it follows that $\mathfrak{h}(n,m(j),\delta) \le \mathfrak{h}(n,m(j+1),\delta)$. Next, we show that the event $\pb{ \jh \le \js }$ contains the event $\mathcal{E}$ given by
\eq{
    \mathcal{E} = \bigcap_{\qty{j \in \J : j \ge \js}}\qty\Big{ \winf\qty\big( \bbv_n(j), \bbv[\bX] ) \le \h(n, m(j), \delta) }.\nn
}
Then, using a standard procedure for obtaining the Lepski bound (e.g., Theorem~5.1 of \citealt{minsker2018sub} and Theorem~3.1 of \citealt{chen2020robust}), we show that the event $\mathcal{E}$, and, therefore the event $\pb{\jh \le \js}$, holds with probability at least $1 - \delta \log_\theta(\mmax/\mmin)$. Lastly, we use the bound on the event $\pb{\jh \le \js}$ to obtain the desired result. 

\textbf{1. Monotonicity of $\h(n, m, \delta)$ in $m$.} Consider the function $f(z; \alpha, \beta) = \alpha\wo(\beta z)/ z$ for fixed constants $\alpha, \beta > 0$. The derivative of $f$ is given by
\eq{
    f'(z; \alpha, \beta) = \f{d}{dz} \qty(\f{\alpha}{z}\wo(\beta z)) &= \alpha\pa{ \f{\beta}{z}\wo'(\beta z) - \f{1}{z^2}\wo(\beta z) }\nn\\[10pt]
    &\num{=}{i}\alpha\pa{ \f{\beta}{z}\pb{\f{\wo(\beta z)}{\beta z (1 + \wo(\beta z))}} - \f{1}{z^2}\wo(\beta z) }\nn\\[10pt]
    &= - \f{\alpha \wo(\beta z)^2}{z^2(1+\wo(\beta z))} < 0 \qq{for all} z > 0.\nn
}
Note that in (i) we have used the fact that the derivative of the Lambert $\wo$ function is given by $\wo'(z) = {\wo(z)}/{z(1+\wo(z))}$. Therefore, it follows that $f$ is non-increasing in $z$. The claim follows by noting that the function $\mathfrak{h}$ is given by
\eq{
    \h(n, m, \delta) = f( n/(2m+1); \alpha_1, \beta_1 )^{1/b} + f(n-m; \alpha_2, \beta_2)^{1/b},\nn
}
for constants $\alpha_1, \beta_1, \alpha_2, \beta_2 > 0$ not depending on $n$ or $m$.

\textbf{2. $\mathcal{E}$ is a subset of $\pb{\jh \le \js}$.} We begin by noting that since $\ms \le m(\js)$, it follows that $2\ms < Q(j) = 2 m(j) + 1$ for all $j \ge \js$ and satisfies the first condition for Theorem~\ref{theorem:momdist-consistency}. By taking 
$$
\delta_1 = e^{-2(1+b)(2\mmax + 1)} \le e^{-2(1+b)Q(j)},
$$ 
and {$\delta_2 = \delta - \delta_1$, note that $\h\qty\big(n, m(j), \delta) = {2}t_1 + {2}t_2$ for the two terms, $t_1, t_2$, appearing in \cref{eq:t1_t2} from the proof of Theorem~\ref{theorem:momdist-consistency} by taking $m=m(j)$ and $Q = Q(j) = 2m(j)+1$.} Therefore, we may use Theorem~\ref{theorem:momdist-consistency} to obtain
\eq{
    \pr\qty\bigg( \winf\qty\Big( \bbv_n(j), \bbv[\bX]) > \h(n, m(j), \delta) ) < \delta \qq{for all} j \ge \js.
    \label{eq:j-condition}
}
Furthermore, by definition of $\jh$, it follows that for all $j < \jh$, there exists at least one $i > j$ such that ${\winf\pa{ \bbv_n(i), \bbv_n(j) } > 2\h(n, m(i), \delta)}$. Therefore,
\eq{
    \pb{ \jh > \js }
    &\subseteq \bigcup_{\mathclap{\pb{ j \in \J : j > \js }}} \qty\Big{ \winf\qty\big( \bbv_n(j), \bbv_n(\js) ) > 2\h( n, m(j), \delta ) }\nn\\
    &\num{\subseteq}{ii} \bigcup_{\mathclap{\pb{ j \in \J : j > \js }}} \qty\Big{ \winf\qty\big( \bbv_n(j), \bbv[\bX])  > \h( n, m(j), \delta ) } \cup \qty\Big{ \winf\qty\big( \bbv_n(\js), \bbv[\bX]) > \h( n, m(\js), \delta ) }\nn\\
    &= \bigcup_{\mathclap{\pb{ j \in \J : j \ge \js }}} \qty\Big{ \winf\qty\big( \bbv_n(j), \bbv[\bX])  > \h( n, m(j), \delta ) } \defeq \mathcal{E}^c,\nn
}
where, in (ii) we have used the fact that $\h(n, m(\js), \delta) \le \h(n, m(j), \delta)$ for all $j > \js$, and
\eq{
    \qty\Big{ \winf\qty\big( \bbv_n(j), \bbv[\bX])  \le \h( n, m(j), \delta ) } \cap \qty\Big{ \winf\qty\big( \bbv_n(\js), \bbv[\bX]) \le \h( n, m(\js), \delta ) } \nn\\
    \subseteq \qty\Big{ \winf\qty\big( \bbv_n(j), \bbv_n(\js) ) \le 2\h( n, m(j), \delta ) }. \qq{} \qq{} \nn
}
By inverting the above inclusion we get the inclusion in (ii). Therefore, we obtain ${\mathcal{E} \subseteq \pb{ \jh \le \js }}$.

\textbf{3. Tail bound for the event $\mathcal{E}$.} Applying a union bound to (34), we obtain
\eq{
    \pr( \mathcal{E}^c ) &= \pr\qty( \bigcup_{\pb{ j \in \J : j \ge \js }} \qty\Big{ \winf\qty\big( \bbv_n(j), \bbv[\bX])  > \h( n, m(j), \delta ) } )\nn\\
    &\le \sum_{\pb{ j \in \J : j \ge \js }} \pr\qty\bigg( \winf\qty\big( \bbv_n(j), \bbv[\bX])  > \h( n, m(j), \delta ) )\nn\\
    &\num{\le}{iv} \sum_{\pb{ j \in \J : j \ge \js }} \!\!\!\delta\nn\\
    &\num{\le}{v} \delta \log_\theta\qty( \f{\theta\mmax}{\mmin} ),\nn
}
where (iv) follows from \cref{eq:j-condition} and (v) uses the fact that $\abs{\J} \le 1 + \log_\theta(\mmax/\mmin)$.

\textbf{4. Bound for $\winf(\bbv_n(\jh), \bbv[\bX])$.} We begin by noting that when the event $\mathcal{E}$ holds, we have that
\eq{
    \winf\qty(\bbv_n(\jh), \bbv[\bX]) &\le \winf\qty(\bbv_n(\jh), \bbv_n(\js)) + \winf\qty(\bbv_n(\js), \bbv[\bX])\nn\\
    &\num{\le}{vi} 2\h( n, m(\js), \delta ) + \h(n, m(\js), \delta)\nn\\
    &\num{\le}{vii} 3\h( n, \theta\ms, \delta ),\nn
}
where the first term in (vi) follows from the definition of $\jh$, which is guaranteed to hold because $\mathcal{E} \subseteq \pb{\jh \le \js}$, and the second term in (vi) follows from the definition of $\mathcal{E}$. The inequality in (vii) uses the fact that $m(\js) < \theta\ms$ and the fact that $h(n, m, \delta)$ is non-decreasing in $m$. Therefore, we have the inclusion
\eq{
    \mathcal{E} \subseteq \qty\Big{ \winf\qty(\bbv_n(\jh), \bbv[\bX]) \le 3\h( n, \theta\ms, \delta ) }.\nn
}
Using the tail bound on $\mathcal{E}$ we obtain
\eq{
    \pr\qty\Big( \winf\qty\big(\bbv_n(\jh), \bbv[\bX]) \le 3\h( n, \theta\ms, \delta ) ) \ge  \pr( \mathcal{E} ) \ge 1 - \delta\log_\theta\qty( \f{\theta\mmax}{\mmin} ),\nn
}
which is the desired result. \QED

\endgroup



\section{Technical Lemmas for Appendix~\ref{sec:proofs}}
\label{sec:technical}


The proof of Theorem~\ref{theorem:momdist-sublevel} relies on the following lemma, which allows us to control the deviation of a pointwise median-of-means estimator from its uncontaminated population counterpart in terms of a Binomial tail probability.

\begin{lemma}\label{lemma:mom}
    Suppose $\pr \in \mathcal{P}(\bX)$ for $\bX \subset \R^d$ and $\Xn = \Xnm \cup \Ym$ is obtained under sampling condition \ref{setting} with $\Xnm$ observed i.i.d. from $\pr$. Let $\pr_n$ denote the empirical measure associated with $\Xn$ and for $2m < Q < n$, let $\pr_q$ be the empirical measure associated with the block $S_q$ for all $q \in [Q]$. Given a statistical functional {$T : \mathcal{P}(\R^d) \rightarrow \F(\R^d)$, let $T_Q(\pr_n) \in \F(\R^d)$} be the pointwise MoM estimator given by
    \eq{
        T_Q(\pr_n)(\xv) = \med\qty\Big{ T(\pr_q)(\xv) : q \in [Q] }, \  \ \textup{ for all } \xv \in \R^d.\nn
    }
    Then, for $t > 0$
    \eq{
        \pr\qty\bigg({ \norminf{ T_Q(\pr_n) - T(\pr) } > t }) \le \pr\qty({ \sum_{q \in A} \xi_{q}(t; n, Q) > \f{Q-2m}{2} }),\nn
    }
    where $A = \qty{ q \in [Q] : S_q \cap \Y_m = \varnothing}$ are the indices for the blocks containing no outliers, and
    \eq{
        \xi_{q}(t; n, Q) \defeq \mathds{1}\qty\Big( \norminf{ T(\pr_q) - T(\pr) } > t ) \ \ \textup{ for all } q \in A. \nn
    }
\end{lemma}

The proof is provided in Appendix~\ref{proof:lemma:mom}. The statement of Lemma~\ref{lemma:mom} holds for empirical processes arising from general classes of pointwise median-of-means estimators. In particular, by taking $T(\pr_q) = \dsf_{s,q}$ to be the distance function w.r.t. block $S_q$, the estimator $\dnq$ satisfies the conditions of Lemma~\ref{lemma:mom}. We also point out that the exponential concentration bound in Theorem~\ref{theorem:momdist-sublevel} is strictly better than similar bounds appearing in other pointwise MoM estimators (e.g., \citealp[Theorem~2]{humbert2020robust}). This is owing to the Chernoff bound (instead of a Hoeffding bound) used for bounding the Binomial tail probability appearing in Lemma~\ref{lemma:mom}. This provides a significant gain for Binomial random variables with shrinking probability \citep{Hagerup1990AGT}.

The next two results, Lemma~\ref{lemma:ab-filtration} and Lemma~\ref{lemma:ab-module} will be of assistance, and serve as generalizations of \cite[Lemma~4.8 \& Proposition~4.9]{anai2019dtm}. 

We state the first result for a general metric space $(\M, \rho)$ and an arbitrary weight function $f$. Here the ball of radius $r$ centered at $\xv \in \M$ is denoted $\Bfx[\rho][][r]$, and for a compact set $\bX \subset \M$, the $r$--offset of $\bX$ w.r.t the metric $\rho$ is given by $\bX[\rho](r) = \bigcup_{\xv \in \bX}\Bfx[\rho][][r]$. The following result provides a handle for the interleavings between $f$-weighted filtrations computed on two nested sets using the same function $f$. 

\begin{lemma}
    Given a metric space $(\M, \rho)$, two subsets $\bX, \bY$ of $\M$ such that $\bX \subseteq \bY$, and a weight function $f: \mathcal{M} \rightarrow \R_{\ge 0}$, let $\Vt[  ][\bX,f][\rho]$ and $\Vt[  ][\bY, f][\rho]$ be their respective $f$--weighted filtrations. If $f$ satisfies the property that
    \eq{
        \inf_{\xv \in \bX}\rho(\xv, \yv) \le f(\yv) + a,\nn
    }
    for {$a \ge 0$} and for all $\yv \in \bY$, then the filtrations are $(\id,\alpha)$--interleaved, i.e.,
    \eq{
        \Vt[][\bX, f][\rho] \subseteq \Vt[][\bY, f][\rho] \subseteq \Vt[\alpha(t)][\bX, f][\rho],\nn
    }
    for $\alpha: t \mapsto 2^{1 - \f 1 p} t + a + \sup_{\xv \in \bX}f(\xv)$ and {for all $t \ge 0$.}
    \label{lemma:ab-filtration}
\end{lemma}

The proof of Lemma~\ref{lemma:ab-filtration} is provided in Appendix~\ref{proof:lemma:ab-filtration}. Since map $\alpha$ appearing in Lemma~\ref{lemma:ab-filtration} is not purely a translation map, it does not lead to a bound in the interleaving metric as per \eref{eq:interleaving-filtration}, and, therefore, a bound in the $\winf$ metric cannot be characterized using Lemma~\ref{lemma:ab-filtration} alone.  The next result, which is stated only for the Euclidean space $(\R^d, \norm{\cdot})$, establishes that for sufficiently large values of $t$, the map $\alpha$ may be replaced by a translation map. The proof is provided in Appendix~\ref{proof:lemma:ab-module}.

\begin{lemma}
    Let $(\mathcal{M},\rho) = (\R^d, \norm{\cdot})$. Suppose $\bX$, $\bY \subset \R^d$ and $\bX \subseteq \Yb$, and $f$ satisfies the same conditions as in Lemma~\ref{lemma:ab-filtration} for {$a\ge 0$}. Let $t({\bX})$ be the filtration value for the simplex corresponding to $\bX$ in $\textup{nerve}\pb{\VVt[ ][\bX, f][\rho]}$, i.e.,
    \eq{
    t({\bX}) \defeq \inf \qty\Big{t>0: {\textstyle \bigcap\limits_{\xv \in \bX} }  B_{f, \rho}(\xv, t) \neq \varnothing},\nn
    }
    and $\beta: t \mapsto t + c(\bX)$ be a non-decreasing map with
    \eq{
        c(\bX) \defeq a +  \sup_{\xv \in \bX}f(\xv) + \pa{1 - \f 1 p}t(\bX).\nn
    }
    Then for all $t \ge t(\bX)$, the homomorphisms $\phi_{t}^{\beta(t)}: \bVt[t][\bX, f][\rho] \rightarrow \bVt[\beta(t)][\bX, f][\rho]$ are trivial, i.e.,
    $$
        {\textup{Im}\qty\big(\phi_{t}^{\beta(t)})} \cong \begin{cases}
            \mathbf{F}      & \text{ \ \ if \ \ } \bVt[t][\bX, f][\rho] = \textup{Hom}_0\pa{\Vt[t][\bX, f][\rho]}\nn      \\
            \pb{\mathsf{0}} & \text{\ \ \  if \ \ } \bVt[t][\bX, f][\rho] = \textup{Hom}_k\pa{\Vt[t][\bX, f][\rho]}, k>0
        \end{cases}.
    $$
    Furthermore, the bottleneck distance between the resulting $f$--weighted persistence diagrams is bounded above as
    \eq{
        \Winf\qty\Big( \dgm\pa{\bVt[  ][\bX, f][\rho]}, \dgm\pa{\bVt[  ][\bY, f][\rho]} ) \le c(\bX).\nn
    }
    \label{lemma:ab-module}
\end{lemma}

\begin{remark}
    Unlike Lemma~\ref{lemma:ab-filtration}, which is stated for general metric spaces, restricting ourselves to the Euclidean space $(\R^d, \norm{\cdot})$ in Lemma~\ref{lemma:ab-module} is sufficient for the objective of this work. However, as outlined in the proof, the only issue arises when \citet[Lemma~B.1]{anai2019dtm} is invoked. While \citet[Lemma~B.1]{anai2019dtm} (which holds for affine spaces satisfying the parallelogram identity) extends naturally to Banach spaces, the extension to general metric spaces will require some care on a case-by-case basis.
\end{remark}


The following lemma provide a useful characterization of the persistence diagram obtained using the sublevel sets of $\dnq$.

\begin{lemma}
    Given samples $\Xn$ and $Q<n$, $V[\dnq]$ and $\alpha: t \mapsto 2^{\f{p-1}{p}} t$, $V{[\R^d, \dnq]}$ are $(\id, \alpha)-$interleaved for  for all $p \ge 1$. In particular, $V[\dnq] = V{[\R^d, \dnq]}$ when $p=1$. 
    \label{lemma:sublevel-equivalence}
\end{lemma}

The proof of Lemma~\ref{lemma:sublevel-equivalence} is provided in Appendix~\ref{proof:lemma:sublevel-equivalence}. Note that, $V^t[\dnq]$ is the sublevel sets of $\dnq$, whereas $V{[\R^d, \dnq]}$ is the $\dnq$-weighted offset of all points in $\R^d$; when $p \neq 1$, they are not necessarily the same.


\subsection{Proof of Lemma~\ref{lemma:mom}}
\label{proof:lemma:mom}
    For $t > 0$, define two events
    \eq{
        E_1 = \pb{\norminf{T_Q(\pr_n) - T(\pr)} \le t},\;{and}\;
        E_2 = \pb{ \#\qty\big{q \in [Q]: \norminf{ T(\pr_q) - T(\pr) } > t} \le \f{Q}{2} }.\nn
    }
    
    First, we show that $E_2 \subseteq E_1$. To this end,
    \eq{
         E_2 &\subseteq  \pb{\#\qty\Big{q \in [Q]: \norminf{T(\pr_q) - T(\pr)} > t} \le \f{Q}{2}}\n
        &\subseteq   \pb{\#\qty\Big{q \in [Q]: \norminf{T(\pr_q) - T(\pr)} \le t} > Q - \f{Q}{2}}\n
        &\subseteq   \pb{\#\qty\Big{q \in [Q]: \forall \xv \in \R^d, T(\pr)(\xv) - t \le T(\pr_q)(\xv) \le T(\pr)(\xv) + t} > \f{Q}{2}}\n
        &\subseteq   \pb{\forall \xv \in \R^d, T(\pr)(\xv) - t \le \med\qty\big{T(\pr_q)(\xv) : q \in [Q]} \le T(\pr)(\xv) + t}\n
        &\subseteq   \pb{\forall \xv \in \R^d, T(\pr)(\xv) - t \le T_Q(\pr_n)(\xv) \le T(\pr)(\xv) + t}\n
        &\subseteq   \pb{\norminf{T_Q(\pr_n) - T(\pr)} \le t} = E_1.\nn
    }
    Therefore, we have $E_2 \subseteq E_1$. Next, note that $E_2$ can be written as 
    \eq{
        E_2 = \pb{\sum_{q=1}^Q \rqt \le \f Q 2},\nn
    }
    where, for each $q \in [Q]$,
    \eq{
        \rqt \defeq \mathds{1}\qty\Big( \norminf{ T(\pr_q) - T(\pr) } > t ).\nn
    } 
    Since $0 \le \rqt \le 1$ a.s., we have that
    \eq{
        \sum_{q=1}^Q \rqt 
        &= \sum_{q \in A}\rqt + \sum_{q\in A^c}\rqt\\
        &\le \sum_{q \in A}\rqt + \abs{A^c} \le \sum_{q \in A}\rqt + m.\nn
    }
    As a result, we can further bound the probability of $E_2$ from below as
    \eq{
        \pr(E_2) \ge \pr\pa{\sum_{q \in A}\rqt \le \f Q 2 - m}.
        \label{mom-lemma-1}
    }
    Combining \eref{mom-lemma-1} with the fact that $E_2 \subseteq E_1$, we obtain
    \eq{
        \pr\qty\bigg({ \norminf{ T_Q(\pr_n) - T(\pr) } > t }) = \pr(E_1^c) \le \pr(E_2^c) \le \pr\pa{\sum_{q \in A}\rqt > \f Q 2 - m},\nn
    }
    which gives us the desired result. \QED


\subsection{Proof of Lemma~\ref{lemma:ab-filtration}}
\label{proof:lemma:ab-filtration}

Since $\bX \subseteq \bY$, the inclusion $\Vt[][\bX, f][\rho] \subseteq \Vt[][\bY, f][\rho]$ holds trivially. For the next part, let ${\vt{t} = \Vt[][\bX, f][\rho]}$ and ${\ut{t} = \Vt[][\bY, f][\rho]}$ denote the respective $f$--weighted filtrations, so as to avoid the notational overload. In order to show the second inclusion, i.e., $\ut{t} \subseteq \vt{\alpha(t)}$, consider $\zv \in \ut{t}$. Then, there exists $\yv \in \bY$ such that $\zv \in \ball{\yv, t}$. If $\yv \in \bX \subset \bY$, then it immediately follows that ${\zv \in \vt{t} \subseteq \vt{\alpha(t)}}$. In what remains, for $\yv \in \bY \setminus \bX$, it is sufficient to show that there exists $\xv \in \bX$ such that ${\zv \in \ball{\xv, \alpha(t)}}$.

To this end, let $\xvy = \arginf_{\xv \in \bX}\rho(\xv, \yv)$ be the projection of $\yv$ onto $\bX$ via $\rho$. Then two following cases arise: (I) $\rho(\xvy, \zv) \le \rho(\xvy, \yv)$, and (II) $\rho(\xvy, \zv) \ge \rho(\xvy, \yv)$ (see Figure~\ref{fig:proof:ab-filtration1}).

\begin{figure}
    \centering
    \begin{subfigure}[b]{0.45\textwidth}
        \centering
        \includegraphics[width=\textwidth]{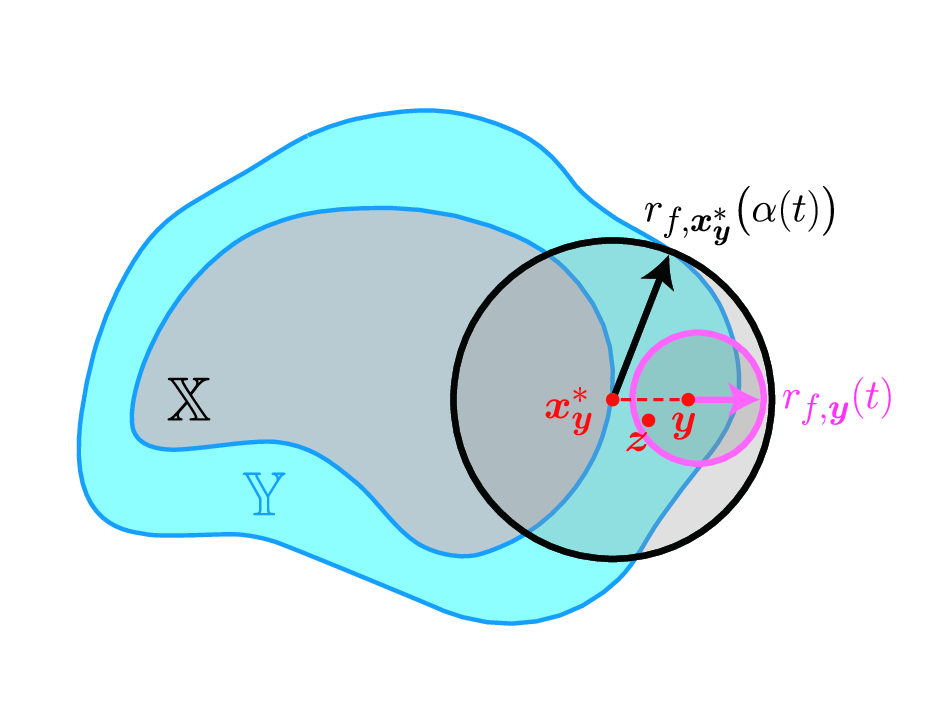}
    \end{subfigure}
    \begin{subfigure}[b]{0.45\textwidth}
        \bigskip
        \includegraphics[width=\textwidth]{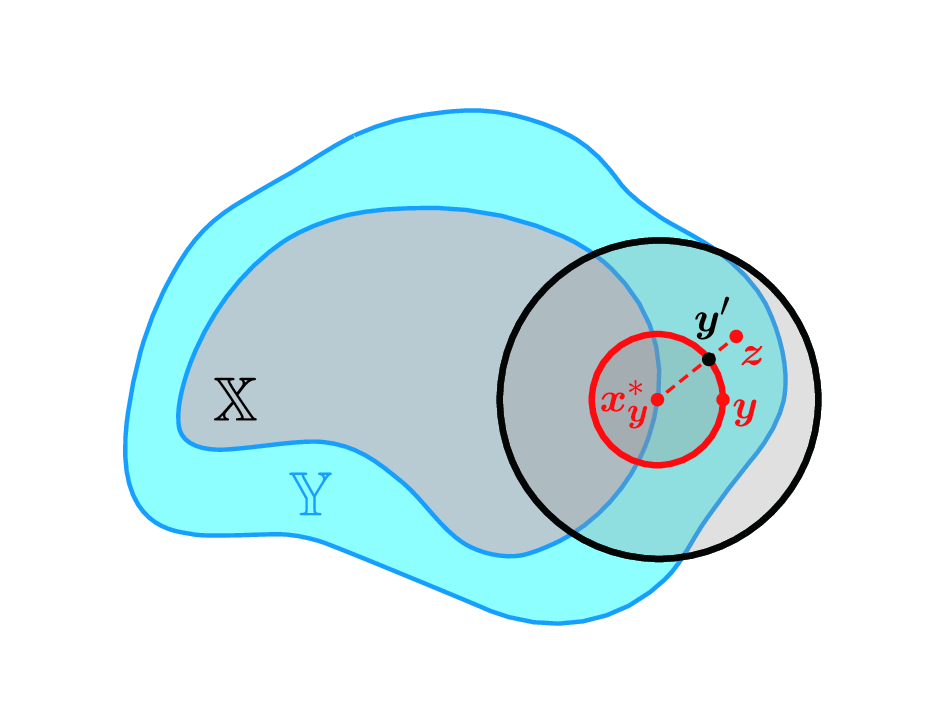}
    \end{subfigure}
    \caption{Illustration of Case I (Left) and Case II (Right).}
    \label{fig:proof:ab-filtration1}
\end{figure}

\textbf{Case I.} The distance between $\xvy$ and $\zv$ will satisfy
\eq{
    \rho(\xvy, \zv) \le \rho(\xvy, \yv) &\num{\le}{i} f(\yv) + a\n
    &\num{\le}{ii} \qty\big({t^p - \rho(\yv, \zv)^p})^{\fpp} + a\n
    &\le t +  a\nn
}
where (i) follows from the assumption on $f$, and (ii) follows from the fact that if $\zv \in \ball{\yv, t}$, then $\rho(\yv, \zv) \le \rfx[f][\yv][t] = \pa{t^p - f(\yv)^p}^{1/p}$. Furthermore, from Lemma~\ref{lemma:useful-inequalities}~(vi) we obtain
\eq{
    \rho(\xvy, \zv) &\le \qty\Big( (t+a + f(\xvy))^p - f(\xvy)^p )^{\fpp}\n
    &\le \qty\Big( \qty\big(t+a + \sup_{\xv \in \bX}f(\xv))^p - f(\xvy)^p )^{\fpp}\n
    &\le \qty\Big( \qty\big(2^{1-\fpp}t+a + \sup_{\xv \in \bX}f(\xv))^p - f(\xvy)^p )^{\fpp}\n
    &= \qty\big( \alpha(t)^p - f(\xvy)^p )^{\fpp} = \rfx[f][\xvy][\alpha(t)],\nn
}
where the last inequality holds because $2^{1-\fpp} \ge 1$. The last line implies that 
$$
{\zv \in \ball{\xvy, \alpha(t)} \subseteq \vt{\alpha(t)}}.
$$

\textbf{Case II.} For $r = \rho(\xvy, \yv)$ let $\yv'$ be the projection of $\zv$ onto $\partial B\pa{\xvy, r}$, i.e.,
$$
    \yv' = \arginf_{\xv' \in \partial B\pa{\xvy, r}}\rho(\xv',\zv).
$$
The point $\yv'$ satisfies the following three properties: (PI)~${\rho(\xvy, \yv')  = \rho(\xvy, \yv)}$, since $\yv' \in \partial\ball{\xvy, r}$; (P-II)~$\rho(\zv, \yv') \le \rho(\zv, \yv)$ by definition of $\yv'$; and (PIII)~$\rho(\xvy, \yv') + \rho(\yv', \zv) \ge \rho(\xvy, \zv)$ from the triangle inequality.

Since $\zv \in \ball{\yv, t}$, when $\rho(\xvy, \yv) \le a$ we may use the triangle inequality to obtain
\eq{
    \rho(\xvy, \zv) \le \rho(\xvy, \yv) + \rho(\zv, \yv) \le a + \qty( t^p - f(\yv)^p )^{\fpp} \le a + t \le a + 2^{1-\fpp}t.\label{eq:rho-1}
}
Alternatively, when $\rho(\xvy, \yv) > a$ we obtain the following inequality,
\eq{
    t^p &\ge \rho(\yv, \zv)^p + f(\yv)^p \n
    &\num{\ge}{iii} \rho(\zv, \yv)^p + \pa{\rho(\xvy, \yv) - a}^p\n
    &\num{=}{iv} \rho(\zv, \yv)^p + \pa{\rho(\xvy, \yv') - a}^p\n
    &\num{\ge}{v} \rho(\zv, \yv')^p + \pa{\rho(\xvy, \yv') - a}^p\n
    &\num{\ge}{vi} \qty\Big(\rho(\xvy, \zv) - \rho(\xvy, \yv'))^p + \qty\Big(\rho(\xvy, \yv') - a)^p\n
    &\num{\ge}{vii} 2^{1-p}\qty\Big( \rho(\xvy, \zv)-  a)^p,
    \label{eq:proof:ab-filtration1}
}
where (iii) holds from the assumption on $f$, (iv--vi) follow from (PI--PIII) respectively, and (vii) uses Lemma~\ref{lemma:useful-inequalities}\,(i). Rearranging the terms of \eref{eq:proof:ab-filtration1} we get $\rho(\xvy, \zv) \le a + 2^{1-\fpp}t$. Therefore, from \eref{eq:rho-1} and \eref{eq:proof:ab-filtration1}, in case (II) we have that
\eq{
    \rho(\xvy, \zv) &\le 2^{1-\fpp}t + a\n
    &\num{\le}{viii} \qty\Big( \qty\big(2^{1-\fpp}t + a + \sup_{\xv \in \bX}f(\xv))^p - f(\xvy)^p )^{\fpp}\n
    &= \rfx[f][\xvy][\alpha(t)],\nn
}
where (viii) uses Lemma~\ref{lemma:useful-inequalities}~(vi). Similar to case (I), we obtain {$\zv \in \ball{\xvy, \alpha(t)^{ }} \subseteq \vt{\alpha(t)}$}. \QED


\subsection{Proof of Lemma~\ref{lemma:ab-module}}
\label{proof:lemma:ab-module}

\providecommand{\xvo}{{\xv_{o}}}

Let {$t \ge t(\bX)$} where 
$$
t(\bX) \defeq \inf\pb{t > 0: \bigcap_{\xv \in \bX}\ball{\xv, t} \neq \varnothing}, 
$$
and let $\xvo \in \bigcap_{\xv \in \bX}\ball{\xv, t}$. To ease the notation, let $\ut{t} = \Vt[t][\bY, f][\rho]$ denote the usual $f$-weighted filtration, and let $\wt{t}$ be defined as
\eq{
    \wt{t} = \pb{\bigcup_{\xv \in \bX}\ball{\xv, \beta(t)}} \cup \pb{\bigcup_{\yv \in \bY \setminus \bX}\ball{\yv, t}},\nn
}
such that $\ut{t} \subset \wt{t} \subset \ut{\beta(t)}$. With this background, the proof closely follows that of \citet[Proposition~4.8]{anai2019dtm}. Specifically, the proof is based on the following outline:
\begin{enumerate}[label=\protect\circled{\arabic*}]
    \item We first establish that for any $\yv \in \bY \setminus \bX$, there exists $\xv = \xvy \in \bX$ such that for all $t \ge t(\bX)$, $\ball{\yv, t} \cup \ball{\xv, \beta(t)}$ is star-shaped around $\xvo$. Since this holds for all $\yv \in \bY \setminus \bX$, it also holds for $\bigcup_{\yv \in \bY \setminus \bX}\ball{\yv, t}$, and, therefore, $\wt{t}$ is star-shaped and contractible to $\xvo$.
    \item The inclusion map $\iota_t: \ut{t} \hookrightarrow \ut{\beta(t)}$ can be decomposed as $\iota_t = j_t \circ \kappa_t$ where ${j_t: \ut{t} \hookrightarrow \wt{t}}$ and ${\kappa_{t}: \wt{t} \hookrightarrow \ut{\beta(t)}}$. Since $\wt{t}$ is star-shaped and contractible, i.e., $\wt{t} \sim \pb{\xvo}$, the linear map between the homology groups induced by $\kappa_{t}$, i.e., $v_t: \bwt{t} \rightarrow \but{\beta(t)}$ will be trivial.
    \item The interleavings $\alpha(t)$ (Lemma~\ref{lemma:ab-filtration}) and $\beta(t)$ are combined to provide the bound in $\Winf$.
\end{enumerate}

\textbf{Claim \circled{1}.} Let $\yv \in \bY \setminus \bX$. We need to show that there exists $\xv \in \bX$ such that ${\ball{\xv, \beta(t)}\cup\ball{\yv, t}}$ is star-shaped around $\xvo$, i.e., for any $\zv \in \ball{\yv, t}$ the {segment} $\Gamma[\xvo, \zv]$ is contained inside the set ${\ball{\xv, \beta(t)}\cup\ball{\yv, t}}$. See Figure~\ref{fig:lemma:ab-module-claim1}.

\begin{figure}
    \includegraphics[width=\textwidth]{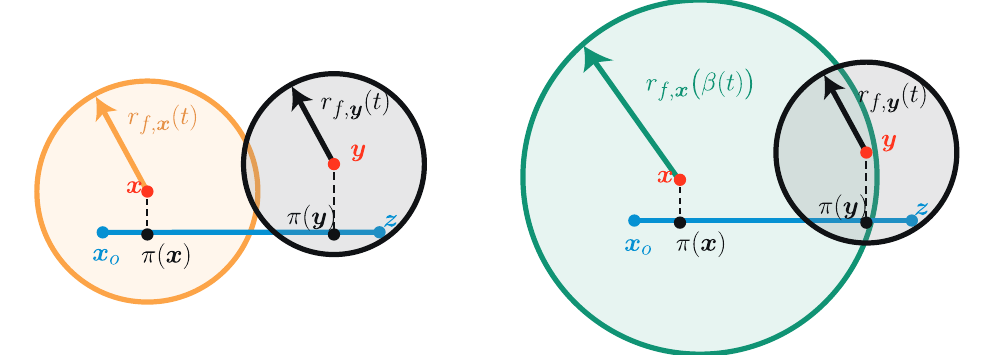}
    \caption{Illustration of Claim \protect\circled{1}.}
    \label{fig:lemma:ab-module-claim1}
\end{figure}

To this end, let $\xv = \arginf_{\zv \in \bX}\rho(\zv, \yv)$ be the projection of $\yv$ onto $\bX$. Note that, from the definition of $\xvo$, $\xvo \in \ball{\xv, t}$ for all $t \ge t(\bX)$. For simplicity, let $S^t = {\ball{\xv, \beta(t)}\cup\ball{\yv, t}}$. Additionally, let $\pi(\xv)$ and $\pi(\yv)$ be the projection of $\xv$ and $\yv$ onto $\Gamma[\xvo, \zv]$, respectively, i.e.,
\eq{
    \pi(\xv) = \arginf_{\xv' \in \Gamma[\xvo, \zv]}\rho(\xv', \xv),\nn
}
mutatis mutandis, the same for $\pi(\yv)$. 

By definition, $\rho(\xv, \pi(\xv)) \le \rho(\xv, \xvo)$ and ${\rho(\yv, \pi(\yv)) \le \rho(\yv, \zv)}$, and consequently, $\pi(\yv) \in \ball{\yv, t}$. This implies that $\Gamma[\pi(\yv), \zv] \subseteq S^t$. What remains to be established is that $\Gamma[\xvo, \pi(\yv)] \subseteq S^t$. In order to show this, note that it is sufficient to show that $\pi(\yv) \in \ball{\xv, \beta(t)}$. Indeed, if this holds, then $\Gamma[\xvo, \pi(\yv)] \subseteq \ball{\xv, \beta(t)} \subseteq S^t$, and it will follow that ${\Gamma[\xvo, \pi(\yv)] \cup \Gamma[\pi(\yv), \zv]  = \Gamma[\xvo, \zv] \subseteq S^t}$.

Let $\tau = \rho(\yv, \pi(\yv))$. Since $\pi(\yv) \in \ball{\yv, t}$, when $\rho(\xv, \yv) > a$ it follows that
\eq{
    \tau &\le \rfx[f][\yv][t]\n
    &\le \qty\Big(t^p - f(\yv)^p)^{\fpp}\n
    &\le \qty\Big(t^p - \qty\big(\rho(\xv, \yv) - a)^p)^{\fpp},\nn
}
where the last inequality follows from the assumption on $f$. Thus, we have
\eq{
    \rho(\xv, \yv) &\le \qty\Big(t^p - \tau^p)^{\fpp} + a.
    \label{eq:lemma:ab-module-eq2}
}
Alternatively, when $\rho(\xv, \yv) \le a$, \eref{eq:lemma:ab-module-eq2} holds trivially. Since $\pi(\xv) \in \ball{\xv, t}$ and $\rho(\xv, \pi(\xv)) \le \rho(\xv, \xvo)$, it follows that
\eq{
    \rho(\xv, \pi(\xv)) \le t(\bX).
    \label{eq:lemma:ab-module-eq3}
}
Since $\rho=\norm{\cdot}$, \citet[Lemma~B.2]{anai2019dtm} holds, which, combined with Eqs.\eqref{eq:lemma:ab-module-eq2}~and~\eqref{eq:lemma:ab-module-eq3} yields
\eq{
\rho(\xv, \pi(\yv))^2 &\num{\le}{i} \qty\Big(\qty\Big(t^p - \tau^p)^{\fpp} + a)^2 + \tau(2t(\bX) - \tau)\n
&\le \qty\Big(t^p - \tau^p)^{\f2p} + \tau(2t(\bX) - \tau) + a^2 + 2a\qty\Big(t^p - \tau^p)^{\fpp}\n
&\num{\le}{ii} (t + \kappa t(\bX))^2 + a^2 + 2at\n
&\le (t + \kappa t(\bX))^2 + a^2 + 2a(t+\kappa t(\bX))\n
&= \qty\big(t+a+\kappa t(\bX))^2,\nn
}
where (i) is a consequence of \citet[Lemma~B.2]{anai2019dtm}, (ii) follows from \cite[Lemma~B.3]{anai2019dtm} and noting that $t^p - \tau^p \le t^p$ since $\tau \le t$, and $\kappa = (1-\fpp)$. The proofs of \citet[Lemma~B.2~\&~B.3]{anai2019dtm} require that the metric $\rho(\cdot, \cdot)$ admits the paralellogram identity for the result to hold. Additionally, from Lemma~\ref{lemma:useful-inequalities}~(vi) we obtain
\eq{
    \rho(\xv, \pi(\yv)) &\le t+a+\kappa t(\bX) \le \qty\Big( \qty\big(t+a+\kappa t(\bX) + \sup_{\xv \in \bX}f(\xv))^p - f(\xv)^p )^{\fpp} = \rfx[f][\xv][\beta(t)].\nn
}
This implies that $\pi(\yv) \in \ball{\xv, \beta(t)}$, and establishes claim \circled{1}.

For claim \circled{2}, note that since $\wt{t} \sim \pb{\xvo}$, for the $k$th homology group $\bwt{t}$, we have that $\bwt{t} \simeq \mathbf{F}$ for $k=0$, and $\bwt{t} \simeq \pb{\mathsf{0}}$ for $k > 0$. Therefore, the map $w_t: \bwt{t} \rightarrow \but{\beta(t)}$ is trivial, and consequently, so is the linear map $\but{t} \rightarrow \but{\beta(t)}$.

In order to show claim \circled{3}, observe that the persistence modules $\but{ }$ and $\bvt{ }$ are
\eq{
    \begin{cases}
        \text{ $(\id,\alpha)$--interleaved } & \text{ for all $t$  and for $\alpha: t \mapsto 2^{1-\fpp}t + a + \sup_{\xv \in \bX}f(\xv)$}\nn         \\
        \text{ $(\id,\beta)$--interleaved }  & \text{ for $t \ge t(\bX)$ and for $\beta: t \mapsto t+a+\kappa t(\bX) + \sup_{\xv \in \bX}f(\xv)$}\nn
    \end{cases}.
}
When $t\le t(\bX)$, from \cite[Lemma~B.1]{anai2019dtm},
\eq{
    \alpha(t) &= t + \qty\Big(2^{1-\fpp}-1)t + a + \sup_{\xv \in \bX}f(\xv) \le t + \kappa t(\bX) + a + \sup_{\xv \in \bX}f(\xv) = \beta(t).\nn
}
Thus, $\alpha(t) \le \beta(t)$ for $t \le t(\bX)$. Since $\beta: t \mapsto t + c(\bX)$ is an additive interleaving for $c(\bX) = \kappa t(\bX) + a + \sup_{\xv \in \bX}f(\xv)$, this implies that
\eq{
    \Winf\qty\big(\dgm\pa{\but{ }}, \dgm\pa{\bvt{ }}) \le c(\bX),\nn
}
which establishes claim \circled{3}. \QED

\subsection{Proof of Lemma~\ref{lemma:sublevel-equivalence}}
\label{proof:lemma:sublevel-equivalence}

\begingroup
\providecommand{\rd}{\R^d}
\renewcommand{\a}{\alpha}

For simplicity, let $f=\dnq$ denote the \md{} function. By definition, $V[f]$ and $V[\rd,f]$ are $( \id, \a )-$interleaved if the following relationship holds
\eq{
    V^t[f] \subseteq V^t[\rd, f] \subseteq V^{\a(t)}[f] \qq{for all} t \ge 0.\nn
}
The first inclusion is straightforward since
\eq{
    V^t[f] \subseteq \bigcup_{\xv \in V^t[f]}B_f(\xv, t) = \bigcup_{\xv \in \rd}B_f(\xv, t) =   V^t[\rd,f].\nn
}

For the second inclusion, suppose $\xv \in V^t[\rd, f]$, i.e., there exists $\yv \in \rd$ such that $\norm{\xv-\yv} \le r_{f,\yv}(t)$. It suffices to show that $\xv \in V^{\a(t)}[f]$. To this end, note that since $\dnq$ is $1-$Lipschitz by Lemma~\ref{lemma:lipschitz}  it follows that
\eq{
    f(\xv) &\le f(\yv) + \norm{\xv-\yv}\n
    &\le f(\yv) + r_{f, \yv}(t)\n
    &= f(\yv) + (t^p - f(\yv)^p)^{\f 1p}\n
    &\num{\le}{i} 2^{\f{p-1}{p}}\qty( f(\yv)^p + (t^p - f(\yv)^p) )^{\f 1p} = 2^{\f{p-1}{p}}t,\nn
}
where (i) follows from an application of Lemma~\ref{lemma:useful-inequalities}~(iii). Since $f(\xv) \le 2^{\f{p-1}{p}}t = \a(t)$, it implies that {$x \in V^{\a(t)}[f]$} and the result follows. When $p=1$, note that $\a(t) = t$, and therefore $V[f] = V[\rd, f]$.
\endgroup
\QED

\section{Auxiliary Results}
\label{sec:auxiliary}

Consider the Huber contamination model with parameter $\eta \in (0, 1/2)$, i.e., 
\begin{align}
    \pr_{\eta, \qr} = (1-\eta)\pr + \eta\qr,\label{eq:huber}
\end{align}
where $\pr$ is the nominal distribution and $\qr$ is the contaminating distribution. The key tool for establishing minimax rates under contamination models is the \textit{modulus of continuity} associated with a loss function.

\begin{definition}\label{def:modulus-of-continuity}
    Let $\qty{\pr_\theta: \theta\in \Theta}$ be a class of distributions, and let $L: \Theta \times \Theta \to \R$ be a loss function. The modulus of continuity of $L$ with respect to the contamination model is given by
    \begin{align}
        \omega(\eta, \Theta) = \sup\qty{ L(\theta, \theta') : \tv({\pr_\theta, \pr_{\theta'}}) \le \frac{\eta}{1-\eta},\; \theta_1, \theta_2 \in \Theta }.\nn
    \end{align}
\end{definition}

In other words, the modulus of continuity ``measures the difficulty of the hardest one-dimensional subproblem'' \citep{donoho1994statistical}. The modulus of continuity is used in the following theorem by \cite{chen2018robust} to establish lower bounds on the minimax risk under $\eta-$Huber contamination.

\begin{theorem}[Theorem~5.1, \citealt{chen2018robust}]\label{thm:chen}
    Suppose there is $\mathfrak{R}_n(0)$ such that for $\eta=0$,
    \begin{align}\label{eq:chen-risk}
        \inf_{\hat{\theta}} \sup_{\theta\in\Theta} \sup_{\qr} \pr_{\eta,\qr}\qty\Bigg{ L(\hat\theta, \theta) > \mathfrak{R}_n(\eta) } \gtrsim c
    \end{align}
    holds for some $c > 0$. Then for any $\eta \in (0, 1/2)$, the minimax bound \eref{eq:chen-risk} holds for $\mathfrak{R}_n(\eta) = \mathfrak{R}_n(0) \wedge \omega_L(\eta, \Theta)$.
\end{theorem}

The following lemma is a collection of well-known inequalities (and their slight variants). We state them here for reference, as they are used frequently in the proofs.

\begin{lemma}
    For $0 < y \le x$ and $p \ge 1$, the following inequalities hold:
    \begin{enumerate}[label=\textup{(\roman*)}, itemsep=7pt]
        \item $x^p + y^p \le (x+y)^p \le 2^{p-1}(x^p + y^p)$;
        \item $2^{1-p}x^{p} - y^p \le (x-y)^p \le x^p - y^p$;
        \item $(x+y)^{\f 1 p} \le x^{\f 1 p} + y^{\f 1 p} \le 2^{\f{p-1}{p}}(x+y)^{\f 1 p}$;
        \item $x^{\f 1 p} - y^{\f 1 p} \le (x-y)^{\f 1 p} \le 2^{\f{p-1}{p}}x^{\f 1 p} - y^{\f 1 p}$;
        \item $y^{1 - \f1p}x^{\f1p} \le x \le y^{1-p}x^p$;
        \item $x \le \qty\Big(\pa{x+y}^p - y^p)^{\fpp}$.
    \end{enumerate}
    \label{lemma:useful-inequalities}
\end{lemma}

\begin{proof}

    \emph{Part (i).} Let $f(y) = (x+y)^p - x^p - y^p$ on the interval $0 < y \le x$. The derivative,
    $$
        f'(y) = p(x+y)^{p-1} - py^{p-1} \ge 0
    $$
    for all $0 < y \le x$ and $p \ge 1$. Therefore $f$ is non-decreasing, and $f(y) \ge f(0) = 0$. This gives us the first inequality. For the second inequality, note that $g(z) = z^p$ is convex for $z\ge 0$. This follows from the fact that $g''(z) = p(p-1)z^{p-2} \ge 0$ for all $z \ge 0$ and $p\ge 1$. By convexity, we obtain
    \eq{
        2^{-p} \pa{x+y}^p = \pa{\half x + \half y}^p \le \f{x^p + y^p}{2},\nn
    }
    which leads to the second inequality.

    \emph{Part (ii).} Let $z = (x-y)$. Applying the first inequality from the preceding part to $z$ and $y$ we get $z^p \le (y+z)^p - y^p$, i.e., $(x-y)^p \le x^p - y^p$. Similarly, from the second inequality, $(z+y)^p \le 2^{p-1}(z^p + y^p)$, which is the same as $2^{1-p}x^p - y^p \le (x-y)^p$.

    \emph{Part (iii).} Taking $a=x^{1/p}$ and $b=y^{1/p}$, from (i) it follows that
    $$
        (x + y) \le (x^{1/p}+y^{1/p})^p \le 2^{p-1}(x + y),
    $$
    and by noting that the map $t \mapsto t^{1/p}$ is increasing on $\R_+$, the result in (iii) follows.

    \emph{Part (iv).} The proof is identical to the proof in Part (ii). The inequalities are obtained by taking $z=(x-y)$, and applying the results of Part (iii).

    \emph{Part (v).} Since $y \le x$, it follows that $1 \le \pa{x/y}^{\f1p} \le x/y \le \pa{x/y}^p$ for $p\ge 1$. By rearranging the terms, we get $x \le y^{1-p}x^p$ and $x \ge y^{1 - \f1p}x^{\f1p}$.

    \emph{Part (vi).} We have $x = (x + y - y) = \qty\big((x+y-y)^p)^{\fpp}$. From Part (ii) we have
    $$
        {(x+y-y)^p \le (x+y)^p - y^p},
    $$
    which, {on rearranging,} yields $x \le \qty\big((x+y)^p - y^p)^{\fpp}$.
\end{proof}


\begin{lemma}[Chernoff-Hoeffding bound simplified]
    Suppose ${ Z_1, Z_2, \dots Z_N }$ are i.i.d. $\textup{Bernoulli}(p)$ random variables. Then, for $0 < \e < 1$,
    \eq{
        \pr\pa{\f{1}{N}\sum_{1 \le i \le N}Z_i > \e} \le \min\qty{1, \exp\qty\Bigg( N \pa{\f{2}{e} + \e \log(p)})}.\nn
    }
    \label{lemma:chernoff-hoeffding}
\end{lemma}

\providecommand{\kl}{\mathsf{KL}}

\begin{proof}

For $0 < \e < 1$, using the Chernoff-Hoeffding bound for binomial random variables \cite[Theorem~1]{hoeffding1963probability} we have
\eq{
    \pr\pa{\f{1}{N}\sum_{1 \le i \le N}Z_i > \e} \le \exp\qty\Bigg(-N \cdot \kl\qty\Big(\textup{Ber}(\e) || \textup{Ber}(p))),
    \label{eq:mom-concentration-4}
}
where $\textup{Ber}(\e)$ and $\textup{Ber}(p)$ are Bernoulli distributions with parameters $\e$ and $p$ respectively, and $\kl(\pr || \qr)$ is the Kullback-Leibler divergence of $\qr$ w.r.t $\pr$. Simplifying the quantity in the exponent, we get
\eq{
    \kl\qty\big(\textup{Ber}(\e) || \textup{Ber}(p)) &= \e\log\pa{\f{\e}{p}} + (1-\e)\log\pa{\f{1-\e}{1-p}}\n
    &= \underbrace{\e \log(\e) + (1-\e)\log(1-\e)}_{\ge -2/e} - \e\log(p) - (1-\e)\log(1-p)\n
    &\ge -\f{2}{e} - \e \log(p),\nn
}
where the last inequality uses the fact that $x\log(x) \ge -1/e$ for all $0 \le x \le 1$, and ${-(1-\e)\log(1-p) \ge 0}$ for all ${0 \le \e, p \le 1}$. By noting that $\kl\qty\big(\textup{Ber}(\e) || \textup{Ber}(p)) \ge 0$, it follows that 
$$
\kl\qty\big(\textup{Ber}(\e) || \textup{Ber}(p)) \ge \max\qty{-\f{2}{e} - \e \log(p), 0}.
$$
Substituting this in Eq.~(35) yields the result.
\end{proof}

\let\sh\undefined
\let\len\undefined



\section{Additional Experiments}
\label{sec:additional-exp}

The tools for data-adaptive construction of $\dnq$-weighted filtrations, in addition to the code for all experiments, are made publicly available in the \href{https://www.github.com/sidv23/RobustTDA.jl/}{\textsf{RobustTDA.jl}} Julia package\footnote{\url{https://www.github.com/sidv23/RobustTDA.jl/}}. In all experiments, the persistence diagrams are computed using the \textsf{Ripserer.jl} backend \cite{cufar2020ripserer}, and we set the parameter $p=1$ for the weighted-filtrations.

\subsection{Comparison of $\bbv[\dnq]$ and $\bbv[\Xn, \dnq]$}
\label{exp:sublevel}

The objective of this experiment is to illustrate that the $\dnq$-weighted filtration $V[\Xn, \dnq]$ reasonably approximates the sublevel filtration $V[\dnq]$. For the same setup as \ref{exp:adaptive}, {$\Xn$ comprises of $n=550$ points obtained by sampling $500$ points on a circle with additive Gaussian noise ($\s=0.01$) and $m=50$ outliers added from a Matérn cluster process.} For $Q=\widehat Q$ selected using Lepski's method, Figure~\ref{fig:sublevel}\,(a) depicts the \md{} function $\dnq$. Figure~\ref{fig:sublevel}\,(b) illustrates the scatter plot for $\Xn$ with the points colored by the weights $\dnq(\xv_i)$ for each $\xv_i \in \Xn$. The shaded regions show the $\dnq$-weighted offsets $V^t[\Xn, \dnq]$ for $t \in \qty{1.5, 1.75, 2, 2.25}$ colored from white to blue. Figure~\ref{fig:sublevel}\,(c) depicts the sublevel persistence diagram $\dgm\pa{ \bbv[\dnq] }$ computed using cubical homology on a grid of resolution $0.5$. As expected by the result of Proposition~\ref{prop:sublevel-2}, the $\dnq$-weighted persistence diagram $\dgm\pa{\bbv[\Xn, \dnq]}$ in Figure~\ref{fig:sublevel}\,(d) captures the essential topological information in $\dgm\pa{ \bbv[\dnq] }$.

\begin{figure}[H]
    \centering
    \begin{subfigure}[b]{0.23\textwidth}
        \centering
        \includegraphics[height=1.1\textwidth]{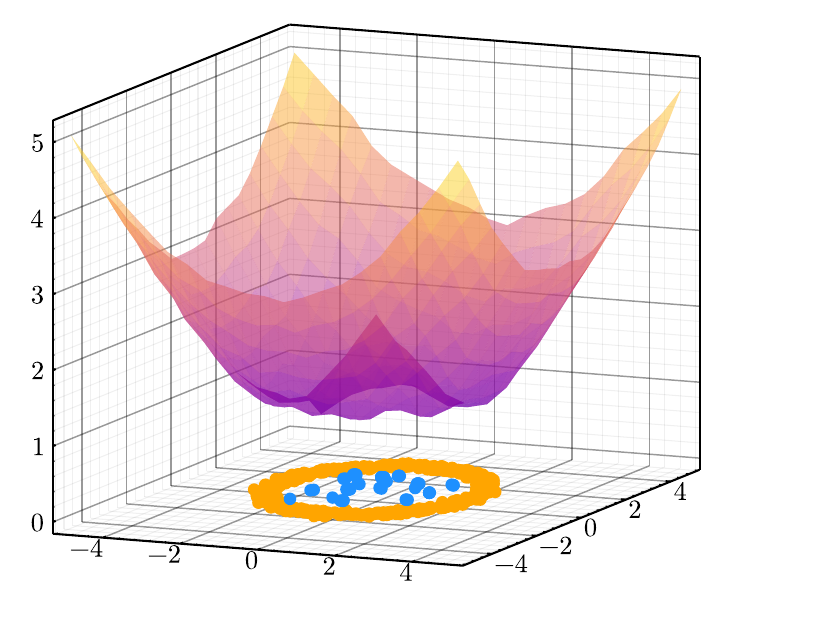}
        \caption{\md{} function $\dnq$}
    \end{subfigure}
    \quad\quad
    \begin{subfigure}[b]{0.23\textwidth}
        \centering
        \includegraphics[height=1.1\textwidth]{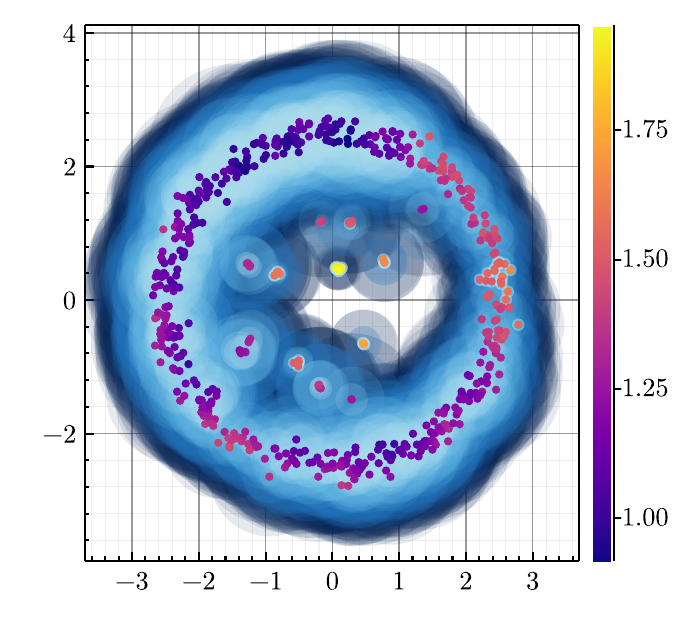}
        \caption{$V^t[\Xn, \dnq]$}
    \end{subfigure}
    \begin{subfigure}[b]{0.23\textwidth}
        \centering
        \includegraphics[width=1.1\textwidth]{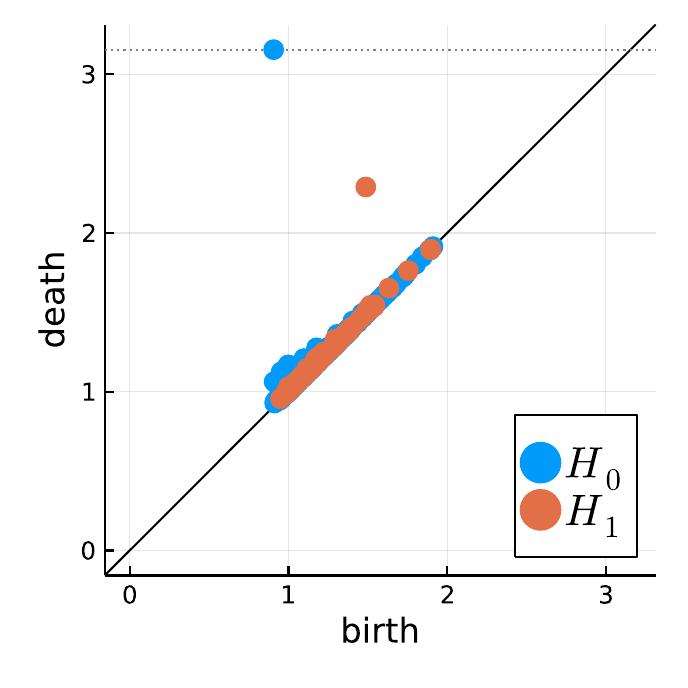}
        \caption{$\dgm\pa{\bbv[\dnq]}$}
    \end{subfigure}
    \begin{subfigure}[b]{0.23\textwidth}
        \centering
        \includegraphics[width=1.1\textwidth]{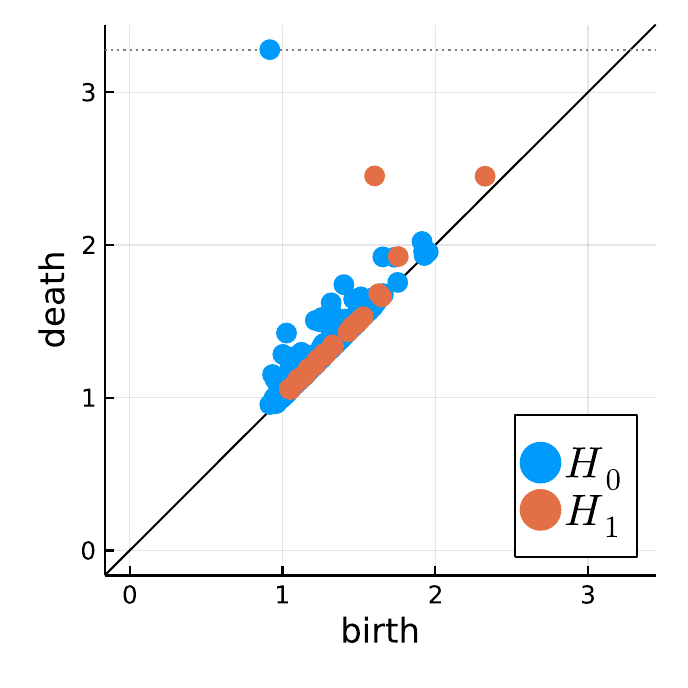}
        \caption{$\dgm\pa{\bbv[\Xn, \dnq]}$}
    \end{subfigure}
    \caption{Comparison of sublevel filtrations with the $\dnq$-weighted filtration.}
    \label{fig:sublevel}
\end{figure}

\begin{figure}[h!]
    \begin{subfigure}[c]{0.44\textwidth}
        \includegraphics[width=\textwidth]{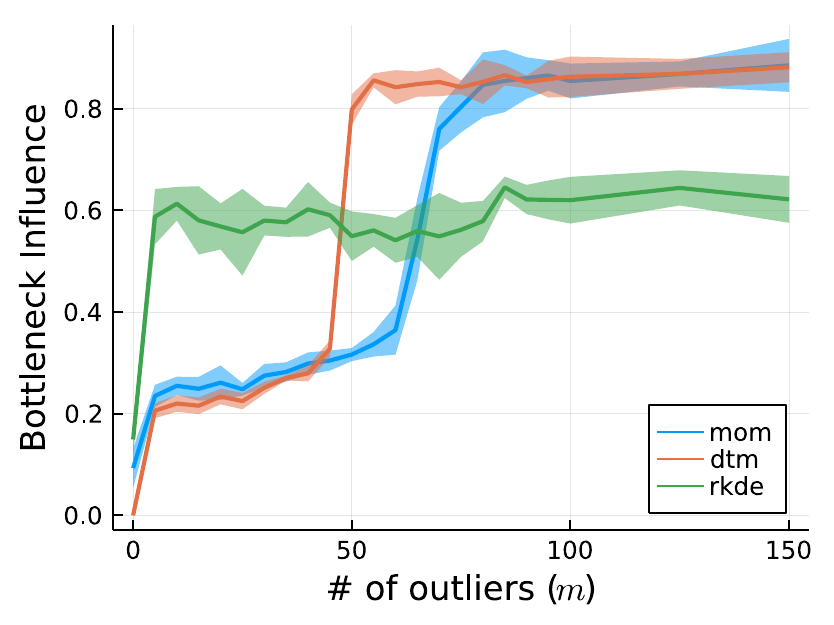}
        \caption{$\text{influence}\pa{\winf; \Xn, f_n, m, \xvo}$}
    \end{subfigure}
    \begin{subfigure}[c]{0.44\textwidth}
        \includegraphics[width=\textwidth]{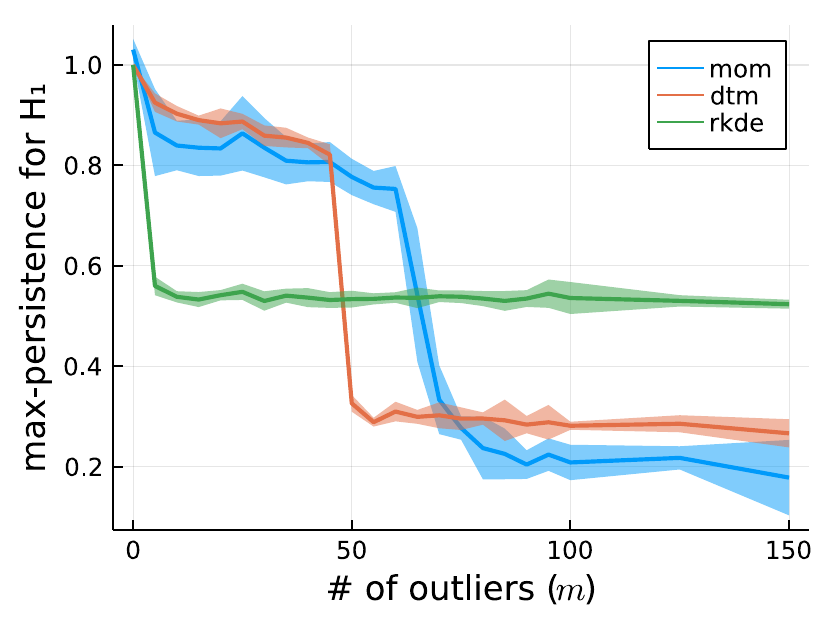}
        \caption{max Persistence for the first order diagram}
    \end{subfigure}
    \vspace*{-7pt}
    \caption{Influence analysis for $\dnq$-weighted filtrations vis-à-vis DTM-based and unweighted filtrations.}
    \label{fig:influence}
\end{figure}

\subsection{Empirical influence analysis}
\label{exp:influence}

\providecommand{\fnms}{\ensuremath{f^{n+m}_{\rho,\s}}}
\providecommand{\Dnms}{\ensuremath{D^{RKDE}_{{n+m}, \rho,\s}}}
\providecommand{\dnms}{\ensuremath{\dsf_{n+m, \rho,\s}}}

In this experiment, we examine the influence of outliers on $\dnq$-weighted filtrations. For $n=500$, points $\Xn$ are sampled uniformly from a circle. We compute the unweighted persistence diagram $D_n = \dgm(\bbv[\Xn])$. In a small neighborhood around the center of the circle, outliers $\Ym$ are sampled uniformly from $[-0.1,0.1]^2$. For the composite sample $\Xn \cup \Ym$ and a fixed value of $Q=100$ \& $k = 50$, we compute the \md{} weighted persistence diagram $D^{MoM}_{n+m,Q} = \dgm(\bbv[\Xn \cup \Ym, \dsf_{n+m,Q}])$, the DTM weighted persistence diagram $D^{DTM}_{n+m,k} = \dgm(\bbv[\Xn \cup \Ym, \delta_{n+m, k}])$, and the RKDE weighted persistence diagram ${\Dnms}$ from the RKDE $\fnms$ using the Hampel loss $\rho$ and a Gaussian kernel $K_\s$. Since the RKDE $\fnms \defeq \sum_{i=1}^{n+m}w_i K_\s(\cdot, \Xv_i)$ does not behave like a distance function, we convert $\fnms$ to a distance-like function $\dnms$ using a similar approach as \cite{phillips2015geometric} to obtain
\eq{
    \dnms(\xv) 
    &\defeq \normh{K_\s(\cdot, \xv) - \fnms}\\ 
    &= \sqrt{\mathop{\sum\sum}\limits_{1 \le i,j \le n+m} w_i w_j K_\s(\Xv_i, \Xv_j) + K_\s(\xv, \xv) - 2 \fnms(\xv)}.\nn
}
The RKDE-weighted persistence diagram $\Dnms = \dgm(\bbv[\Xn \cup \Ym, \dnms])$ is then computed using the $\dnms$-weighted filtration on the composite sample. The bandwidth of the kernel and the parameters for the Hampel loss function are selected using the same approach as in \cite{vishwanath2020robust}. For each diagram, we compute the birth time $b(\qty{\xvo})$ for the first outlier $\xvo \in \Ym$, and the bottleneck influence $\winf(D_{n+m}, D_n)$, as described in Section~\ref{sec:influence}. We generate $10$ such samples for each value of $m$, and report the average in Figure~\ref{fig:influence}. 

{
From Figure~\ref{fig:influence}\,(a), we note that $D^{MoM}_{n+m,Q}$ and $D^{DTM}_{n+m,k}$ show similar behavior, although the outliers consistently appear earlier in the {DTM persistence diagram $D^{DTM}_{n+m,k}$.} Furthermore, for $D^{MoM}_{n+m,Q}$, we observe the sharp transition that occurs between $m=50$ and $m=80$, which is due to the fact that the theoretical guarantees for $\dnq$ from Theorem~\ref{theorem:momdist-influence} are valid only when $2m < Q = 100$. Similarly, from Theorem~\ref{theorem:momdist-consistency}, the outliers are guaranteed to have little influence on $D^{MoM}_{n+m,Q}$ whenever $m \le 50$, as seen in Figure~\ref{fig:influence}\,(a). In addition to the influence of the outliers in the bottleneck metric, we also compute the maximum persistence for the order-$1$ persistence diagram in Figure~\ref{fig:influence}\,(b). While the RKDE remains resilient to uniform outliers, we note that $\Dnms$ is significantly impacted by the outliers placed at a single point in the center of the circle. This is evidenced by the sharp transitions for $\Dnms$ in Figures~\ref{fig:influence}\,(a, b). However, unlike $\dsf_{n+m, Q}$ and $\delta_{n+m, k}$, by construction ${\norminf{\dnms} \le \sup_{\xv} \sqrt{2 K_\s(\xv, \xv)} < \infty}$. Therefore, the impact the outliers have on $\Dnms$ is bounded; and despite being more sensitive to the high-density outliers, the resulting influence on $\Dnms$ in Figures~\ref{fig:influence}\,(a, b, c) is bounded. 
}


\endgroup

\end{document}